\documentclass{amsart}
\usepackage{xy, amssymb, verbatim}
\usepackage{mathtools}
\xyoption{all}
\title[Cohomology of the Morava stabilizer group at large primes]{The mod $p$ cohomology of the Morava stabilizer group at large primes}
\date{October 2024}
\author{Mohammad Behzad Kang}
\author{Andrew Salch}
\theoremstyle{plain}
\newtheorem{prop}{Proposition}[subsection]
\newtheorem{theorem}[prop]{Theorem}

\newtheorem{corollary}[prop]{Corollary}
\newtheorem{definition}[prop]{Definition}
\newtheorem{definition-proposition}[prop]{Definition-Proposition}
\newtheorem{lemma}[prop]{Lemma}
\newtheorem{question}[prop]{Question}
\newcounter{lettered}
\newtheorem{letteredtheorem}[lettered]{Theorem}

\theoremstyle{definition}
\newtheorem{conventions}[prop]{Conventions}

\newtheorem{remark}[prop]{Remark}
\newtheorem{example}[prop]{Example}
\newtheorem{observation}[prop]{Observation}

\DeclareMathOperator{\pr}{{\rm pr}}
\DeclareMathOperator{\id}{{\rm id}}
\DeclareMathOperator{\Spec}{{\rm Spec}}
\DeclareMathOperator{\strict}{{\rm str}}
\DeclareMathOperator{\Proj}{{\rm Proj}}
\DeclareMathOperator{\unr}{}
\DeclareMathOperator{\Ext}{{\rm Ext}}
\DeclareMathOperator{\Mod}{{\rm Mod}}
\DeclareMathOperator{\Conn}{{\rm Conn}}
\DeclareMathOperator{\Irr}{{\rm Irr}}
\DeclareMathOperator{\Def}{{\rm Def}}
\DeclareMathOperator{\Spf}{{\rm Spf}}
\DeclareMathOperator{\holim}{{\rm holim}}
\DeclareMathOperator{\Cotor}{{\rm Cotor}}
\DeclareMathOperator{\Iso}{{\rm Iso}}
\DeclareMathOperator{\core}{{\rm core}}
\DeclareMathOperator{\medial}{{\rm medial}}
\DeclareMathOperator{\et}{{\rm \acute{e}t}}
\DeclareMathOperator{\tensor}{\otimes}

\DeclareMathOperator{\Aut}{{\rm Aut}}
\DeclareMathOperator{\Gal}{{\rm Gal}}
\DeclareMathOperator{\ad}{{\rm ad}}

\DeclareMathOperator{\im}{{\rm im}}
\DeclareMathOperator{\coim}{{\rm coim}}
\DeclareMathOperator{\Tor}{{\rm Tor}}
\DeclareMathOperator{\extendedGn}{{\bf G}_n}

\newcommand{\C}{\mathbb{C}}

\newcommand{\floor}[1]{\lfloor #1 \rfloor}

\newcommand{\Z}{\mathbb{Z}}
\newcommand{\R}{\mathbb{R}}

  \newcommand{\interior}[1]{%
  {\kern0pt#1}^{\mathrm{o}}%
}
\usepackage{hyperref}
\usepackage{cleveref}
\begin{document}
\begin{abstract}
We calculate the cohomology of the extended Morava stabilizer group of height $n$, with trivial mod $p$ coefficients, for all heights $n$ and all primes $p>>n$. The result is an exterior algebra on $n$ generators. A brief sketch of the method: we introduce a family of deformations of Ravenel's Lie algebra model $L(n,n)$ for the Morava stabilizer group scheme. This yields a family of DGAs, parameterized over an affine line and smooth except at a single point. The singular fiber is the Chevalley-Eilenberg DGA of Ravenel's Lie algebra. Consequently the cohomology of the singular fiber is the cohomology of the Morava stabilizer group, at large primes. We prove a derived version of the invariant cycles theorem from Hodge theory, which allows us to compare the cohomology of the singular fiber to the fixed-points of the Picard-Lefschetz (monodromy) operator on the cohomology of a smooth fiber. Finally, we use some new methods for constructing small models for cohomology of reductive Lie algebras to show that the cohomology of the Picard-Lefschetz fixed-points on a smooth fiber agrees with the singular cohomology $H^*(U(n);\mathbb{F}_p)$ of the unitary group, which is the desired exterior algebra.
\end{abstract}
\maketitle
\tableofcontents

\section{Introduction}
\subsection{What this paper is about}
\label{What this paper is about}

In this paper we calculate the cohomology $H^*(\Aut(\mathbb{G}_{1/n});(\mathbb{F}_{p})_{triv})$ of the full Morava stabilizer group scheme $\Aut(\mathbb{G}_{1/n})$ of height $n$, with trivial mod $p$ coefficients, for all primes $p>>n$. Equivalently, we calculate the cohomology $H^*(\extendedGn;\mathbb{F}_{p^n})$ of the extended Morava stabilizer group $\extendedGn = \Aut(\mathbb{G}_{1/n}\otimes_{\mathbb{F}_p}\mathbb{F}_{p^n})\rtimes \Gal(\mathbb{F}_{p^n}/\mathbb{F}_p)$ of height $n$ for all $p>>n$, with $\Gal(\mathbb{F}_{p^n}/\mathbb{F}_p)$ acting on $\mathbb{F}_{p^n}$ via its natural action and with $\Aut(\mathbb{G}_{1/n})$ acting trivially on the Galois fixed-points $\mathbb{F}_p\subseteq \mathbb{F}_{p^n}$. 
In particular, we verify an old folklore conjecture that $H^*(\extendedGn;\mathbb{F}_{p^n})$ is isomorphic to the exterior $\mathbb{F}_p$-algebra $\Lambda(x_1, \dots ,x_{2n-1})$ on a single odd generator in each odd degree from $1$ to $2n-1$.
Here is the main theorem:
\begin{letteredtheorem}\label{letteredthm 1}
Let $n$ be a positive integer. Then, for all sufficiently large primes $p$, $H^*\left(\extendedGn; \mathbb{F}_{p^n}\right)$ is isomorphic, as a graded $\mathbb{F}_p$-algebra, to the associated graded of a finite filtration on the singular cohomology $H^*(U(n);\mathbb{F}_p)$ of the unitary group. In particular, $H^*\left(\extendedGn; \mathbb{F}_{p^n}\right)$ is isomorphic to $H^*(U(n);\mathbb{F}_p)$ as graded $\mathbb{F}_p$-vector spaces.
\end{letteredtheorem}
This calculation with {\em mod $p$} coefficients seems to be complementary to the calculation of the cohomology of $\extendedGn$ with {\em rational} coefficients in the preprint \cite{barthel2024rationalizationknlocalsphere} of Barthel--Schlank--Stapleton--Weinstein: as there is no known {\em single} integral lift defined simultaneously for all primes $p$, there is no way to deduce a mod $p$ result for all sufficiently large primes from a rational result, and conversely. So there seems to be no way to deduce our result from theirs, or theirs from ours. Our methods are also entirely different from those of \cite{barthel2024rationalizationknlocalsphere}.

For various rephrasings of Theorem \ref{letteredthm 1}, we refer readers to \cref{Review of Morava stabilizer groups}, which has a brief summary of basic properties of the different versions of the Morava stabilizer group, including the applications of their cohomology groups in spectral sequence calculations in stable homotopy theory. The known methods for computing the $K(n)$-local and $E(n)$-local stable homotopy groups of finite spectra---mostly notably, the $K(n)$-local and $E(n)$-local stable homotopy groups of spheres---all begin with knowledge of the cohomology ring \begin{align}\label{coh ring 1} H^*\left(\extendedGn; E(\mathbb{G}_{1/n}\otimes_{\mathbb{F}_p}\mathbb{F}_{p^n})_*/(p,u_1, \dots ,u_{n-1})\right) &\cong H^*\left(\extendedGn; \mathbb{F}_{p^n}[u^{\pm 1}]\right).\end{align}
Consequently it is of fundamental importance to calculate the ring \eqref{coh ring 1}. At primes $p>>n$, where the structure of the ring \eqref{coh ring 1} is independent of the prime $p$, the ring \eqref{coh ring 1} has been described only for $n\leq 4$: the cases $n=1,2$ are in \cite{MR0458423}, the case $n=3$ is handled in \cite{ravenel1977cohomology},\cite{MR1173994},\cite{MR4301320}, and the case $n=4$ is described additively (and multiplicatively, up to possible multiplicative extensions in a certain spectral sequence) in the preprint \cite{height4}. 

For all $n$, the ring \eqref{coh ring 1} is periodic: multiplying by $u^{p^n-1}$ is an isomorphism. Hence the ring \eqref{coh ring 1} consists of a certain smaller graded ring, $H^*\left(\extendedGn; \mathbb{F}_{p^n}[u^{\pm 1}]\right)/(1-u^{p^n-1})$, which is then repeatedly periodically. At primes $p>>n$, the $\mathbb{F}_p$-linear dimension of that smaller graded ring is as follows:
\begin{equation}
\begin{array}{llll}
\mbox{n}          & \dim_{\mathbb{F}_p}H^*\left(\extendedGn; \mathbb{F}_{p^n}[u^{\pm 1}]\right)/(1-u^{p^n-1})\\
1 & 2 \\
2 & 12 \\
3 & 152 \\
4 & 3440.
\end{array}\end{equation}
Clearly $H^*\left(\extendedGn; \mathbb{F}_{p^n}[u^{\pm 1}]\right)/(1-u^{p^n-1})$ is bigger than an exterior algebra on $n$ generators. What we calculate in this paper, and prove to be isomorphic to such an exterior algebra, is the subring 
\begin{align*}
 H^*\left(\extendedGn; \mathbb{F}_{p^n}[u^{\pm (p^n-1)}]\right)/(1-u^{p^n-1}) 
  &\cong H^*\left(\extendedGn; \mathbb{F}_{p^n}\{u^{0}\}\right) \\
  &\cong H^*\left(\extendedGn; \mathbb{F}_{p^n}\right) \\
\end{align*} 
of $H^*\left(\extendedGn; \mathbb{F}_{p^n}[u^{\pm 1}]\right)/(1-u^{p^n-1})$. 

A topological way to think about this subring is as follows: for all primes $p$ and positive integers $n$ such that the $p$-primary Smith-Toda complex $V(n-1)$ exists, the $E(n)$-Adams spectral sequence for $V(n-1)$ has signature
\begin{align}
\label{adams ss 120} E_2^{s,t} \cong H^s\left(\extendedGn; \mathbb{F}_{p^n}\{ u^{t/2}\} \right) & \Rightarrow \pi_{t-s}(L_{E(n)}V(n-1)),\end{align} 
with $E_2^{s,t}$ trivial for $t$ odd. The notation $\mathbb{F}_{p^n}\{ u^{t/2}\}$ denotes the $\mathbb{F}_{p^n}$-linear span of $u^{t/2}$ inside $\mathbb{F}_{p^n}[u^{\pm 1}]$. 
Plotting \eqref{adams ss 120} with the usual (Adams) convention, the subring $H^*\left(\extendedGn; \mathbb{F}_{p^n}\right)$ calculated in Theorem \ref{letteredthm 1} is precisely the diagonal line of slope $-1$ through bidegree $(0,0)$ in this spectral sequence. This same line is repeated periodically every $2(p^n-1)$ stems as one looks to the left or the right in the $E_2$-page.

We now sketch how Theorem \ref{letteredthm 1} is proven.

\subsection{The reduction to Lie algebra cohomology} 
\label{The reduction to Lie algebra cohomology}

The first step is to compare the cohomology of the Morava stabilizer group to the cohomology of a Lie algebra. This relies heavily on ideas of Ravenel \cite{MR0420619}, \cite{ravenel1977cohomology}, \cite[chapter 6]{MR860042}: for each prime $p$ and each integer $n$, Ravenel constructed an $n^2$-dimensional solvable Lie graded $\mathbb{F}_p$-algebra $L(n,n)$ with the property that, for $p>n+1$, there exists a spectral sequence
\begin{align}
\label{rmss} E_1^{*,*} \cong H^{*,*}(L(n,n);\mathbb{F}_p) &\Rightarrow H^*\left(\extendedGn; \mathbb{F}_{p^n}[u^{\pm 1}]\right)/(1-u^{p^n-1}). \end{align}
We refer to $L(n,n)$ as {\em Ravenel's Lie algebra model for the Morava stabilizer group.} 
In the preprint \cite{salch2023ravenels}, we called \eqref{rmss} the {\em Ravenel-May spectral sequence}. The main result of that preprint is that the Ravenel-May spectral sequence collapses immediately at all primes $p>>n$, and that there are furthermore no multiplicative filtration jumps in the abutment. 
Consequently, for $p>>n$, the Lie algebra cohomology $H^{*,*}(L(n,n);\mathbb{F}_p)$ is isomorphic, as a graded $\mathbb{F}_p$-algebra, to the cohomology $H^*\left(\extendedGn; \mathbb{F}_{p^n}[u^{\pm 1}]\right)/(1-u^{p^n-1})$. That is, at sufficiently large primes, Ravenel's Lie algebra model has the same cohomology, with mod $p$ coefficients, as the Morava stabilizer group. {\em To be clear, Theorem \ref{letteredthm 1} relies on the main result of the preprint \cite{salch2023ravenels}.} 

As a consequence, we may calculate the subring $H^*\left(\extendedGn; (\mathbb{F}_{p^n})_{triv}\right)$ by calculating the part of the Lie algebra cohomology $H^{*,*}(L(n,n);\mathbb{F}_p)$ in internal degrees\footnote{The cohomology of $L(n,n)$ is bigraded: the first grading is the cohomological grading, i.e., the grading in which the degree $i$ summand is the $i$th cohomology group. The other grading is the {\em internal grading}, induced by the grading on the Lie algebra $L(n,n)$ itself.} divisible by $2(p^n-1)$. The cohomology of the Lie algebra $L(n,n)$ is the cohomology of its Chevalley-Eilenberg \cite{MR0024908} differential graded algebra $CE^{\bullet}(L(n,n))$. 
Consequently our task is to show that the sub-DGA of $CE^{\bullet}(L(n,n))$ consisting of elements in internal degrees divisible by $2(p^n-1)$ has cohomology $\Lambda(x_1, x_3, \dots ,x_{2n-1})$. We call this sub-DGA of $CE^{\bullet}(L(n,n))$ the {\em critical complex of height $n$}, and we write $cc^{\bullet}(L(n,n))$ for it.

\subsection{Deformations of Ravenel's Lie algebra model for the Morava stabilizer group} 

One of the new constructions in this paper is a one-parameter family of deformations of the Lie algebra $L(n,n)$. In \cref{Deformations of Ravenel's model.}, we construct a family of Lie $k$-algebras, parameterized over the affine line $\mathbb{A}^1_{k}$, whose fiber at $0$ is Ravenel's {\em solvable} Lie model $L(n,n)$, and whose fiber at every nonzero point in $\mathbb{A}^1_k$ is instead a {\em reductive} Lie $k$-algebra. We take the fiberwise Chevalley-Eilenberg differential graded algebra of these Lie algebras, yielding a bundle $\Lambda^{\bullet}_n$ of differential graded $k$-algebras over $\mathbb{A}^1_{k} \backslash \{ 0\}$. 

Then we bring in the machinery of connections and monodromy: in Theorem \ref{main thm 1} we show that there is an essentially unique differential-graded multiplicative connection on this bundle of DGAs which is equivariant with respect to a certain natural $C_n$-action. By parallel transport in this connection around a loop circling the puncture at $0$ in $\mathbb{A}^1_{k} \backslash \{ 0\}$, we get a Picard-Lefschetz (monodromy) operator $T: H^*((\Lambda^{\bullet}_n)_{sm}) \rightarrow H^*((\Lambda^{\bullet}_n)_{sm})$ on the cohomology of any of the smooth fibers $(\Lambda^{\bullet}_n)_{sm}$ of $\Lambda^{\bullet}_n$. We wish to compare the monodromy fixed-points $H^*\left((\Lambda^{\bullet}_n)_{sm}\right)^T$ with the cohomology of the singular fiber $H^*\left((\Lambda^{\bullet}_n)_{0}\right)$. 

The situation resembles that of the local invariant cycles theorem, originally conjectured by Griffiths in \cite{MR0258824}, and proven by Clemens in \cite{MR0444662}; see \cite{MR0756848} for a very good expository account. The local invariant cycles theorem asserts the following. We write $D$ for the unit disk in the complex plane. Suppose we are given
\begin{itemize}
\item a proper flat holomorphic map $\mathcal{X}\rightarrow D$,
\item such that $\mathcal{X}$ is a K\"{a}hler manifold,
\item such that the fiber $\mathcal{X}_t$ over each nonzero $t\in D$ is a smooth complex variety,
\item and such that the fiber $\mathcal{X}_0$ over $0\in D$ is a divisor with normal crossings.
\end{itemize}
Then $H^*(\mathcal{X}_0;\mathbb{Q})\cong H^*(\mathcal{X};\mathbb{Q})$. 
Write $\mathcal{X}_{sm}$ for any smooth fiber, i.e., $\mathcal{X}_{sm} = \mathcal{X}_t$ for some fixed nonzero $t\in D$. 
Let $D^*$ denote the punctured disk $D \backslash \{0\}$.
Corresponding to a choice of generator of $\pi_1(D^*)$, we get a Picard-Lefschetz (monodromy) operator $T: H^*(\mathcal{X}_{sm};\mathbb{Q}) \rightarrow  H^*(\mathcal{X}_{sm};\mathbb{Q})$. 
The local invariant cycles theorem states that the composite map 
 \[ H^*(\mathcal{X}_0 ;\mathbb{Q})\stackrel{\cong}{\longrightarrow}H^*(\mathcal{X};\mathbb{Q}) \rightarrow H^*(\mathcal{X}_{sm};\mathbb{Q}) \]
lands in the fixed-points of $T$, and furthermore, the resulting map $H^*(\mathcal{X}_{0};\mathbb{Q}) \rightarrow  H^*(\mathcal{X}_{sm};\mathbb{Q})^T$ is surjective. A $p$-adic version of the theorem was proven by Deligne in \cite{MR0601520}.

In Theorem \ref{main local inv cycles thm} we prove a certain derived analogue of the local invariant cycles theorem. See that theorem, and Corollary \ref{local inv cycles cor}, for the statements, which are technical. The upshot\footnote{Compare this use of deformations to the standard (since \cite{MR4281382} and \cite{MR4574661}) uses of deformations in stable homotopy theory: it is commonplace to talk about a ``deformation of homotopy theories'' as being a family of homotopy theories containing some parameter $\tau$. One refers to the whole family as a deformation of the homotopy theory obtained by setting $\tau$ to zero. The standard example, due to \cite{MR4281382}, is in $\mathbb{C}$-motivic stable homotopy theory: the category of $p$-complete cellular $\mathbb{C}$-motivic spectra has a parameter $\tau$ such that setting $\tau$ to zero recovers Hovey's \cite{MR2066503} stable derived category of $BP_*BP$-comodules, while inverting $\tau$ recovers the classical category of $p$-complete spectra. The former is a special fibre, the latter is a generic fibre, and the two fit together over $\Spec$ of a discrete valuation ring, i.e., the Sierpinski topological space, with two points, one of which is generic.

In our setting, we have a much richer underlying space of our family of deformations. Instead of having only a generic point and a special point, we have a whole affine line, our parameterized family is smooth except at $0$, and the geometry of the resulting punctured affine line---in particular, the fact that we can wind around the singular fiber, yielding a nontrivial monodromy action of $\pi_1(\mathbb{A}^1 - \{ 0\})$---plays a critical role in our arguments. This is much closer to the way that deformations are used in geometry, e.g. in the study of variation of Hodge structure as in \cite{MR0258824}, than the way that deformations have recently been used and written about in stable homotopy theory. 

We remark that our parameterized family of DGAs indeed lifts to a parameterized family of homotopy theories: an alternative construction of our family of DGAs is by means of a certain family of Dieudonn\'{e} modules parameterized over $\mathbb{A}^1$, smooth except at $0$, whose singular fiber is formal and whose smooth fibers are \'{e}tale. One can take the Morava/Lubin-Tate $E$-theory of the whole family to get a meaningful parameterized family of homotopy theories, a family of deformations of the usual Morava $E_n$. The second author plans to return to this construction in later work, and is grateful to Anish Chedalavada for an insightful discussion about it.} is that, for our bundle of DGAs $\Lambda^{\bullet}_n$, 
\begin{itemize}
\item we get a surjective map $H^*\left((\Lambda^{\bullet}_n)_{0}\right) \twoheadrightarrow H^*\left((\Lambda^{\bullet}_n)_{sm}\right)^T$ just as in the classical local invariant cycles theorem,
\item but more importantly, we also get an isomorphism $H^*\left((\Lambda^{\bullet}_n)_{sm}\right)^T \cong H^*\left(((\Lambda^{\bullet}_n)_{0})^T\right)$. 
\end{itemize}
It is not obvious that $H^*\left(((\Lambda^{\bullet}_n)_{0}\right)^T$ ought to make sense at all, i.e., that there ought to be any meaningful version of the Picard-Lefschetz operator on the {\em singular} fiber. The origin of a natural such operator is explained in \cref{The monodromy fixed-points...}. It has the property that its fixed-point DGA $((\Lambda^{\bullet}_n)_{0})^T$ is a certain well-behaved sub-DGA of the Chevalley-Eilenberg complex $(\Lambda^{\bullet}_n)_0 = CE^{\bullet}(L(n,n))$ of Ravenel's Lie algebra model for the Morava stabilizer group. We call this sub-DGA the {\em first-subscript complex} of $L(n,n)$. 

In Theorem \ref{ss collapse thm 304} we show that, if $p>2n^2$, then the first-subscript complex has cohomology isomorphic to the cohomology of the critical complex. This uses the derived local invariant cycles theorem together with an additional spectral sequence argument. 

It takes several additional steps, sketched below in \cref{The main theorem}, to show that the cohomology of the critical complex is isomorphic to $\Lambda(x_1, x_3, \dots ,x_{2n-1})$. But as a noteworthy consequence of our approach, we get the cohomology of a {\em very large} sub-DGA of $CE^{\bullet}(L(n,n))$: the first-subscript complex is of $\mathbb{F}_p$-linear dimension approximately $1/n$ times the $\mathbb{F}_p$-linear dimension of $CE^{\bullet}(L(n,n))$ itself. In that sense, we have in fact calculated about ``$1/n$ of $H^*(\extendedGn;\mathbb{F}_{p^n}[u^{\pm 1}])$'' at large primes, i.e., it is not merely that we have calculated the cohomology in internal degrees congruent to $0$ modulo $2(p^n-1)$, but we have also shown that an enormous subcomplex of $CE^{\bullet}(L(n,n))$ in degrees {\em not} congruent to $0$ modulo $2(p^n-1)$ contributes {\em nothing} to the cohomology of the Morava stabilizer group.

For example, at height $n=5$, the critical complex in $CE^{\bullet}(L(n,n))$ is a modest $247,552$-dimensional vector space, comprising only about 0.7 percent of the full Chevalley-Eilenberg complex of $L(5,5)$, whose vector space dimension is $2^{(5^2)} = 33,554,432$. The height $5$ first-subscript complex is instead $6,710,912$-dimensional, which is just over one-fifth the size of the full Chevalley-Eilenberg complex of $L(5,5)$. We not only show that the cohomology of the critical complex is the exterior algebra $\Lambda(x_1,x_3,x_5,x_7,x_9)$; we also show that the cohomology of the much larger {\em first-subscript complex} is also $\Lambda(x_1,x_3,x_5,x_7,x_9)$. Further explanation and discussion of this point can be found in \cref{Cohomology of the singular...}.

\subsection{Outline of proof of the main theorem}
\label{The main theorem}

\begin{itemize}
\item Using results on Lie theory in positive characteristic, we know that the cohomology $H^*(U(n);k)$ of the compact Lie group $U(n)$ is isomorphic to the cohomology $H^*(\mathfrak{gl}_n(k);k)$ of the Lie algebra $\mathfrak{gl}_n(k)$, for any finite field $k$ of characteristic greater than $3n-3$. This result is not new (if $k$ were instead a field of characteristic zero, the result is extremely well-known), but we spell out a bit of detail in Proposition \ref{gln cohomology}.
\item Consequently $H^{*}(U(n);k)$ is isomorphic to the cohomology of the Chevalley-Eilenberg complex $\Lambda^{\bullet}_k(\mathfrak{gl}_n(k)^*)$ for any finite field $k$ of characteristic greater than $3n-3$. In \cref{Models for Lie algebra cohomology} we develop some theory for building {\em models} of Lie algebra cohomology. That is, given a Lie $k$-algebra $L$, one always has the Chevalley-Eilenberg differential graded $k$-algebra $CE^{\bullet}(L)$, but there may be also be a sub-DGA $A^{\bullet}$ of $CE^{\bullet}(L)$ which is quasi-isomorphic to $CE^{\bullet}(L)$. We then call $A^{\bullet}$ a {\em model} for the cohomology of $L$. 

Our main application of this theory, in Theorem \ref{main chain htpy application thm}, establishes that the Chevalley-Eilenberg complex $\Lambda^{\bullet}_{\mathbb{F}_p(\omega)}(\mathfrak{gl}_n(\mathbb{F}_p(\omega))^*)$ is quasi-isomorphic to the critical complex $cc^{\bullet}(\mathfrak{gl}_n(\mathbb{F}_p(\omega)))$, using a grading on $\mathfrak{gl}_n(\mathbb{F}_p(\omega))$ defined in Definition \ref{gradings def}. That is, the critical complex $cc^{\bullet}(\mathfrak{gl}_n(\mathbb{F}_p(\omega)))$ is a {\em model} for the cohomology of $\mathfrak{gl}_n(\mathbb{F}_p(\omega))$. Here $\omega$ is a primitive $n$th root of unity.
\item In Theorem \ref{gln and cc cohomology}, we give a descent argument to show that the isomorphism of rings $H^*(cc^{\bullet}(\mathfrak{gl}_n(\mathbb{F}_p(\omega)))) \cong H^{*}(U(n);\mathbb{F}_p(\omega))$ in fact comes from an isomorphism of rings $H^*(cc^{\bullet}(\mathfrak{gl}_n(\mathbb{F}_p))) \cong H^{*}(U(n);\mathbb{F}_p)$, as long as $p$ does not divide $n$.
\item We then use our deformed Ravenel model and the derived local invariant cycles theorem, to prove that, for all primes $p>2n^2$, the cohomology of the critical complex $cc^{\bullet}(\mathfrak{gl}_n(\mathbb{F}_p))$ is isomorphic to the cohomology of the critical complex $cc^{\bullet}(L(n,n))$. This requires the development of a substantial body of theory on deformations, connections, monodromy, and invariant cycles for DGAs, as described above in \cref{The reduction to Lie algebra cohomology}. This material occupies \cref{Deformations of Ravenel's model.}. The main result, the isomorphism $H^*(cc^{\bullet}(\mathfrak{gl}_n(\mathbb{F}_p)))\cong H^*(cc^{\bullet}(L(n,n)))$, is Corollary \ref{main cor 1}.
\item At this point, we have isomorphisms
\begin{align}
\label{iso aa 1}  H^*(U(n);\mathbb{F}_p) 
   &\cong H^*(\mathfrak{gl}_n(\mathbb{F}_p);\mathbb{F}_p) \\
\label{iso aa 2}   &\cong H^*(cc^{\bullet}(\mathfrak{gl}_n(\mathbb{F}_p))) \\
\label{iso aa 3}   &\cong H^*(cc^{\bullet}(L(n,n)) \\
\label{iso aa 4}   &\cong H^{*,2(p^n-1)\cdot *}(L(n,n);\mathbb{F}_p)
\end{align}
for all primes $p>2n^2$. For readers who are uncomfortable with results that depend on a preprint, we point out that the isomorphisms \eqref{iso aa 1} through \eqref{iso aa 4} do {\em not} rely on the results of \cite{salch2023ravenels}.
\end{itemize}
Now we invoke the main result of the preprint \cite{salch2023ravenels}, which shows that \eqref{iso aa 4} is isomorphic to $H^*(\extendedGn;\mathbb{F}_{p^n})$ for $p>>n$, due to collapse of the Ravenel-May spectral sequence \eqref{rmss}, but does not yield any specific bound on the primes $p$ for which this isomorphism exists. Our use of this result from \cite{salch2023ravenels} is the only reason that Theorem \ref{letteredthm 1} is stated only to hold for $p>>n$, rather than for all $p>2n^2$. If the Ravenel-May spectral sequence can be shown to collapse for any particular prime $p$ greater than $2n^2$, then due to isomorphisms \eqref{iso aa 1} through \eqref{iso aa 4}, we will have a proof of Theorem \ref{letteredthm 1} for {\em that particular prime} $p$.

\subsection{Acknowledgments and conventions}

In 2018, Agnes Beaudry and Doug Ravenel talked with us about the question of the cohomology of the extended Morava stabilizer group with trivial coefficients. Beaudry and Ravenel had put some thought into this problem, but they graciously left off with their investigations so as to avoid us stepping on each other's toes. We are grateful to them for this, as well as for their discussions with us and their interest and support.

Pavel Etingof, Julia Pevtsova, and Dan Nakano generously helped us in a search for references which answer the following question: given a classical semisimple Lie algebra $\mathfrak{g}$ over a field of characteristic $p$, how large much $p$ be in order that the cohomology of $\mathfrak{g}$ is exterior on generators whose degrees are given in terms of the exponents of $\mathfrak{g}$, as it is in characteristic zero? We are grateful to them for their help. Dan Nakano showed us the paper \cite{MR0821318} of Friedlander--Parshall, which is strong enough for our purposes in this paper. The relevant result from Friedlander--Parshall is recalled below, as Theorem \ref{sln cohomology}. 

We thank Dan Frohardt for coming up with an ingenious combinatorial argument to prove an asymptotic dimension estimate which we had conjectured; see the footnote at the end of \cref{Cohomology of the singular...}.

By 2020, the first half of this paper was finished, but we did not yet know how to prove that the critical complex in the singular fiber of our parameterized family of DGAs had the same cohomology as the critical complex in a smooth fiber. Around that time, Wayne Raskind was the dean of the College of Liberal Arts and Sciences at our institution, Wayne State University. In conversation, Wayne remarked that this situation is quite similar to that of the local invariant cycles theorem. We then had to go learn about the local invariant cycles theorem, but after a few months of work, we had the proof we needed. An excellent conversation to have with one's dean! We thank Wayne for this important suggestion which turned out to be of critical importance for this project.

\begin{conventions}\label{conventions}\leavevmode
\begin{itemize}
\item The symbol $\floor{x}$ denotes the integer floor of $x$.
\item
We will often use symbols involving subscripts which are defined modulo $n$. In order to simplify various formulas, it will be convenient to use the representatives $1, \dots ,n$ for the residue classes modulo $n$, rather than the representatives $0, \dots ,n-1$. We will write $Z_n$ for the set $\{ 1 , \dots ,n\}$. 
\item Given a commutative ring $R$, we will write $R^u$ for the subset $\{x: x^n = 0\mbox{\ for\ some\ } n\} \cup \{ x: x^n = 1\mbox{\ for\ some\ } n\}$ of $R$, i.e., $R^u$ is the union of the nilradical of $R$ and the set of roots of unity in $R$.
\item When $A$ is a graded $R$-algebra and $M$ a graded $A$-module, we will sometimes need to refer to:
\begin{enumerate}
\item graded $R$-linear derivations $d: A\rightarrow M$, that is, $R$-linear maps such that $d(ab) = (da)b + (-1)^{\left|a\right|}a(db)$ for all homogeneous $a,b\in A$,
\item ungraded $R$-linear derivations $d: A\rightarrow M$, that is, $R$-linear maps such that $d(ab) = (da)b + a(db)$ for all $a,b\in A$,
\item and ungraded $R$-linear derivations $d: A\rightarrow M$ which preserve degree, that is, grading-preserving $R$-linear maps such that $d(ab) = (da)b + a(db)$ for all $a,b\in A$.
\end{enumerate}
 We will always refer to the first as {\em graded $R$-linear derivations}, to the second as simply {\em $R$-linear derivations,} and to the third as {\em grading-preserving $R$-linear derivations.}
\item Given a commutative ring $R$ and an $R$-module $M$, we write $\Lambda_R(M)$ for the exterior/Grassmann $R$-algebra of $M$. When $M$ is the $R$-linear dual of a Lie algebra, then the exterior algebra of $M$ inherits a differential from the Lie bracket. We write $\Lambda^{\bullet}_R(M)$ for the resulting differential graded $R$-algebra, i.e., the Chevalley-Eilenberg complex \cite{MR0024908} of the Lie algebra.
\item All differential graded algebras in this paper will be graded so that the differential increases degree by $1$.
\item All graded Lie algebras in this paper have bracket of degree zero, and they furthermore satisfy the alternating and Jacobi identities, {\em not} their Koszul versions which differ from the classical identities by a sign. Explicitly, given an abelian group $A$, when we say that $\mathfrak{g}$ is an $A$-graded Lie $k$-algebra, we will mean that $\mathfrak{g}$ is a Lie $k$-algebra which is equipped with a $k$-module splitting $\mathfrak{g} = \bigoplus_{a\in A}\mathfrak{g}_a$ such that the Lie bracket of an element in degree $i$ and an element in degree $j$ lies in degree $i+j$.
\item In this paper, all formal group laws are implicitly understood to be one-dimensional.
\item Given a prime $p$ and a positive integer $n$, we write $\mathbb{G}_{1/n}$ for the $p$-typical formal group law over $\mathbb{F}_p$ classified by the ring map $BP_*\rightarrow \mathbb{F}_p$ sending the Hazewinkel generator $v_n\in BP_*$ to $1\in\mathbb{F}_p$ and sending the other Hazewinkel generators $v_i\in BP_*, i\neq n$, to zero.
\item In \cref{Monodromy section}, we will often refer to $\mathbb{A}^1_R$ and $\mathbb{A}^1_R - \{0\}$, where $R$ is a commutative ring. Occasionally we will mean $\mathbb{A}^1_R$ or $\mathbb{A}^1_R - \{0\}$ regarded as a scheme; usually we will instead mean the set of closed points in $\mathbb{A}^1_R$ or $\mathbb{A}^1_R - \{0\}$. We will rely on context to make it clear which is meant. 
\end{itemize}
\end{conventions}

\section{Models for Lie algebra cohomology}
\label{Models for Lie algebra cohomology}

\subsection{Derivations and chain homotopy retractions on Chevalley-Eilenberg DGAs}

In this section, we develop some tools for constructing cochain subcomplexes $C^{\bullet}$ of a differential graded algebra $A^{\bullet}$ such that the inclusion map $C^{\bullet}\hookrightarrow A^{\bullet}$ is a quasi-isomorphism. 
Our main applications are to the case where $A^{\bullet}$ is the Chevalley-Eilenberg complex $\Lambda^{\bullet}(L^*)$ of a Lie algebra $L$. The resulting cochain subcomplexes $C^{\bullet}$ of $\Lambda^{\bullet}(L^*)$ then have the property that $H^*(C^{\bullet})$ is isomorphic to the cohomology $H^*(L;k)$ of the Lie algebra $L$. In that sense, $C^{\bullet}$ is a ``model'' for the cohomology of $L$. 

Our ``models'' for the cohomology of Lie algebras are not minimal models in the sense of Sullivan \cite{MR0646078}, but we suspect that our models could be understood as a case of some generalization or adaptation of Sullivan's theory. We have not pursued that idea, however.

\subsection{Some elementary linear algebra}
\label{elementary lin alg}

This subsection collects some elementary linear-algebraic definitions and lemmas that will be used in the proof of Theorem \ref{thm on diffs giving chain htpies}. We leave off most proofs in this subsection.

\begin{definition}\label{def of circledast}
Let $k$ be a field, let $V_1,V_2$ be finite-dimensional $k$-vector spaces, and let $D_1: V_1\rightarrow V_1$ and $D_2: V_2\rightarrow V_2$ be $k$-linear transformations. We will write $D_1\circledast D_2$ for the $k$-linear transformation $V_1\otimes_k V_2 \rightarrow V_1\otimes_k V_2$ given by the formula
\[ \left( D_1\circledast D_2\right)\left( v_1\otimes v_2\right) = (D_1(v_1))\otimes v_2 + v_1\otimes (D_2(v_2)).\]
\end{definition}
\begin{remark}\label{hopf alg remark}
By choosing ordered bases for $V_1$ and $V_2$, we can express $\circledast$ as an operation on square matrices. For example,
given ordered bases $v_1,w_1$ and $v_2,w_2$ for $V_1$ and $V_2$, respectively, we have
\[\left[\begin{array}{ll} a & b \\ c & d \end{array} \right] \circledast \left[\begin{array}{ll} e & f \\ g & h \end{array} \right]
  = \left[ \begin{array}{llll} a+e & f & b & 0 \\ g & a+h & 0 & b \\ c & 0 & d+e & f \\ 0 & c & g & d+h \end{array} \right],\]
if we choose the ordered basis $(v_1\otimes v_2,v_1\otimes w_2,w_1\otimes v_2,w_1\otimes w_2)$ for $V_1\otimes_k V_2$.
Of course $\circledast$ is not equal to, or (usually) conjugate to, the tensor/Kronecker product $\otimes$ of matrices. 

A more structural perspective: the data of a $k$-vector space equipped with a $k$-linear endomorphism is equivalent to the data of a $k[x]$-module. The tensor/Kronecker product $\otimes$ of representations corresponds to the symmetric monoidal product on the category of $k[x]$-modules arising from the coproduct $k[x]\rightarrow k[x]\otimes_k k[x]$ sending $x$ to $x\otimes x$. Similarly, the product $\circledast$ of representations corresponds to the symmetric monoidal product on the category of $k[x]$-modules arising from the coproduct $k[x]\rightarrow k[x]\otimes_k k[x]$ sending $x$ to $x\otimes 1 + 1\otimes x$. From this perspective, $\circledast$ is extremely natural. In a sense, it is even more natural than the tensor/Kronecker product, since $\circledast$ arises from a Hopf algebra structure on $k[x]$, while $\otimes$ arises from a non-Hopf bialgebra structure on $k[x]$. 

Surely $\circledast$ must have been studied before, but we do not know what name it may have already been given, or where it may have already been written about. 
\end{remark}

\begin{lemma}\label{eigenvalues of circledast}
Let $k$ be a field, let $V_1,V_2$ be finite-dimensional $k$-vector spaces. Let $D_1: V_1\rightarrow V_1$ and $D_2: V_2 \rightarrow V_2$ be $k$-linear transformations. Then the set of eigenvalues of $D_1\circledast D_2$ is the set of all sums of the form $\lambda_1 + \lambda_2$ in $k$, where $\lambda_1$ is an eigenvalue of $D_1$ and $\lambda_2$ is an eigenvalue of $D_2$.
\end{lemma}
\begin{proof}
Routine.
\end{proof}

\begin{definition}\label{def of u-property}
Let $k$ be a field, let $A$ be a graded $k$-algebra, 
and let $D: A \rightarrow A$ be a degree-preserving $k$-linear derivation. We will say that $D$ {\em has the $u$-property} if there exists a set $I$ of integers and, for each $i\in I$, an $k$-linear basis $B^i$ for $A^i$ such that:
\begin{itemize}
\item $A$ is generated, as a $k$-algebra, by $\bigcup_{i\in I} A^i$, i.e., by homogeneous elements whose degrees are in the set $I$,
\item for each $i\in I$, $B^i$ is a set of generalized eigenvectors\footnote{To avoid any possible misunderstanding: when we say that the $k$-linear basis $B^i$ for $A^i$ consists of generalized eigenvectors for the action of $D$, we mean that the action of $D$ on $A^i$ is a Jordan matrix when expressed in terms of some total ordering on the basis $B^i$. In other words, we mean that $B^i$ is a union of finite subsets, each of which admits a total ordering and an element $\lambda\in k$ such that $D(v) = \lambda v$ for the maximal element $v$ in the finite subset, and $D(v) = \lambda v + v^{\prime}$ for each nonmaximal $v$ in the finite subset, where $v^{\prime}$ is the successor element to $v$ in the total ordering.} for the action of $D$ on $A^i$,
\item for each $i\in I$, the eigenvalues of the action of $D$ on $B^i$ are in $k^u$ (as defined in Conventions \ref{conventions}),
\item and, for each ordered tuple $(b_1, \dots ,b_j)$ of elements of $\bigcup_{i\in I} B^i$, either the product $b_1b_2\dots b_j$ is zero, or the sum $\lambda_1 + \dots + \lambda_j$ of their eigenvalues is in $k^u$.
\end{itemize}

We will say that $D$ {\em has the potential $u$-property} if $D\otimes_k K: A\otimes_k K \rightarrow A\otimes_k K$ has the $u$-property for some algebraic field extension $K/k$.
\end{definition}

\begin{example}\label{u-prop example}
Our purpose in introducing the $u$-property is to use it, in Theorem \ref{thm on diffs giving chain htpies}, to construct chain-homotopy deformation retracts of the Chevalley-Eilenberg DGA of a Lie $k$-algebra $L$. Since the Chevalley-Eilenberg DGA of $L$ is the exterior algebra on the linear dual of $L$, our applications of Definition \ref{def of u-property} will always be in the case that the DGA $A^{\bullet}$ is an exterior algebra on some vector space $V$ concentrated in grading degree $1$. In this case, the set $I$ from Definition \ref{def of u-property} is simply $\{1\}$. 
Here are some examples of such exterior algebras, and the derivations with the $u$-property on those exterior algebras. 
\begin{itemize}
\item When $k = \mathbb{R}$, the only roots of unity in $k$ are square roots of unity, so if $D: \Lambda^{\bullet}(V) \rightarrow\Lambda^{\bullet}(V)$ has the $u$-property, then the action of $D$ on $\Lambda^{\bullet}(V)$ can have at most three eigenvalues, that is, $-1,0$, and $1$. Given two elements $x,y\in V$, we have $xy=0$ in $\Lambda^2(V)$ if and only if $x$ is a scalar times $y$. Hence, if $D$ has the $u$-property, then the action of $D$ on $V$ can decompose into Jordan blocks of at most three distinct eigenvalues: a one-dimensional eigenspace with eigenvalue $1$, a one-dimensional eigenspace with eigenvalue $-1$, and a generalized eigenspace (of any dimension) with eigenvalue $0$.
\item When $k = \mathbb{C}$, there are far more roots of unity in $k$, but there are very few subsets $S\subseteq \mathbb{C}^u$ with the property that the sum of every pair of distinct elements is either $0$ or a root of unity in $\mathbb{C}$. In particular, such subsets exist with four elements (e.g. $\{0,1,\zeta_3,\zeta_3^2\}$ with $\zeta_3$ a primitive cube root of unity), but no such subsets exist with more than four elements. 
Consequently, the action of $D$ on $A$ can have at most four eigenvalues, and the action of $D$ on $V$ can decompose into Jordan blocks of at most four distinct eigenvalues: up to three one-dimensional eigenspaces with nonzero eigenvalues, and a generalized eigenspace (of any dimension) with eigenvalue $0$.
\item The situation is dramatically different when $k$ is a finite field, since in a finite field, every nonzero element is a root of unity. Consequently, if $k$ is a finite field, $D$ is a degree-preserving $k$-linear derivation on a finite-dimensional graded $k$-algebra $A$, and $K$ is a field extension of $k$ which contains all the eigenvalues of the action of $D\otimes_k \overline{k}$ on $A\otimes_k \overline{k}$, then the base-changed derivation $D\otimes_k K: A\otimes_k K \rightarrow A\otimes_k K$ {\em always has the $u$-property}. 
\end{itemize}
\end{example}

\begin{lemma}\label{basis of eigenvectors}
If $k$ is a field, $A$ is a graded $k$-algebra such that $A^n$ is a finite-dimensional $k$-vector space for each integer $n$, and $D: A\rightarrow A$ is a degree-preserving $k$-linear derivation with the $u$-property, then for each integer $n$, $A^n$ admits a basis of generalized eigenvectors for the action of $D$, such that all eigenvalues are in $k^u$.

Furthermore, if the action of $D$ on $A^i$ is diagonalizable for all $i\in I$, then the action of $D$ on $A^{\bullet}$ is diagonalizable. Here $I$ is as defined in Definition \ref{def of u-property}.
\end{lemma}
\begin{proof}
This is routine, but we give the details, since the introduction of the product $\circledast$ in Definition \ref{def of circledast} permits a nice simplification of one part of the argument. 
Let $I$ and $B^i$ be as in Definition \ref{def of u-property}.
Since $D$ has the $u$-property, $A$ is generated as an $k$-algebra by $A^i$ for $i\in I$. 
In particular, the $k$-module homomorphism
\begin{align}\label{k-module mult map} \nabla: \bigoplus_{(i_1, \dots ,i_{\ell})\in I^{\ell} : \sum_{j=1}^{\ell}i_j = n} A^{i_1}\otimes_k A^{i_2}\otimes_k \dots \otimes_k A^{i_{\ell}} &\rightarrow A^n,\end{align}
given by multiplication in $A$, is surjective. 
We equip the domain of \eqref{k-module mult map} with the $k$-linear action of $D$ induced by iteration application of the Leibniz rule $D(v\otimes w) = D(v)\otimes w + v\otimes D(w)$ (i.e., $D$ acts by $\circledast$, as in Definition \ref{def of circledast}), so that $\nabla$ commutes with the action of $D$. By Lemma \ref{eigenvalues of circledast}, every eigenvalue of the action of $D$ on $A^{i_1}\otimes_k A^{i_2}\otimes_k \dots \otimes_k A^{i_{\ell}}$ is the sum of $\ell$ elements of $k^u$, and since $\nabla$ is an epimorphism in the category of finite-dimensional $k$-vector spaces equipped with a $k$-linear endomorphism, each eigenvalue of the action of $D$ on $A^n$ is a sum of $\ell$ elements of $k^u$.

However, not every sum of $\ell$ elements of $k^u$ is an element of $k^u$, so we are not done yet. 
Since $A$ has the $u$-property, the eigenvalues of the action of $D$ on each $A^i$ are already in $k$ (i.e., without need to extend the base field $k$), so each $A^i$ admits a Jordan decomposition over $k$. Given a Jordan block $J_{m_j}(\lambda_j)$ of $A_j$ for each $j=1, \dots ,\ell$, if $\sum_{j=1}^{\ell}\lambda_j$ is {\em not} in $k^u$, the $u$-property of $A$ gives us that $\nabla(v_1\otimes v_2\otimes\dots\otimes v_{\ell})=0$ whenever $v_j\in J_{m_j}(\lambda_j)$ for all $j\in \{ 1,2,\dots ,\ell\}$. So $\nabla$ vanishes on every element in a $k$-linear basis for $J_{m_1}(\lambda_1)\otimes_kJ_{m_2}(\lambda_2) \otimes_k \dots \otimes_k J_{m_{\ell}}(\lambda_{\ell})$. 

Consequently, in the decomposition of $\bigoplus_{(i_1, \dots ,i_{\ell})\in I^{\ell} : \sum_{j=1}^{\ell}i_j = n} A^{i_1}\otimes_k A^{i_2}\otimes_k \dots \otimes_k A^{i_{\ell}}$ into a direct sum of $\circledast$-products of Jordan blocks, a given summand $J_{m_1}(\lambda_1) \circledast\dots\circledast J_{m_{\ell}}(\lambda_{\ell})$ has nonzero image under $\nabla$ only if $\sum_{j=1}^{\ell}\lambda_j\in k^u$. Since $\nabla$ is surjective, each eigenvalue of the action of $D$ on $A^n$ is an element of $k^u$.

The last part is even easier: suppose that the action of $D$ on $A^i$ is diagonalizable for all $i\in I$. It is elementary that, if $x,y$ are eigenvectors for the action of $D$ with eigenvalues $\lambda_x,\lambda_y$ respectively, then since $D$ is an ungraded derivation, $xy$ is an eigenvector for $D$ with eigenvalue $\lambda_x + \lambda_y$. Since every element in $A^{\bullet}$ is a $k$-linear combination of products of elements in $\{ A^i: i\in I\}$ and since $D$ acts diagonalizably on $\bigoplus_{i\in I} A_i$, we have that every element in $A^{\bullet}$ is an $k$-linear combination of eigenvectors for the action of $D$, i.e., $D$ acts diagonalizably on $A^{\bullet}$.
\end{proof}

\begin{prop}\label{iterated derivation is idempotent}
Suppose $k$ is a field, $A$ is a graded $k$-algebra such that $A$ is a finite-dimensional $k$-vector space, and $D: A\rightarrow A$ is a degree-preserving $k$-linear derivation which has the potential $u$-property. Suppose that either $k$ has positive characteristic or that the action of $D$ on $A$ is diagonalizable. Then there exists some positive integer $n$ such that the $n$th iterate of $D$, $D^{\circ n}: A \rightarrow A$, is diagonalizable and idempotent.
\end{prop}
\begin{proof}
This is a matter of routine linear algebra, but we include the proof, to make it clear that $n$ can be chosen so that $D^{\circ n}$ is indeed diagonalizable and not merely {\em potentially} diagonalizable. Choose some field extension $K/k$ such that $D\otimes_k K: A\otimes_k K \rightarrow A\otimes_k K$ has the $u$-property.
By Lemma \ref{basis of eigenvectors}, each $A^n\otimes_k K$ admits a basis of generalized eigenvectors for the action of $D\otimes_k K$ with the property that all the eigenvalues are in $K^u$. Since $K$ is a field, $K^u$ is the union of $\{0\}$ with the set of roots of unity in $K$. Finite-dimensionality of $A$ as a $k$-vector space gives us that the action of $D\otimes_k K$ on $A\otimes_k K$ has only finitely many eigenvalues, so let $i$ be the least common multiple of the orders of the roots of unity that occur as eigenvalues of the action of $D\otimes_k K$ on $A\otimes_k K$. Hence $\left(D\otimes_k K\right)^{\circ i} = D^{\circ i}\otimes_k K$ has its eigenvalues contained in $\{ 0,1\}$. Hence $D^{\circ i}$ also has its eigenvalues contained in $\{0,1\}$.

If the action of $D$ on $A$ is diagonalizable, then we are done: $D^{\circ i}$ is idempotent and also diagonalizable.
However, if the action of $D$ on $A$ is not diagonalizable, then $D^{\circ i}$ might not be idempotent on $A$: if the action of $D$ on $A$ has a very large (relative to $i$) Jordan block, for example, then $D^{\circ i}$ might still be a block sum of large Jordan blocks, and consequently not idempotent. 

However, when $k$ has positive characteristic, we can fix this by taking an even larger power of $D$. Let $p$ be the characteristic of $k$. Given an $n$-by-$n$ Jordan block matrix $J$ and a prime number $\ell$, let $\ell^a$ be the smallest power of $\ell$ such that $\ell^a>n$. Then the entries in $J^{\ell^a}$ above the main diagonal are all divisible by $\ell$. In particular, if $\ell=p$, then $J^{\ell^a}$ is a diagonal matrix over $k$. 

So let $p^b$ be the smallest power of $p$ such that $p^b$ is greater than  the size of every Jordan block in the Jordan decomposition of the action of $D$ on $A$. Then $D^{\circ p^b}$ acts diagonally on $A$. Letting $n$ be the least common multiple of $i$ and of $p^b$, we get that $D^{\circ n}$ acts diagonally on $A$ with its eigenvalues contained in $\{0,1\}$. So $D^{\circ n}$ is idempotent and diagonalizable, as desired.
\end{proof}

\begin{observation}\label{dh+hd is a deriv}
Beginning in Theorem \ref{thm on diffs giving chain htpies} we will use the following elementary observation to build chain homotopies with good multiplicative properties. If $A$ is a graded $k$-algebra and $h_1,h_2$ are graded $k$-linear derivations on $A$ of odd degree (i.e., $h_1$ and $h_2$ each raise or lower grading degree by some odd number of degrees), then the composites $h_1h_2$ and $h_2h_1$ are not usually derivations at all. But 
$h_1h_2 + h_2h_1$ is an {\em ungraded} derivation on $A$. 
In particular, if $h$ is a graded $k$-linear derivation of degree $-1$ on a differential graded $k$-algebra $A$, then $hd+dh$ is a {\em ungraded}, but nevertheless degree-preserving, $k$-linear derivation on $A$. (See Conventions \ref{conventions} for our terms for various grading conditions on derivations.)
\end{observation}

\subsection{How to build a model of a DGA from a derivation on that DGA}

Now we can use the various elementary observations in \cref{elementary lin alg} to prove something which is a bit less elementary. The main result in this section is Corollary \ref{main thm cor}, which gives a simple and surprisingly general way to take a DGA over a finite field, and construct a sub-DGA of it which has the same cohomology.
\begin{theorem}\label{thm on diffs giving chain htpies}
Let $k$ be a field and let $A^{\bullet}$ be a differential graded $k$-algebra which is additionally equipped with a graded $k$-linear derivation $h: A^{\bullet} \rightarrow A^{\bullet-1}$.
Let $I$ be a set of integers such that $A^{\bullet}$ is generated, as an $k$-algebra, by its elements\footnote{One can always simply take $I = \mathbb{Z}$, and then the statement of the theorem is simpler. The advantage of allowing $I$ to be strictly smaller than $\mathbb{Z}$ is that, in practical situations, it is easier to check the hypotheses of the theorem in a more restricted range of grading degrees. In all of our motivating applications of Theorem \ref{thm on diffs giving chain htpies}, we shall let $I$ be $\{ 1\}$.} in degrees $\in I$.
Make the following assumptions:
\begin{enumerate}
\item $A^{\bullet}$ is finite-dimensional as an $k$-vector space,
\item the action of $dh+hd: A^{\bullet} \rightarrow A^{\bullet}$ on $A^{\bullet}$ has the potential $u$-property,
\item $h\circ h$ vanishes on a set of homogeneous $k$-algebra generators for $A^{\bullet}$,
\item and either $k$ has positive characteristic or the action of $dh+hd$ on $A^{i}$ is potentially diagonalizable 
for all $i\in I$.
\end{enumerate}
Then there exists some integer $n$ such that the kernel of $(dh)^n+(hd)^n: A^{\bullet} \rightarrow A^{\bullet}$ is a chain-homotopy deformation retract of $A^{\bullet}$. If the action of $dh+hd$ on $A^i$ is potentially diagonalizable for all $i\in I$, then we may take $n$ to be $1$, and in that case $\ker (dh+hd)$ is furthermore a sub-DGA of $A^{\bullet}$.
\end{theorem}
\begin{proof}
By Observation \ref{dh+hd is a deriv}, $dh+hd$ is an ungraded derivation, and the assumptions made in the statement of the Theorem are enough for Proposition \ref{iterated derivation is idempotent} to imply that $(dh+hd)^n$ is idempotent for some $n$. Since $hh$ vanishes on a set of generators for the $k$-algebra $A^{\bullet}$, routine induction on word length yields that $hh=0$. Of course $dd=0$ since $A^{\bullet}$ is a DGA, so 
\begin{equation}\label{dh+hd eq 1} (dh+hd)^n = (dh)^n + (hd)^n = d(h(dh)^{n-1}) + (h(dh)^{n-1})d.\end{equation}
Clearly $\id_{A^{\bullet}} - \left((dh)^n + (hd)^n\right)$ commutes with $d$, i.e., it is a chain map. Equation \eqref{dh+hd eq 1} establishes that $h(dh)^{n-1}$ is a chain homotopy between $\id_{A^{\bullet}}$ and $\id_{A^{\bullet}} - \left((dh)^n + (hd)^n\right)$.

Let $C^{\bullet}$ denote the image of the map $\id_{A^{\bullet}} - \left((dh)^n + (hd)^n\right): A^{\bullet}\rightarrow A^{\bullet}$. 
Let $\iota: C^{\bullet}\rightarrow A^{\bullet}$ denote the inclusion map, and let $\pr: A^{\bullet}\rightarrow C^{\bullet}$ denote the map $\id_{A^{\bullet}} - \left((dh)^n + (hd)^n\right)$ with its codomain restricted to its image, $C^{\bullet}$.

It is a general fact that, if $\phi$ is an idempotent endomorphism of an abelian group, then so is $\id - \phi$.
Hence, since $(dh+hd)^n = (dh)^n + (hd)^n$ is idempotent, $\id - \left((dh)^n + (hd)^n\right)$ is also idempotent. So for any $x\in C^{\bullet}$, we have
 $(\pr\circ \iota)(x) = \left(\id - \left((dh)^n + (hd)^n\right)\right)(x) = x$. Consequently $\pr\circ \iota = \id_{C^{\bullet}}$, so $C^{\bullet}$ is a retract of $A^{\bullet}$.

Idempotence of $(dh)^n + (hd)^n$, along with its diagonalizability (established by Proposition \ref{iterated derivation is idempotent}), gives us that $A^{\bullet}$ has an $k$-linear basis consisting of eigenvectors for the action of $(dh)^n+(hd)^n$, whose eigenvalues are each $0$ or $1$. Consequently $C^{\bullet}$, the image of $\id - \left((dh)^n + (hd)^n\right): A^{\bullet}\rightarrow A^{\bullet}$, is equal to the $0$-eigenspace of $(dh)^n + (hd)^n$.

If we are in the circumstance that $dh+hd$ acts potentially diagonalizably on $A^i$ for all $i\in I$, then by Lemma \ref{basis of eigenvectors},
$A^{\bullet}\otimes_k K$ has a $K$-linear basis consisting of eigenvectors for the action of $(dh+hd)\otimes_k K$, for some field extension $K/k$. Each such eigenvector has eigenvalue which is either 
\begin{itemize}
\item an $m$th root of unity for some integer $m$, in which case that eigenvector is in the $1$-eigenspace for $\left((dh)^m+(hd)^m\right)\otimes_k K$, 
\item or $0$, in which case that eigenvector is in the kernel of $\left((dh)^m+(hd)^m\right)\otimes_k K$. 
\end{itemize}
Since $K$ is a field, its only root of zero is zero itself. So the kernel of $\left((dh)^m+(hd)^m\right)\otimes_k K$ on $A^{\bullet}$ is equal to the kernel of $\left(dh+hd\right)\otimes_k K$ on $A^{\bullet}\otimes_k K$, i.e., $\ker\left((dh)^m+(hd)^m\right) = \ker\left(dh+hd\right)$. Our point is that, when $dh+hd$ acts potentially diagonalizably on $A^i$ for all $i\in I$, then we can indeed take $n$ to be $1$ in the statement of the theorem. It is elementary and routine to verify that $\ker (dh+hd)$ is a sub-DGA of $A^{\bullet}$, using Observation \ref{dh+hd is a deriv}.

Finally, we already know that 
\begin{align*}
 \id_{A^{\bullet}} - (\iota\circ\pr)
  &= \id_{A^{\bullet}} - \left(\id_{A^{\bullet}} - \left((dh)^n + (hd)^n\right)\right) \\
  &= d(h(dh)^{n-1}) + (h(dh)^{n-1})d,
\end{align*}
so $C^{\bullet}$ is not only a retract of $A^{\bullet}$, but in fact a chain-homotopy deformation retract of $A^{\bullet}$.
\end{proof}

\begin{corollary}
Let $k,A^{\bullet},h,n$ be as in the statement of Theorem \ref{thm on diffs giving chain htpies}. Then the cohomology of the sub-cochain-complex $\ker \left((dh)^{n} + (hd)^{n}\right)$ of $A^{\bullet}$ is isomorphic to $H^*(A^{\bullet})$.
\end{corollary}

\begin{corollary}\label{main thm cor}
Let $k$ be a finite field and let $A^{\bullet}$ be a differential graded $k$-algebra. Suppose that $I$ is a set of integers such that $A^{\bullet}$ is generated, as an $k$-algebra, by elements in grading degrees in $I$. Suppose furthermore that $A^{\bullet}$ is a finite-dimensional $k$-vector space.
If $h: A^{\bullet} \rightarrow A^{\bullet-1}$ is a $k$-linear graded derivation such that $h(h(x))=0$ for all $x\in A^i$ and all $i\in I$, 
then there exists some integer $m$ such that $\ker \left((dh)^m + (hd)^m\right)$ is a chain-homotopy deformation retract of $A^{\bullet}$.

Furthermore, if the action of $dh+hd$ on $A^i$ is potentially diagonalizable for all $i\in I$, then we can take $m$ to be $1$, and in that case $\ker (dh+hd)$ is also a sub-DGA of $A^{\bullet}$.
\end{corollary}
\begin{proof}
Over a finite field, every linear operator has the potential $u$-property, as explained in Example \ref{u-prop example}. Hence this is a special case of Theorem \ref{thm on diffs giving chain htpies}.
\end{proof}

\subsection{Small models for the cohomology of a Lie algebra}

In this subsection, we apply Theorem \ref{thm on diffs giving chain htpies} and Corollary \ref{main thm cor} to the Chevalley-Eilenberg DGAs of Lie algebras. 

\begin{definition}
Let $R$ be a commutative ring, and let $L$ be a Lie $R$-algebra. By a {\em model for the cohomology of $L$,} we mean a sub-DGA $B^{\bullet}$ of the Chevalley-Eilenberg DGA $\Lambda^{\bullet}_R(L^*)$ such that the inclusion $B^{\bullet}\hookrightarrow\Lambda^{\bullet}_R(L^*)$ is a quasi-isomorphism.
\end{definition}

Recall the following classical definitions: the {\em adjoint representation} of a Lie $k$-algebra $L$ is the $k$-linear representation $\ad: L \rightarrow \mathfrak{gl}(L)$ given by $\ad(\ell)(x) = [\ell,x]$. 
An element $\ell\in L$ is called {\em semisimple} if $\ad(\ell)\in \mathfrak{gl}(L)$ is potentially diagonalizable.

Let $L$ be a finite-dimensional Lie algebra over a field $k$. Here are three observations which, taken together, drastically simplify the setup for Theorem \ref{thm on diffs giving chain htpies} and Corollary \ref{main thm cor}:
\begin{enumerate}
\item The Chevalley-Eilenberg DGA $\Lambda^{\bullet}_k(L^*)$ is generated in cohomological degree $1$, so the set of integers $I$ in the statement of Theorem \ref{thm on diffs giving chain htpies} and Corollary \ref{main thm cor} can be taken to be simply $\{ 1\}$. 
\item Furthermore, every $k$-linear derivation $h: L^* = \Lambda^1_k(L^*) \rightarrow\Lambda^0_k(L^*) = k$ extends uniquely to a graded $k$-linear derivation $\Lambda^{\bullet}_k(L^*) \rightarrow\Lambda^{\bullet-1}_k(L^*)$. 
\item There is a natural bijection between $k$-linear functionals $h: L^*\rightarrow k$ and elements of $L$. 
Hence every $k$-linear functional $h$ of the kind appearing in the statement of Corollary \ref{main thm cor} must arise from some element $\ell\in L$. 
\end{enumerate}
The point is that, once the Lie algebra $L$ is specified, in order to apply Corollary \ref{main thm cor} to the Chevalley-Eilenberg complex $\Lambda^{\bullet}_k(L^*)$, we need only specify a single element $\ell$ of $L$. One needs to check that $dh+hd:\Lambda^{\bullet}_k(L^*) \rightarrow \Lambda^{\bullet}_k(L^*)$ is potentially diagonalizable. It is a routine exercise to check that $dh+hd: L^* = \Lambda^{1}_k(L^*) \rightarrow \Lambda^{1}_k(L^*) = L^*$ is dual to $2\cdot \ad_{\ell}$. Hence, if $k$ has characteristic $\neq 2$, then $dh+hd$ is potentially diagonalizable if and only if $\ell$ is semisimple. Corollary \ref{main thm cor} now gives us:
\begin{theorem}\label{main thm on models for lie algs}
Let $L$ be a finite-dimensional Lie algebra over a finite field $k$ of characteristic $\neq 2$. Let $\ell$ be a semisimple element of $L$, and write $h$ for the $k$-linear functional $h: L^{*}\rightarrow k$ given by $h(\phi) = \phi(\ell)$. Then $h$ extends to a unique graded $k$-linear derivation $h^{\ell}: \Lambda^{\bullet}_k(L^*) \rightarrow\Lambda^{\bullet-1}_k(L^*)$, and the sub-DGA $\ker (dh^{\ell}+h^{\ell}d)$ of $\Lambda^{\bullet}_k(L^*)$ is a model for the cohomology of $L$.
\end{theorem}

Theorem \ref{main thm on models for lie algs} varies in its usefulness, depending on the Lie algebra $L$. In the trivial case of Theorem \ref{main thm on models for lie algs}, one lets the semisimple element $\ell$ be zero, and the hypotheses of the theorem are satisfied, but then the theorem does not tell us anything useful. In that case, we would have $dh^{\ell}+h^{\ell}d = 0$, so the theorem would simply state that $\Lambda^{\bullet}_k(L^*)$ is a model for the cohomology of $L$, which of course we already know perfectly well.

Let us fix a Lie algebra $L$. In order to get a handle on the nontrivial cases of Theorem \ref{main thm on models for lie algs}, preorder the semisimple elements of $L$ by letting $\ell\leq \ell^{\prime}$ if $\ker (dh^{\ell} + h^{\ell}d)\subseteq \ker (dh^{\ell^{\prime}} + h^{\ell^{\prime}}d)$. This preordering is not necessarily antisymmetric, i.e., it may happen that $\ell\leq \ell^{\prime}$ and $\ell^{\prime}\leq \ell$ are both true even when $\ell\neq \ell^{\prime}$. Let us write $\ell\sim \ell^{\prime}$ if $\ell\leq \ell^{\prime}$ and $\ell^{\prime}\leq \ell$ are both true. 

For some Lie algebras $L$, {\em every} semisimple element is equivalent to zero in the above sense; for example, this happens when $L$ is the three-dimensional Heisenberg Lie algebra. For such Lie algebras, Theorem \ref{main thm on models for lie algs} is useless: the only model it produces for the cohomology of $L$ is the entire Chevalley-Eilenberg complex itself. From a handful of calculations in special cases, we suspect that $L$ is nilpotent if and only if all semisimple elements in $L$ are equivalent to zero, but we have not attempted to prove this. 

In the opposite extreme, we have the semisimple Lie algebras, for which there tend to be relatively many equivalence classes of semisimple elements. For semisimple $L$, Theorem \ref{main thm on models for lie algs} produces an assortment of different models for the cohomology of $L$. We have not investigated the structure of the preorder of semisimple elements, e.g. whether it is uniserial up to equivalence, or how many minimal equivalence classes it has. Perhaps this is worth looking into, but leave off such investigations for another time. 

Our main motivating case of Theorem \ref{main thm on models for lie algs} is the reductive case $L = \mathfrak{gl}_n(k)$, which we handle in \cref{Applications to the cohomology...}.

\subsection{The critical complex in $\Lambda^{\bullet}_k(\mathfrak{gl}_n(k)^*)$ models the modular Lie algebra $\mathfrak{gl}_n(k)$}
\label{Applications to the cohomology...}

Throughout, fix a positive integer $n$ and a finite field $k$.

We begin with a result which is not new: Theorem \ref{sln cohomology} is the Dynkin type $A_{n-1}$ case of a 1986 theorem of Friedlander and Parshall.
\begin{theorem}\label{sln cohomology}\cite[Theorem 1.2]{MR0821318}
Let $n$ be a positive integer, and let $p$ be a prime satisfying $p\geq 3n-3$. Then the cohomology $H^*(\mathfrak{sl}_n(\mathbb{F}_p);\mathbb{F}_p)$ of the Lie algebra $\mathfrak{sl}_n(\mathbb{F}_p)$ is isomorphic, as a graded $\mathbb{F}_p$-algebra, to the exterior $\mathbb{F}_p$-algebra $\Lambda_{\mathbb{F}_p}(x_3, x_5, \dots ,x_{2n-1})$, with $x_i$ in degree $i$.
\end{theorem}

\begin{prop}\label{gln cohomology}
Let $n,p$ be as in Theorem \ref{sln cohomology}. Then the cohomology $H^*(U(n); \mathbb{F}_p)$ of the unitary group is isomorphic to the cohomology $H^*(\mathfrak{gl}_n(\mathbb{F}_p); \mathbb{F}_p)$ of the Lie algebra $\mathfrak{gl}_n(\mathbb{F}_p)$. 
\end{prop}
\begin{proof}
We show $H^*(U(n); \mathbb{F}_p) \cong H^*(\mathfrak{gl}_n(\mathbb{F}_p); \mathbb{F}_p)$ by direct computation of each side. From well-known Serre spectral sequence calculations, one gets that $H^*(U(n); \mathbb{F}_p)$ is $\Lambda_{\mathbb{F}_p}(x_1, x_3, x_5, ..., x_{2n-1})$. Similarly, $\mathfrak{gl}_n(\mathbb{F}_p)$ splits as the product of a one-dimensional abelian Lie algebra $T$ with the semisimple Lie algebra $\mathfrak{sl}_n(\mathbb{F}_p)$. The cohomology of the former is $\Lambda_{\mathbb{F}_p}(x_1)$, while the cohomology of the latter is $\Lambda_{\mathbb{F}_p}(x_3, \dots ,x_{2n-1})$, by Theorem \ref{sln cohomology}. Consequently $H^*(\mathfrak{gl}_n(\mathbb{F}_p); \mathbb{F}_p)$ is, like $H^*(U(n);\mathbb{F}_p)$, isomorphic to $\Lambda_{\mathbb{F}_p}(x_1, x_3, x_5, ..., x_{2n-1})$.
\end{proof}

\begin{prop}\label{def of Adot}
Let $R$ be a commutative ring, and let $n$ be a positive integer. The Chevalley-Eilenberg differential graded $R$-algebra $\Lambda^{\bullet}_R(\mathfrak{gl}_n(R)^*)$ of the Lie $R$-algebra $\mathfrak{gl}_n(R)$ has $R$-linear basis
$\{ h_{i,j}: i,j\in Z_n \}$, where $h_{i,j}: \mathfrak{gl}_n(R)\rightarrow R$ is the linear functional that reads off the entry in a matrix $M\in\mathfrak{gl}_n(R)$ in row $j$ and column $i+j$. The differential on $\Lambda^{\bullet}_R(\mathfrak{gl}_n(R)^*)$ is given by
\begin{align}
\label{diff 0945} d(h_{i,j}) &= \sum_{\ell=1}^n h_{\ell,j}h_{i-\ell,j+\ell},
\end{align}
with the first and second subscripts in $h_{i,j}$ understood as being defined modulo $n$.
\end{prop}
\begin{proof}
Write $x_{i,j}$ for the matrix in $\mathfrak{gl}_n(R)$ with a $1$ in row $j$ and column $i+j$, and a $0$ in all other entries. It is elementary to verify that the Lie bracket in $\mathfrak{gl}_n(R)$ is given by
\begin{align}\label{lie bracket 0945}
 [x_{i,j},x_{k,\ell}] &=  \delta_{i+j}^{\ell} x_{i+k,j} - \delta_{k+\ell}^j x_{i+k,\ell}.\end{align}
Dualizing formula \eqref{lie bracket 0945} yields \eqref{diff 0945}.
\end{proof}

We will refer to the set of products of elements in the set $\{ h_{i,j}: i,j\in Z_n \}$ as the {\em $h$-basis} for $\Lambda^{\bullet}(\mathfrak{gl}_n(R)^{\ast})$.

\begin{definition}\label{gradings def}
Let $n$ be a positive integer, and let $k$ be a field of characteristic $p>0$. 
\begin{itemize}
\item Given integers $r,c\in Z_n$ and $\alpha\in k$, we write $M_{r,c}(\alpha)$ for the matrix with $\alpha$ in row $r$ and column $c$, and zero in all other entries. 
\item
The Lie algebra $\mathfrak{gl}_n(k)$ admits a $\mathbb{Z}/\left(\frac{p^n-1}{p-1}\right)\mathbb{Z}$-grading\footnote{See Conventions \ref{conventions} for our conventions for gradings on Lie algebras.} given by letting the matrix $M_{r,c}(\alpha)$ be in degree $p^r(p^{c-r}-1)/(p-1)$. We call this the {\em reduced internal grading on $\mathfrak{gl}_n(k)$.}
\item 
There is also a $\mathbb{Z}$-grading on the differential graded algebra $\Lambda^{\bullet}_k(\mathfrak{gl}_n(k)^*)$ given by 
$\left|\left| h_{i,j}\right|\right| = p^j\frac{p^i-1}{p-1}$.
Upon reducing this $\mathbb{Z}$-grading to a $\mathbb{Z}/\left( \frac{p^n-1}{p-1}\right)\mathbb{Z}$-grading, it coincides with the $\mathbb{Z}/\left( \frac{p^n-1}{p-1}\right)\mathbb{Z}$-grading on $\Lambda^{\bullet}_k(\mathfrak{gl}_n(k)^*)$ induced by the internal grading on $\mathfrak{gl}_n(k)$. Hence we refer to this $\mathbb{Z}$-grading on $\Lambda^{\bullet}_k(\mathfrak{gl}_n(k)^*)$ as the {\em reduced internal grading}\footnote{Recall that the purpose of this paper is to calculate the cohomology of the Morava stabilizer group---equivalently, the Morava stabilizer algebra $\Sigma(n)\cong K(n)_*\otimes_{BP_*}BP_*BP\otimes_{BP_*} K(n)_*$---with appropriate coefficients. The Hopf algebra $\Sigma(n)$ inherits a grading from $BP_*BP$, traditionally called the ``internal grading.'' In this grading, $\Sigma(n)$ is concentrated in degrees divisible by $2(p-1)$. Under the comparison we will ultimately draw between the cohomology of the Hopf algebra $\Sigma(n)$ and the cohomology of the Lie algebra $\mathfrak{gl}_n(\mathbb{F}_p)$, the cohomology of $\mathfrak{gl}_n(\mathbb{F}_p)$ in reduced internal degree $m$ corresponds to the cohomology of $\Sigma(n)$ in internal degree $2(p-1)m$. Hence the name ``reduced internal grading'': it is simply the internal grading coming from $BP_*BP$, but ``reduced'' by being divided by $2(p-1)$.}.
\item 
We say that an element of $\Lambda^{\bullet}_k(\mathfrak{gl}_n(k)^*)$ is 
\begin{itemize}
\item {\em homogeneous} if it is homogeneous with respect to the cohomological grading,
\item {\em internally homogeneous} if it is homogeneous with respect to the reduced internal grading,
\item and {\em bihomogeneous} if it is both homogeneous and internally homogeneous.
\end{itemize}
\item
The differential $d$ on $\Lambda^{\bullet}_k(\mathfrak{gl}_n(k)^*)$ preserves 
the reduced internal grading. Consequently there is a sub-DGA of $\Lambda^{\bullet}_k(\mathfrak{gl}_n(k)^*)$ consisting of the internally homogeneous elements whose reduced internal degree is a multiple of $\frac{p^n-1}{p-1}$. We call $cc^{\bullet}(\mathfrak{gl}_n(k))$ for this sub-DGA of $\Lambda^{\bullet}_k(\mathfrak{gl}_n(k)^*)$, and we call it the {\em critical complex} of $\mathfrak{gl}_n(k)$. 
\end{itemize}
\end{definition}

Our present goal is to apply Corollary \ref{main thm cor} to show that the critical complex is a model for the Lie algebra cohomology of $\mathfrak{gl}_n(\mathbb{F}_p)$. That goal will be accomplished in Theorem \ref{main chain htpy application thm}. The idea is as follows: we would like to define a linear functional $\lambda$ on $\Lambda^{\bullet}_k(\mathfrak{gl}_n(k)^*)$ and a derivation $h$ on $\Lambda^{\bullet}_k(\mathfrak{gl}_n(k)^*)$ which sends each $1$-cochain $x$ to some element $h(x)\in k$ such that
\begin{itemize}
\item if $x$ is an eigenvector of the action of $dh+hd$ on the $1$-cochains, then $\lambda(x)$ is equal to the eigenvalue of $dh+hd$ acting on $x$, and 
\item for all bihomogeneous $y\in \Lambda^{\bullet}_k(\mathfrak{gl}_n(k)^*)$, we have $\lambda(y) = 0$ if and only if the reduced internal degree $\left|\left| y\right|\right|$ is divisible by $\frac{(p^n-1)}{p-1}$. 
\end{itemize}
These two conditions would suffice to guarantee that the kernel of $dh+hd$ is equal to the kernel of $\lambda$, which is in turn equal to the critical complex. Then Corollary \ref{main thm cor} will ensure that the critical complex has the same cohomology as $\Lambda^{\bullet}_k(\mathfrak{gl}_n(k)^*)$. 

This is indeed the strategy we pursue in the proof of Theorem \ref{main chain htpy application thm}, and it works straightforwardly if $n$ is prime. However, if $n$ is not prime, then we are unable to find a {\em single} linear functional satisfying both of the above conditions, essentially because the polynomial $\frac{x^n-1}{x-1}$ fails to be irreducible over $\mathbb{Q}$ if $n$ is not prime. For composite $n$, the argument is a bit more subtle, requiring several linear functionals, one for each irreducible factor of the polynomial $\frac{x^n-1}{x-1}$, and consequently several applications of \ref{main chain htpy application thm}. 

Here is how we produce the relevant linear functionals. Every $k$-linear map $\mathfrak{gl}_n(k)^* \rightarrow k$ extends uniquely to a $k$-linear derivation $\Lambda^{\bullet}_k(\mathfrak{gl}_n(k)^*)\rightarrow \Lambda^{\bullet-1}_k(\mathfrak{gl}_n(k)^*)$, so we simply need to pick the right $k$-linear maps $\lambda,h: \mathfrak{gl}_n(k)^* \rightarrow k$.  They are as follows:

\begin{definition}\label{def of lambda and h}
Let $n$ be a positive integer, let $k$ be a field, and let $\omega$ be an $n$th root of unity (not necessarily primitive!) in $k$ satisfying $\omega\neq 1$.
Let $\lambda^{\omega}: \Lambda^{\bullet}(\mathfrak{gl}_n(k)^*)\rightarrow k$ be the $k$-linear function given by
\begin{align}
\label{def of lambda} \lambda^{\omega}(h_{i,j}) &= \sum_{\ell=1}^{i} \omega^{j+\ell}.
\end{align}

Meanwhile, let $h^{\omega}: \Lambda^{\bullet}(\mathfrak{gl}_n(k)^*)\rightarrow \Lambda^{\bullet-1}_k(\mathfrak{gl}_n(k)^*)$ be the $k$-linear graded derivation determined uniquely by its value\footnote{We apologize for how the left-hand side of the formula \eqref{def of h} features the letter $h$ playing two entirely roles. Here is our justification for this notation: the symbol $h_{i,j}$ is an element of the $h$-basis for $\Lambda^{\bullet}(\mathfrak{gl}_n(R)^{\ast})$, and the notation $h_{i,j}$ is traditional in homotopy theory. For example, as a consequence of the results of \cref{Cohomology of the singular...}, the element $h_{2,0} + h_{2,1}$ in $\Lambda^{2}(\mathfrak{gl}_2(R)^{\ast})$ is indeed responsible for the well-known (e.g. in \cite{MR0458423}) element $\zeta_2 := h_{2,0}+ h_{2,1}$ in the cohomology of the height $2$ Morava stabilizer group. Meanwhile, the symbol $h^{\omega}$ {\em without} a subscript in \eqref{def of h} is a derivation which, using Theorem \ref{main thm on models for lie algs}), will turn out to be a chain homotopy, and it is traditional to write $h$ for a chain homotopy. We occasionally write $h$ with a {\em superscript} (e.g. $h^{\omega}$) when $h$ denotes a chain homotopy, e.g. in \eqref{def of lambda} and in the statement of Theorem \ref{main thm on models for lie algs}, but when the letter $h$ has a {\em subscript} it will always denote an $h$-basis element, not a chain homotopy. We hope this avoids any possible confusion.}
\begin{align}
\label{def of h} h^{\omega}(h_{i,j}) &= \left\{ \begin{array}{ll} \omega^j \frac{\omega}{1-\omega} &\mbox{\ if\ } i=n \\ 0 &\mbox{\ otherwise}\end{array}\right.
\end{align}
on the $1$-cochains $\Lambda^{1}_k(\mathfrak{gl}_n(k)^*) = \mathfrak{gl}_n(k)^*$.
\end{definition}
Under the correspondence between $k$-linear functionals $\mathfrak{gl}_n(k)^*\rightarrow k$ and elements of $\mathfrak{gl}_n(k)$, $h^{\omega}$ corresponds to the diagonal matrix whose diagonal entry on the $j$th row is $\omega^j\frac{\omega}{1-\omega}$, e.g. $\frac{1}{2} \left[ \begin{array}{cc} 1 & 0 \\ 0 & -1 \end{array}\right]$ if $n=2$, and $\frac{1 - \omega^2}{3}\left[ \begin{array}{ccc} 1 & 0 & 0 \\ 0 & \omega & 0 \\ 0 & 0 & \omega^2 \end{array}\right]$ if $n=3$.

\begin{lemma} \label{dh+hd and lambda compatibility}
Suppose that $k$ is a finite field containing an $n$th root of unity $\omega$ satisfying $\omega\neq 1$.
For each $i,j$,
the $1$-cochain $h_{i,j}$ of $\Lambda^{\bullet}_k(\mathfrak{gl}_n(k)^*)$ is an eigenvector for the action of $dh^{\omega}+h^{\omega}d$ on $\Lambda^{1}_k(\mathfrak{gl}_n(k)^*)$, with eigenvalue $\lambda^{\omega}(h_{i,j})$.
\end{lemma}
\begin{proof}
Routine calculation.
\end{proof}

\begin{theorem}\label{main chain htpy application thm}
Suppose $k$ is a finite field of characteristic not dividing $n$, and suppose that $k$ contains a primitive $n$th root of unity. Then the inclusion $cc^{\bullet}(\mathfrak{gl}_n(k))\hookrightarrow \Lambda^{\bullet}_k(\mathfrak{gl}_n(k)^*)$ is a quasi-isomorphism.
\end{theorem}
\begin{proof}
The argument was already sketched preceding Definition \ref{def of lambda and h}. Choose a primitive $n$th root of unity $\omega$ in $k$. We apply Theorem \ref{thm on diffs giving chain htpies}, using the graded $k$-linear derivation $h^{\omega}: \Lambda^{\bullet}_k(\mathfrak{gl}_n(k)^*) \rightarrow \Lambda^{\bullet-1}_k(\mathfrak{gl}_n(k)^*)$ defined in equation \eqref{def of h} of Definition \ref{def of lambda and h}. Since $k$ is a finite field, $k = k^u$, so the action of $dh^{\omega}+h^{\omega}d$ on $\Lambda^{\bullet}_k(\mathfrak{gl}_n(k)^*)$ has the $u$-property\footnote{We let the set $I$ in the definition of the $u$-property, Definition \ref{def of u-property}, be $I = \{1\}$, since $\Lambda^{\bullet}_k(\mathfrak{gl}_n(k)^*)$ is generated as an $k$-algebra by elements in degree $1$.} if and only if $\Lambda^{1}_k(\mathfrak{gl}_n(k)^*)$ has a basis of generalized eigenvectors for the action of $dh^{\omega}+h^{\omega}d$, i.e., if and only if all the eigenvalues of $dh^{\omega}+h^{\omega}d$ in the algebraic closure $\overline{k}$ are already contained in $k$. By Lemma \ref{dh+hd and lambda compatibility}, the $h$-basis is a basis of eigenvectors for the action of $dh^{\omega}+h^{\omega}d$ on $\Lambda^{1}_k(\mathfrak{gl}_n(k)^*)$, and the eigenvalue of $dh^{\omega}+h^{\omega}d$ acting on $h_{i,j}$ is $\lambda^{\omega}(h_{i,j})$, defined in \eqref{def of lambda}, which is certainly contained in $k$.

We have $h^{\omega}\circ h^{\omega} = 0$ on $\Lambda^{1}_k(\mathfrak{gl}_n(k)^*)$ for degree reasons, and both of the conditions ``$k$ has positive characteristic'' and ``the action of $dh^{\omega}+h^{\omega}d$ on $\Lambda^{\bullet}_k(\mathfrak{gl}_n(k)^*)$ is diagonalizable'' are satisfied. So all the hypotheses of Theorem \ref{thm on diffs giving chain htpies} are satisfied, and we are in the stronger case, the situation when $dh^{\omega}+h^{\omega}d$ acts diagonalizably on $\Lambda^{\bullet}_k(\mathfrak{gl}_n(k)^*)$. So the kernel of $dh^{\omega} + h^{\omega}d$ acting on $\Lambda^{\bullet}_k(\mathfrak{gl}_n(k)^*)$ is a chain-homotopy deformation retract of $\Lambda^{\bullet}_k(\mathfrak{gl}_n(k)^*)$. From now on, we will write $\ker (dh^{\omega}+h^{\omega}d)$ for the kernel of $dh^{\omega}+h^{\omega}d$ acting on $\Lambda^{\bullet}_k(\mathfrak{gl}_n(k)^*)$ (not just on the $1$-cochains $\Lambda^{1}_k(\mathfrak{gl}_n(k)^*)$).

From the value of $\lambda^{\omega}(h_{i,j})$, we can recover the reduced internal degree of $h_{i,j}$ {\em modulo $\Irr(\omega, \mathbb{Q},p)$}, i.e., modulo the $n$th cyclotomic polynomial evaluated at $p$. If $n$ is prime, then $\Irr(\omega, \mathbb{Q},p) = \frac{p^n-1}{p-1}$, so the kernel of $dh^{\omega} - h^{\omega}d$ is precisely the critical complex, and a single application of Theorem \ref{thm on diffs giving chain htpies} finishes the proof. If $n$ is instead {\em not} prime, then we have the following argument: for each irreducible factor of the polynomial $\frac{x^n-1}{x-1}\in\mathbb{Q}[x]$, choose a root $\omega_i$ of that factor in $k$. Write $\omega_1, \dots ,\omega_m$ for the resulting set of roots of unity in $k$. For each $i=1, \dots ,m$, Theorem \ref{thm on diffs giving chain htpies} yields that the sub-DGA $\ker ( dh^{\omega_i} + h^{\omega_i}d) \subseteq \Lambda^{\bullet}_k(\mathfrak{gl}_n(k)^*)$ is a chain-homotopy deformation retract. 

In general, if $A^{\bullet},B^{\bullet}$ are sub-DGAs of some DGA $C^{\bullet}$ such that the inclusions $A^{\bullet}\subseteq C^{\bullet}$ and $B^{\bullet}\subseteq C^{\bullet}$ are chain-homotopy equivalences, it is not necessarily true that the intersection $A^{\bullet}\cap B^{\bullet}$ is chain-homotopy equivalent to $C^{\bullet}$. However, in our situation, we can use the grading to our advantage: $\ker ( dh^{\omega_i} + h^{\omega_i}d)$ is the sub-DGA of $\Lambda^{\bullet}_k(\mathfrak{gl}_n(k)^*)$ consisting of all elements of reduced internal degree divisible by the irreducible polynomial of $\omega_i$ evaluated at $p$. Since the chain-homotopy equivalences furnished by each application of Theorem \ref{thm on diffs giving chain htpies} respect the internal grading, the intersection $\bigcap_{i=1}^m \ker (dh^{\omega_i} + h^{\omega_i}d)$ {\em is} still chain-homotopy equivalent to $\Lambda^{\bullet}_k(\mathfrak{gl}_n(k)^*)$. 

By construction, the intersection $\bigcap_{i=1}^m \ker (dh^{\omega_i} + h^{\omega_i}d)$ has $k$-linear basis consisting of the monomials in the $h$-basis for $\Lambda^{\bullet}_k(\mathfrak{gl}_n(k)^*)$ whose reduced internal degree is divisible by $\frac{p^n-1}{p-1}$, i.e., 
$\bigcap_{i=1}^m \ker (dh^{\omega_i} + h^{\omega_i}d) = cc^{\bullet}(\mathfrak{gl}_n(k))$ as a sub-DGA of $\Lambda^{\bullet}_k(\mathfrak{gl}_n(k)^*)$.
\end{proof}

\subsection{The descent problem for $H^*(U(n);k)$}
\label{The descent problem}

In this section we show that, if $K/k$ is a finite extension of fields and $A$ is a graded $k$-algebra such that $A\otimes_k K$ is isomorphic to $H^*(U(n); K)$ as a graded $K$-algebra, then under mild hypotheses, $A \cong H^*(U(n);k)$ as a graded $k$-algebra. This is a straightforward application of standard techniques in Galois descent; see chapter 17 of \cite{MR547117} for an excellent exposition. The computation of the relevant cohomology group is not difficult, but takes a few steps. Since we do not know of a place in the literature where this result or the relevant calculations appear, we present the relevant calculations in this section, although we suspect everything in this section must be already ``well-known to those who know it well.''

\begin{prop}\label{descent prop}
Let $d$ be a positive integer, and let $p$ be a prime which does not divide $d$. 
Let $K/k$ be a degree $d$ cyclic Galois extension of characteristic $p$, and let $V$ be a graded $k$-vector space in which all elements are in positive odd degrees and in which, for each integer $i$, the degree $i$ summand of $V$ is finite-dimensional.
Write $\Lambda_k(V)$ for the exterior $k$-algebra on the vector space $V$.
If $A$ is a graded $k$-algebra such that $A\otimes_k K$ is isomorphic to $\Lambda_k(V)\otimes_k K$ as a graded $K$-algebra, then $A$ is isomorphic to $\Lambda_k(V)$ as a graded $k$-algebra. 
\end{prop}
\begin{proof}
For any graded $k$-algebra $B$,
let $\Aut(B\otimes_k K)$ denote the automorphism group of $B\otimes_k K$ in the category of graded $K$-algebras. One knows from classical descent theory (e.g. chapters 17 and 18 of \cite{MR547117}) that the pointed set\linebreak $H^1(\Gal(K/k); \Aut(B\otimes_k K))$ is in bijection with the set of isomorphism classes of graded $k$-algebras $A$ such that $B\otimes_k K$ is isomorphic to $A\otimes_kK$ as a graded $K$-algebras.

The case at hand is $B = \Lambda_k(V)$. Of course $\Lambda_k(V)\otimes_k K \cong \Lambda_K(V)$.
Since an exterior algebra on a set of generators in odd grading degrees is the free graded-commutative algebra on those generators, we can uniquely specify a graded $k$-algebra endomorphism of $\Lambda_k(V)$ by specifying a grading-preserving $k$-linear map $V \rightarrow \Lambda_k(V)$, and any such $k$-linear map yields a graded $k$-algebra endomorphism. 

The set of degree $i$ elements of the automorphism group $\Aut(\Lambda_k(V)\otimes_k K)$ consists of the $K$-linear automorphism group of of $V^i\otimes_k K$ extended by $K$-linear endomorphisms of exterior products of lower-degree homogeneous elements of $V\otimes_k K$. That is, $\Aut(\Lambda_k(V)\otimes_k K)^i$ sits in an extension of nonabelian $\Gal(K/k)$-modules
\begin{equation}\label{ses 3309} 1 \rightarrow W_i \rightarrow \Aut(\Lambda_k(V)\otimes_k K)^i \rightarrow GL(V^i\otimes_k K) \rightarrow 1\end{equation}
where $W_i$ is a $K$-vector space with some action of $\Gal(K/k)$. 
The Galois cohomology set $H^1(\Gal(K/k);W_i)$ is trivial since $p$ does not divide the order of $\Gal(K/k)$, and the Galois cohomology set $H^1(\Gal(K/k);GL(V^i\otimes_k K))$ is trivial by a nonabelian version of Hilbert 90, e.g. Proposition 3 of Chapter X of \cite{MR0554237}. By the exact sequence of pointed sets induced in nonabelian Galois cohomology by \eqref{ses 3309}, $H^1(\Gal(K/k); \Aut(\Lambda_k(V)\otimes_k K)^i)$ is trivial. This is true for all integers $i$, so $H^1(\Gal(K/k); \Aut(\Lambda_k(V)\otimes_k K))$ is trivial, i.e., every $K/k$-twist of $\Lambda_k(V)$ is trivial.
\end{proof}


As a consequence of Proposition \ref{gln cohomology}, Theorem \ref{main chain htpy application thm} and Proposition \ref{descent prop}\footnote{To be completely explicit: we apply Proposition \ref{descent prop} in the case where the field extension $K/k$ is $\mathbb{F}_p(\omega)/\mathbb{F}_p$, where $\omega$ is a primitive $n$th root of unity. The degree $d = [K : k]$ is the least integer $m$ such that $n$ divides $p^m-1$. For the hypotheses of Proposition \ref{descent prop} to be satisfied, we only need $p$ to fail to divide that integer $d$. It suffices to let $p>\phi(n)$, since $d$ is a divisor of the degree of the $n$th cyclotomic polynomial, i.e., $\phi(n)$, but of course the hypotheses of Proposition \ref{descent prop} are also satisfied for many pairs $p,n$ with $p\leq \phi(n)$, as well.}, we have:
\begin{theorem}\label{gln and cc cohomology}
Let $n$ be a positive integer, and let $p$ be a prime such that $p > \phi(n)$ and such that $p$ does not divide $n$. Then we have an isomorphism of graded rings
\begin{align*}
 H^*\left(cc^{\bullet}(\mathfrak{gl}_n(\mathbb{F}_p))\right)
  &\cong H^*(\mathfrak{gl}_n(\mathbb{F}_p);\mathbb{F}_p)
\end{align*}
between the cohomology of the critical complex of $\mathfrak{gl}_n(\mathbb{F}_p)$ and the cohomology of the Lie algebra $\mathfrak{gl}_n(\mathbb{F}_p)$.
\end{theorem}
To be clear, the constraints $p>\phi(n)$ and $p\nmid n$ could be omitted from the statement of Theorem \ref{gln and cc cohomology} if it were possible to prove Theorem \ref{main chain htpy application thm} without adjoining a primitive $n$th root of unity to $\mathbb{F}_p$. We do not know whether this is possible, however.

\section{Monodromy and invariant cycles in a family of deformations of Ravenel's model}
\label{Monodromy section}

\subsection{Deformations of Ravenel's Lie algebra model for the Morava stabilizer group}
\label{Deformations of Ravenel's model.}

\begin{definition}\label{def of Lambda eps}
Let $R$ be a commutative ring, let $n$ be a positive integer, let $p$ be a prime number, and let $\epsilon \in R$.
\begin{itemize}
\item We write $\Lambda^{\bullet}_{n,\epsilon}$ for the differential graded $R$-algebra given by the exterior $R$-algebra
\begin{align*} \Lambda^{\bullet}_{n,\epsilon} &= \Lambda_R(h_{i,j}: i,j\in Z_n),\end{align*}
with each generator $h_{i,j}$ in cohomological degree $1$, and with differential
\begin{align}
\label{lambda diff formula} d(h_{i,j}) &= \sum_{\ell =1}^{i-1} h_{\ell,j}h_{i-\ell,j+\ell} + \epsilon\sum_{\ell=i}^n h_{\ell,j}h_{i-\ell+n,j+\ell}. \end{align}
We equip $\Lambda^{\bullet}_{n,\epsilon}$ with an additional $\mathbb{Z}/2(p^n-1)\mathbb{Z}$-grading, called the {\em internal grading}, given by letting the internal degree of $h_{i,j}$ be $2(p^i-1)p^j$. 
\end{itemize}
\end{definition}
The differential \eqref{lambda diff formula} increases cohomological degree by one, and preserves the internal degree.

\begin{definition}\label{crit complex def-prop}
Let $n$ be a positive integer, let $R$ be a commutative ring, and let $\epsilon\in R$.
\begin{itemize}
\item By the {\em critical complex} we mean the differential graded $R$-subalgebra $cc_{n,\epsilon}^{\bullet}$ of $\Lambda^{\bullet}_{n,\epsilon}$ consisting of all linear combinations of bihomogeneous elements of internal degree zero.  
\item By the {\em first-subscript complex} we mean the differential graded $R$-subalgebra $FSC^{\bullet}$ of $\Lambda^{\bullet}_{n,\epsilon}$ consisting of all linear combinations of bihomogeneous elements whose first subscripts, in the $h$-basis, sum to zero modulo $n$. 
\end{itemize}
\end{definition}

For example, when $n=3$, the element $h_{10}h_{20}$ is in the first-subscript complex, but not in the critical complex. This is because $h_{10}h_{20}$ has internal degree $2(p^2 + p - 2)$, which is not divisible by $2(p^3-1)$. Meanwhile, again when $n=3$, the element $h_{10}h_{21}$ is in both the first-subscript complex and the critical complex.

Setting $\epsilon$ to zero, formula \eqref{lambda diff formula} recovers Ravenel's formula for the dual of the Lie bracket in $L(n,n)$ defined in \cref{The reduction to Lie algebra cohomology}, originally from \cite{ravenel1977cohomology} (see also \cite[Theorem 6.3.8]{MR860042}). Hence the differential graded algebra $\Lambda_{n,0}^{\bullet}$ is the Chevalley-Eilenberg complex of Ravenel's Lie model $L(n,n)$. The internal grading on $\Lambda_{n,0}^{\bullet}$ is chosen so that it agres with the internal grading in Ravenel's spectral sequence \eqref{rmss}, hence also the internal grading on the cohomology of the Morava stabilizer group. 

We may also set $\epsilon$ to nonzero values in the ground ring $R$. Considering $\Lambda_{n,\epsilon}^{\bullet}$ as a function in the variable $\epsilon$, we may think of $\epsilon \mapsto \Lambda_{n,\epsilon}^{\bullet}$ as a one-parameter family of derived deformations of Ravenel's model for the Morava stabilizer group, parameterized by $\mathbb{A}^1_R$. 

This family of derived deformations can also be approached from a non-derived perspective. 
Each of the differential graded algebras $\Lambda_{n,\epsilon}^{\bullet}$ is a finite-dimensional exterior algebra generated by $1$-cochains. Hence, by Chevalley-Eilenberg \cite{MR0024908}, $\Lambda_{n,\epsilon}^{\bullet}$ is naturally isomorphic to the Chevalley-Eilenberg DGA of some Lie algebra, which we will call $L_{\epsilon}(n)$. Consequently the parameterized family of DGAs $\epsilon \mapsto \Lambda_{n,\epsilon}^{\bullet}$ is the fiberwise Chevalley-Eilenberg DGA of the parameterized family of Lie algebras $\epsilon \mapsto L_{\epsilon}(n)$ on $\mathbb{A}^1_R$.

Nevertheless, it is the derived point of view that is more useful for our purposes in this paper. By restricting each fiber $\Lambda_{n,\epsilon}^{\bullet}$ to its critical complex, we get a parameterized family of DGAs on $\mathbb{A}^1_R$ which is smooth away from $\epsilon=0$. It is this parameterized family of DGAs that we will ultimately need to understand, since it is our task in this section to prove that the $\epsilon=0$ critical complex $cc_{n,0}^{\bullet} = cc^{\bullet}(L_0(n))\cong cc^{\bullet}(L(n,n))$ is quasi-isomorphic to the $\epsilon=1$ critical complex $cc_{n,1}^{\bullet} = cc^{\bullet}(L_1(n))$. In other words, 
\begin{itemize}
\item even though the family of DGAs $\epsilon \mapsto \Lambda_{n,\epsilon}^{\bullet}$ is singular at $\epsilon=0$, 
\item and the cohomology of $\Lambda_{n,0}^{\bullet}$ is {\em not} isomorphic to the cohomology of $\Lambda_{n,\epsilon}^{\bullet}$ for $\epsilon=0$,
\item if we restrict to the family of smaller DGAs $\epsilon \mapsto cc_{n,\epsilon}^{\bullet}$ inside $\Lambda_{n,\epsilon}^{\bullet}$, we {\em do} get an isomorphism $H^*(cc_{n,\epsilon}^{\bullet}) \cong H^*(cc_{n,0}^{\bullet})$ for all $\epsilon$.
\end{itemize}
While the parameterized family of DGAs $\epsilon\mapsto \Lambda_{n,\epsilon}^{\bullet}$ is the fiberwise Chevalley-Eilenberg DGAs of the parameterized family of Lie algebras $\epsilon\mapsto L_{\epsilon}(n)$, it is not the case that the parameterized family of DGAs $\epsilon\mapsto cc^{\bullet}_{n,\epsilon}$ is the fiberwise Chevalley-Eilenberg DGA of any parameterized family of Lie algebras. This is because any Chevalley-Eilenberg DGA is generated in cohomological degree $1$, but for $n>1$, $cc^{\bullet}_{n,\epsilon}$ is not generated in cohomological degree $1$. For example, while $h_{1,n-1}h_{n-1,0}\in cc^2_{n,\epsilon}\subseteq \Lambda_{n,\epsilon}^2$ is a product of $1$-cochains in $\Lambda_{n,\epsilon}^{\bullet}$, the $1$-cochains $h_{1,n-1}$ and $h_{n-1,0}$ in $\Lambda_{n,\epsilon}^{\bullet}$ are not in $cc^{\bullet}_{n,\epsilon}$ for $n>1$, and $h_{1,n-1}h_{n-1,0}$ is an indecomposable $2$-cochain in $cc^{\bullet}_{n,\epsilon}$ for all $n>1$.

As already noted, the fiber $L_0(n)$ of our parameterized family of Lie algebras at $\epsilon=0$ is Ravenel's model for the height $n$ Morava stabilizer group. We mentioned the fiber $L_1(n)$ at $\epsilon=1$ in the previous paragraph. The reason the fiber at $1$ is significant is that the Lie algebra $L_1(n)$ is isomorphic to the Lie algebra $\mathfrak{gl}_n(R)$ of $n$-by-$n$ matrices in $R$ with commutator bracket as the Lie bracket, by Proposition \ref{def of Adot}. For general $\epsilon\neq 0$, the Lie algebra $L_{\epsilon}(n)$ is still (nontrivially) isomorphic to $\mathfrak{gl}_n(R)$. We prove this below, in Corollary \ref{isos exist cor}, by constructing a differential graded {\em connection} on the bundle of DGAs $\epsilon\mapsto \Lambda_{\epsilon,n}^{\bullet}$ restricted to the punctured affine line $\mathbb{A}^1_R\backslash \{0\}$. Parallel transport in this connection furnishes DGA isomorphisms $\Lambda_{\epsilon,n}^{\bullet} \cong \Lambda_{1,n}^{\bullet}$ for all $\epsilon\neq 0$, and consequently Lie algebra isomorphisms $L_{\epsilon}(n) \cong L_{1}(n)\cong \mathfrak{gl}_n(R)$ for all $\epsilon \neq 0$, as well as monodromy representations which play a pivotal role in the main theorems in this paper.


\begin{definition}\label{def of Lambda eps 2}
Let $R,n,p,\epsilon$ be as in Definition \ref{def of Lambda eps}.
\begin{itemize}
\item
We refer to the $R$-linear basis $\{ h_{i,j}: i,j\in Z_n\}$ for $\Lambda^1_{n,\epsilon}$ as the {\em $h$-basis} for $\Lambda^1_{n,\epsilon}$. More generally, we refer to the $R$-linear basis $\{ h_{i_1,j_1}\wedge \dots \wedge h_{i_{\ell},j_{\ell}}\}$ for $\Lambda^{\ell}_{n,\epsilon}$ as the {\em $h$-basis} for $\Lambda^{\ell}_{n,\epsilon}$. 
\item We write $L_{\epsilon}(n)$ for the Lie $R$-algebra such that $\Lambda^{\bullet}_{n,\epsilon}$ is the Chevalley-Eilenberg complex of $L_{\epsilon}(n)$. 
\item
We write $\sigma$ for the differential graded $R$-algebra automorphism of $\Lambda^{\bullet}_{n,\epsilon}$ given by $\sigma(h_{i,j}) = h_{i,j+1}$. 
This automorphism $\sigma$ generates an action of the cyclic group $C_n$ on $\Lambda^{\bullet}_{n,\epsilon}$ given by cyclically permuting the second subscript in the $h$-basis. We write ${}_\sigma \Lambda_{n,\epsilon}^{\bullet}$ for the {\em $C_n$-equivariant} differential graded $R$-algebra given by $\Lambda_{n,\epsilon}^{\bullet}$ with $C_n$ acting by $\sigma$. 
\item
When $R\rightarrow S$ is a homomorphism of commutative rings and $\phi$ is an element of order $n$ in the group of automorphisms $\Aut_{R-alg}(S)$ of  $S$ which fix $R$, we write $\tilde{\sigma}$ for the $\phi$-semilinear action of $\sigma$ on $\Lambda^{\bullet}_{n,\epsilon}\otimes_R S$. That is, 
\begin{itemize}
\item $\tilde{\sigma}$ is $R$-linear (but not $S$-linear, unless $\phi=\id$), 
\item the action of $\tilde{\sigma}$ on $\Lambda^{\bullet}_{n,\epsilon}\subseteq \Lambda^{\bullet}_{n,\epsilon}\otimes_R S$ coincides with the action of $\sigma$ on $\Lambda^{\bullet}_{n,\epsilon}$,
\item and $\tilde{\sigma}(xy) = \phi(x)\tilde{\sigma}(y)$ for any $x\in S$ and any $y\in \Lambda^{\bullet}_{n,\epsilon}\otimes_R S$.\end{itemize}
\end{itemize}
We write ${}_{\tilde{\sigma}} \Lambda_{n,\epsilon}^{\bullet}$ for the $C_n$-equivariant differential graded $S$-algebra given by $\Lambda_{n,\epsilon}^{\bullet} \otimes_R S$ with $C_n$ acting by $\tilde{\sigma}$. 
\end{definition}
The most important case of $\tilde{\sigma}$ is the case where $R$ is the finite field $\mathbb{F}_p$, and $S$ is its extension field $\mathbb{F}_{p^n}$, and $\phi$ is the Frobenius endomorphism of $\mathbb{F}_{p^n}$. 

\begin{example}\label{deformed Lie bracket special case} If we let $n = 2$ in Definition \ref{def of Lambda eps}, then the Lie $R$-algebra $L_{\epsilon}(2)$ has underlying $R$-module $M_2(R)$, the $R$-module of $2 \times 2$ matrices with entries in $R$, and Lie bracket 
\begin{align*}
\left[ \begin{bmatrix}
    a & b \\
    c & d
\end{bmatrix}, \begin{bmatrix}
    e & f \\
    g & h
\end{bmatrix} \right]_{\epsilon} := \begin{bmatrix}
    bg - cf & \epsilon(af + bh - be - df) \\
    \epsilon(ce + dg - ag - ch) & cf - bg
\end{bmatrix},
\end{align*}
i.e., the commutator bracket but with the off-diagonal entries scaled by a factor of $\epsilon$.

When $\epsilon = 1$, $L_{\epsilon}(2)$ coincides with the classical $\mathfrak{gl}_2(R)$. Regarding $L_{\epsilon}(2)$ as a function of $\epsilon$, the function $\epsilon \mapsto L_{\epsilon}(2)$ is a family of Lie $R$-algebras parametrized over the affine line $\mathbb{A}_R^1$, singular only at the point $\epsilon = 0$. The fiber over the singular point is the solvable Lie algebra $L(2,2)$ of Ravenel \cite{ravenel1977cohomology},\cite{MR860042}, while the fiber over each point $\epsilon \neq 0$ is the reductive Lie algebra $L_{\epsilon}(2)$, isomorphic to $\mathfrak{gl}_2(R)$.   
\end{example}

\begin{prop}\label{critical complex is a subset of first-subscript complex} Let $n$, $R$ be as in Definition \ref{crit complex def-prop} and let $p$ be a prime such that $p > 2n^2$. Then the critical complex is a subset of the first-subscript complex.
\end{prop}
\begin{proof}   
We introduce the following notation: given a nonzero element $h_{i_1,j_1}h_{i_2,j_2} \cdots h_{i_m,j_m}$ of $\Lambda^{\bullet}_{n}$, its associated {\em angle-bracket symbol} is an $n$-tuple of integers, written $\langle a_1, a_2, \dots ,a_n\rangle$, uniquely determined by the following rules:
\begin{itemize}
\item The angle-bracket symbol of $h_{i,j}$ is $\langle a_1, a_2, \dots ,a_n\rangle$, with $a_j=-1$ and $a_{i+j}=1$, and all other $a_k$ entries equal to zero. Here the subscript $a_k$ is to be understood as being defined modulo $n$, so that, for example, $a_{2+(n-1)} = a_1$.
\item The angle-bracket symbol of a product $x\cdot y$ is equal to the angle-bracket symbol of $x$ plus the angle-bracket symbol of $y$.
\end{itemize}
So, for example, if $n=3$, then the angle-bracket symbol of $h_{10}$ is $\langle -1,1,0\rangle$, the angle-bracket symbol of $h_{21}$ is $\langle 1,-1,0\rangle$, and the angle-bracket symbol of $h_{10}h_{21}$ is $\langle 0,0,0\rangle$.   

Our proof of the claim follows from several observations about angle-bracket symbols:
\begin{enumerate}
\item The internal degree of an element $x = h_{i_1,j_1}h_{i_2,j_2} \cdots h_{i_m,j_m}$, reduced modulo $2(p^n-1)$, depends only on the angle-bracket symbol of $x$. This is because
\begin{itemize}
\item the angle-bracket symbol of $h_{i,j}$ has a single entry which is $-1$, and this entry is in the $j$th coordinate,
\item the angle-bracket symbol of $h_{i,j}$ has a single entrry which is $1$, and this entry is in the $(i+j)$th coordinate, 
\item the internal degree of $h_{i,j}$ is $2p^j(p^i-1)$, i.e., the internal degree of $h_{i,j}$ is determinable modulo $2(p^n-1)$ from the angle-bracket symbol of $h_{i,j}$,
\item and both the angle-bracket symbol and the internal degree are additive on products of elements of the form $h_{i,j}$.
\end{itemize}
\item In particular, if the angle-bracket symbol of $x$ is $\langle 0,0,0, \dots ,0\rangle$, then $x$ is in the critical complex.
\item The converse is also true if $p>>n$, and in particular, if $p>2n^2$. This requires some explanation. Since $\Lambda^{\bullet}_n$ is an exterior algebra, in a nonzero element $h_{i_1,j_1}h_{i_2,j_2} \cdots h_{i_m,j_m}$ the pairs $(i_1,j_1), (i_2, j_2), \dots ,(i_m,j_m)$ must be pairwise distinct. Consequently each entry in the angle-bracket symbol of $h_{i_1,j_1}h_{i_2,j_2} \cdots h_{i_m,j_m}$ is contained in the closed interval $[-m,m]$. 

Since $\Lambda^{\bullet}_{n}$ is $n^2$-dimensional, we have the bound $n^2\geq m$. The assumption that $p>2n^2$ together with the bound $n^2\geq m$ entails the inequality $p>2m$. The inequality $p>2m$ ensures that each residue class modulo $p^n-1$ has {\em at most one} expression as a linear combination of the powers $p^0,p^1, \dots ,p^{n-1}$ of $p$ in which the coefficients lie in the interval $[-m,m]$. Consequently, if $h_{i_1,j_1}h_{i_2,j_2} \cdots h_{i_m,j_m}$ is in the critical complex, then its angle-bracket symbol must be zero.
\item The sum of the first subscripts $i_1 + i_2 + \dots + i_n$ can be recovered, modulo $n$, from the angle-bracket symbol of $h_{i_1,j_1}h_{i_2,j_2} \cdots h_{i_m,j_m}$. Given the angle-bracket symbol $\langle a_1, \dots ,a_{n}\rangle$, one simply takes the linear combination $\sum_{k=1}^{n}k\cdot a_k$ and reduces it modulo $n$. To see that the resulting residue class agrees with the residue class of $i_1 + i_2 + \dots + i_n$ modulo $n$, one can carry out an easy induction:
\begin{itemize}
\item check that the two agree when the angle-bracket symbol in question is the angle-bracket symbol of a single element $h_{i,j}$,
\item then observe that both are additive on products of elements of the form $h_{i,j}$.
\end{itemize}
\end{enumerate}
With the above observations in hand, the proof is very easy: given an element $h_{i_1,j_1}h_{i_2,j_2} \cdots h_{i_m,j_m}$ of the critical complex, its angle-bracket symbol must be zero. Consequently $\sum_{k=1}^m i_k$ must be zero modulo $n$, i.e., $h_{i_1,j_1}h_{i_2,j_2} \cdots h_{i_m,j_m}$ is in the first-subscript complex.
\end{proof}
\begin{remark}\label{remark on prime bound}
The bound $p>2n^2$ in Proposition \ref{critical complex is a subset of first-subscript complex} is not sharp: for example, the conclusion of Proposition \ref{critical complex is a subset of first-subscript complex} is true for $p=2$ and $1\leq n\leq 3$, even though the inequality $p>2n^2$ is not satisfied in those cases. However, the bound cannot be removed entirely, or the resulting claim is false. For example, if $p=2$ and $n=4$, the element $h_{20}h_{21}h_{30}h_{31}$ is in the critical complex but not in the first-subscript complex. 

By Proposition \ref{critical complex is a subset of first-subscript complex} together with brute-force computer calculation to handle values of $p$ which are $\leq 2n^2$, one gets the following:
\begin{itemize}
\item $cc^{\bullet}_n\subseteq FSC^{\bullet}_n$ for heights $n\leq 3$ and for all $p$.
\item $cc^{\bullet}_4\subseteq FSC^{\bullet}_4$ iff $p>2$.
\item $cc^{\bullet}_5\subseteq FSC^{\bullet}_5$ iff $p>3$. A counterexample to the containment when $p=2$ is $h_{24}h_{14}h_{42}$. A counterexample when $p=3$ is $h_{34}h_{43}h_{11}h_{14}$.
\end{itemize}
\end{remark}

\begin{definition}\label{def of h-diagonal}
Given elements $\delta,\epsilon\in R$, we say that a homomorphism of differential graded $R$-algebras $f: \Lambda^{\bullet}_{n,\epsilon} \rightarrow \Lambda^{\bullet}_{n,\delta}$ is {\em $h$-diagonal} if the $h$-basis diagonalizes the action of $f$ on the $1$-cochains $\Lambda^{1}_{n,\epsilon} \rightarrow \Lambda^{1}_{n,\delta}$. That is, $f$ is $h$-diagonal if and only if, for each $i$ and $j$, $f(h_{i,j})$ is an $R$-linear multiple of $h_{i,j}$.
\end{definition}


\begin{prop}\label{h-diag isos}
Let $\delta,\epsilon$ be nonzero elements of a field $k$. Let $k$ be the ground ring $R$ in the definition of $\Lambda^{\bullet}_{n,\epsilon}$ in Definition \ref{def of Lambda eps}. 
\begin{itemize}
\item 
Then the set of $h$-diagonal isomorphisms 
\begin{align}
\label{h-diag iso 2308}
\Lambda^{\bullet}_{n,\epsilon} &\rightarrow \Lambda^{\bullet}_{n,\delta}\end{align} is in bijection with the set of zeroes of the polynomial $\frac{\delta}{\epsilon} - \prod_{i\in Z_n}x_i\in k[x_1, \dots ,x_n]$ in $k^n$.
\item
Furthermore, the set of $\sigma$-equivariant $h$-diagonal isomorphisms $\Lambda^{\bullet}_{n,\epsilon} \rightarrow \Lambda^{\bullet}_{n,\delta}$ is in bijection with the set of zeroes of the polynomial $\frac{\delta}{\epsilon} - x^n\in k[x]$ in $k$. 
\item
Suppose furthermore that $K$ is a field extension of $k$, and suppose that $K$ is perfect and that the $q$-Frobenius map 
\begin{align*}
 \phi: K&\rightarrow K \\
 \phi(x)&= x^q 
\end{align*}
 is an element of order $n$ in $\Gal(K/k)$.
Then the set of $\tilde{\sigma}$-equivariant $h$-diagonal isomorphisms $\Lambda^{\bullet}_{n,\epsilon}\otimes_k K \rightarrow \Lambda^{\bullet}_{n,\delta}\otimes_k K$ is in bijection with the set of zeroes of the polynomial $\frac{\delta}{\epsilon} - x^{(q^n - 1)/(q-1)}\in k[x]$ in $k$. 
\end{itemize}
\end{prop}
\begin{proof}
An $h$-diagonal morphism $\Lambda^{\bullet}_{n,\epsilon} \rightarrow \Lambda^{\bullet}_{n,\delta}$ is determined by specifying, for each $h_{i,j}$, the scalar $\alpha_{i,j}\in k$ such that $f(h_{i,j}) = \alpha_{i,j}h_{i,j}$. $f$ is a homomorphism between exterior algebras over $k$, each generated in degree $1$, it is clear that $f$ is a homomorphism of graded $k$-algebras, no matter what choices are made for $\alpha_{i,j}$. To verify that $f$ is a morphism of DGAs, we need only verify that $f$ commutes with the differential. This is a straightforward calculation, writing $d_{\epsilon},d_{\delta}$ for the differentials on $\Lambda^{\bullet}_{n,\epsilon}$ and on $\Lambda^{\bullet}_{n,\delta}$, respectively:
\begin{align*}
 \sum_{\ell =1}^{i-1} \alpha_{\ell,j}\alpha_{i-\ell,j+\ell}h_{\ell,j}h_{i-\ell,j+\ell} \ \ \ \ \ \ \  \ \ \ \ \ \ \  & \\ \ \ \ \ \ \ \  + \epsilon\sum_{\ell=i}^n \alpha_{\ell,j}\alpha_{i-\ell+n,j+\ell} h_{\ell,j}h_{i-\ell+n,j+\ell} 
  &= f\left( d_{\epsilon}(h_{i,j})\right) \\ 
  &= d_{\delta}\left( f(h_{i,j})\right) \\
  &= \alpha_{i,j}\sum_{\ell =1}^{i-1} h_{\ell,j}h_{i-\ell,j+\ell} + \alpha_{i,j}\delta\sum_{\ell=i}^n h_{\ell,j}h_{i-\ell+n,j+\ell} ,
\end{align*}
yielding the equations
\begin{align}
\label{eq fp34n} \alpha_{\ell,j}\alpha_{i-\ell,j+\ell} &= \alpha_{i,j} \mbox{\ if\ } 1\leq \ell <i ,\mbox{\ \ and}\\
\label{eq fp34o} \epsilon\alpha_{\ell,j}\alpha_{i-\ell,j+\ell} &= \delta \alpha_{i,j} \mbox{\ if\ } i\leq \ell \leq n \mbox{\ and\ } i\neq n.
\end{align}
Using equation \eqref{eq fp34n}, we get
\begin{itemize}
\item $\alpha_{2,j} = \alpha_{1,j}\alpha_{1,j+1}$ for all $j$,
\item $\alpha_{3,j} = \alpha_{1,j}\alpha_{2,j+1} = \alpha_{1,j}\alpha_{1,j+1}\alpha_{1,j+2}$, for all $j$,
\item and so on, arriving by induction at $\alpha_{i,j} = \alpha_{1,j} \dots \alpha_{1,j+i-1}$ for all $i,j$.
\end{itemize}
Consequently all the parameters $\alpha_{i,j}$ are determined uniquely by the parameters $\alpha_{1,1}, \dots ,\alpha_{1,n}$. Write $x_j$ for $\alpha_{1,j}$. The only restriction on the parameters $x_1, \dots ,x_n$ is imposed by equation \eqref{eq fp34o}, which requires (in the case $i=1=\ell$) that 
\begin{align*}
 \alpha_{1,j} 
  &= \frac{\epsilon}{\delta} \alpha_{1,j}\alpha_{n,j+1} \\
  &= \frac{\epsilon}{\delta} \alpha_{1,j}\alpha_{1,j+1} \dots \alpha_{1,j+n} \\
  &= \frac{\epsilon}{\delta} \alpha_{1,j}^2\alpha_{1,j+1} \dots \alpha_{1,j+n-1} ,
\end{align*}
hence $\alpha_{1,j} = \frac{\epsilon}{\delta} \alpha_{1,j} x_1 \dots x_n$ and the product $x_1\dots x_n$ is equal to $\frac{\delta}{\epsilon}$ (where $\alpha_{1,j} \neq 0$ since $\alpha_{1,j} = 0$ would imply that $h_{1,j}$ is in the kernel of $f$, disqualifying $f$ from being an $h$-diagonal isomorphism), as claimed.

The action of $\sigma$ cyclically permutes the second subscript on the $h$-basis. Consequently $f$ is $\sigma$-equivariant if and only if $\alpha_{i,j} = \alpha_{i,j+1}$ for all $i,j$, i.e., $x_j = x_{j+1}$ for all $j$. So one gives a $\sigma$-equivariant $h$-diagonal isomorphism \eqref{h-diag iso 2308} simply by making a choice of element $x_1\in k$ such that $x_1^n = \frac{\delta}{\epsilon}$.

Similarly, since the action of $\tilde{\sigma}$ sends $\alpha h_{i,j}$ to $\alpha^q h_{i,j+1}$, the map $f$ is $\tilde{\sigma}$-equivariant if and only if $\alpha_{i,j}^q = \alpha_{i,j+1}$ for all $i,j$, i.e., $x_j^q = x_{j+1}$ for all $j$. So one gives a $\tilde{\sigma}$-equivariant $h$-diagonal isomorphism \eqref{h-diag iso 2308} simply by making a choice of element $x_1\in k$ such that $x_1\cdot x_1^q \cdot \dots \cdot x_1^{q^{n-1}} = \frac{\delta}{\epsilon}$.
\end{proof}

\begin{corollary}\label{isos exist cor}
Let $\delta,\epsilon$ be units in the commutative ring $R$. 
\begin{itemize}
\item Suppose that $R$ contains an $n$th root of $\delta/\epsilon$. Then the $C_n$-equivariant differential graded $R$-algebras ${}_\sigma \Lambda^{\bullet}_{n,\epsilon}$ and ${}_\sigma \Lambda^{\bullet}_{n,\delta}$ are isomorphic to one another.
\item Suppose that $R$ contains a $\frac{q^{n}-1}{q-1}$th root of $\delta/\epsilon$. Then the $C_n$-equivariant differential graded $S$-algebras ${}_{\tilde{\sigma}} \Lambda^{\bullet}_{n,\epsilon}$ and ${}_{\tilde{\sigma}} \Lambda^{\bullet}_{n,\delta}$ are isomorphic to one another.
\end{itemize}
\end{corollary}  

Suppose the commutative ring $R$ is a field. 
We can think of $\epsilon\mapsto \Lambda^{\bullet}_{n,\epsilon}$ as a bundle of differential bigraded $R$-algebras over the punctured affine line $\mathbb{A}_R^1\setminus\{0\} = \Spec R[x^{\pm 1}]$. As a vector bundle, $\epsilon\mapsto \Lambda^{\bullet}_{n,\epsilon}$ is the trivial bundle of rank $n^2$. However, as we move from fiber to fiber, the differential changes, so as a bundle of DGAs, $\epsilon\mapsto \Lambda^{\bullet}_{n,\epsilon}$ is nontrivial. 
A nice way to think of the bundle $\epsilon\mapsto \Lambda^{\bullet}_{n,\epsilon}$ of differential bigraded $R$-algebras is as a single differential bigraded $R[x]$-algebra $\Lambda_n^{\bullet}$:

\begin{definition}\label{def of lambda punc}
Let $R$ be a commutative ring, let $p$ be a prime, and let $n$ be a positive integer. We write $\Lambda^{\bullet}_{n}$ for the differential bigraded $R[x]$-algebra given by the exterior $R[x]$-algebra
\begin{align*} \Lambda^{\bullet}_{n} &= \Lambda_{R[x]}(h_{i,j}: i,j\in Z_n),\end{align*}
with each generator $h_{i,j}$ in cohomological degree $1$ and internal degree $2(p^i-1)p^j$, and with differential
\begin{align*}
 d(h_{i,j}) &= \sum_{\ell =1}^{i-1} h_{\ell,j}h_{i-\ell,j+\ell} + x\sum_{\ell=i}^n h_{\ell,j}h_{i-\ell+n,j+\ell}. \end{align*}
Again, there is an evident action of the cyclic group $C_n$ on $\Lambda^{\bullet}_{n}$ given by $\sigma(h_{i,j}) = h_{i,j+1}$.

We will write $\Lambda^{\bullet}_{n,punc}$ for the localization $\Lambda^{\bullet}_n[x^{-1}]$, as a differential bigraded $R[x^{\pm 1}]$-algebra.
\end{definition}
 By reducing $\Lambda_n^{\bullet}$ modulo the ideal $(x - \delta)$, we get the fiber $\Lambda_{n,\delta}^{\bullet}$ of the family of DGAs $\epsilon\mapsto \Lambda_{n,\epsilon}^{\bullet}$ over the point $\delta\in \mathbb{A}^1_R$. 

The purpose of the localization in Definition \ref{def of lambda punc}, when we define $\Lambda_{n,punc}^{\bullet}$, is simply to remove the singular fiber, i.e., the fiber at $\epsilon = 0$.

If one treats $x$ as a parameter, then $cc_{n,\epsilon}^{\bullet}$ may be filtered by powers of $x$, yielding a decreasing filtration 
\[ cc_{n,\epsilon}^{\bullet} \supseteq (x) \cdot cc_{n,\epsilon}^{\bullet} \supseteq (x^2) \cdot cc_{n,\epsilon}^{\bullet} \supseteq \cdots.\] Similarly, $FSC^{\bullet}$ may be filtered by powers of $x$, yielding a decreasing filtration \[ FSC^{\bullet} \supseteq (x) \cdot FSC^{\bullet} \supseteq (x^2) \cdot FSC^{\bullet} \supseteq \cdots.\] As a consequence of Proposition \ref{critical complex is a subset of first-subscript complex}, the former is a filtered cochain subcomplex of the latter, as long as $p>2n^2$.
The map of spectral sequences induced by this map of filtered cochain complexes will play a pivotal role in the proof of Theorem \ref{ss collapse thm 304}.


Let $k$ be a field that contains an $n$th root of $\delta/\epsilon$. By Proposition \ref{h-diag isos} and Corollary \ref{isos exist cor}, for any pair of points $\delta,\epsilon$ in the punctured affine line $\mathbb{A}^1_k - \{0\}$, we have a choice of isomorphism of the fiber of $\Lambda^{\bullet}_{n,punc}$ over $\delta$ with the fiber of $\Lambda^{\bullet}_{n,punc}$ over $\epsilon$. In fact, we have one choice of such an isomorphism for each $n$th root of $\delta/\epsilon$. This suggests the possibility of a {\em connection} on the bundle $\Lambda^{\bullet}_{n,punc}$, with locally constant parallel transport (that is, small perturbations to the ``path'' between two fibers gives us the same parallel transport isomorphism), whose parallel transport isomorphisms are the isomorphisms constructed in Proposition \ref{h-diag isos}, and such that the various choices of parallel transport isomorphism between the fibers correspond to different homotopy classes of paths, in $\mathbb{A}^1_k - \{0\}$, between the points $\epsilon$ and $\delta$. 

When $k=\mathbb{C}$, the previous sentence can be taken literally: we can regard $\mathbb{A}^1_{\mathbb{C}} - \{ 0\}$ as a punctured plane with the standard Euclidean topology, and then ``paths'' are paths in the familiar topological sense. Of course, when $k$ is some other field, e.g. a finite field, it takes some technology to make sense of what one ought to mean by a ``path'' in $\mathbb{A}^1_{k} - \{ 0\}$. The Tannakian formalism\footnote{We thank L. Candelori for telling us about the Tannakian approach to monodromy in positive characteristic.} offers one reasonable approach to this, as in \cite{MR1883387}. However, in the present situation, bringing in the Tannakian formalism is not really necessary: we can avoid the weight of the general theory by working concretely with the explicit examples of interest. There are only two ``flavors'' of connection on $\Lambda_{n,punc}^{\bullet}$ that we will consider in this paper: {\em Artin-Schreier-type} connections, and {\em Kummer-type} connections. For these two types of connections, we calculate explicit parallel transport isomorphisms. 


Below, in \cref{Parallel transport.}, we review some of the rudiments of the theory of connections and parallel transport, presenting the ideas in a way that is amenable to application to Artin-Schreier-type and Kummer-type connections (applications which we carry out in \cref{Parallel transport...Artin-Schreier} and in \cref{Parallel Kummer analytic at 0...}), and particularly to the bundle of differential graded $k$-algebras $\Lambda_n^{\bullet}$ (which we carry out in \cref{Differential graded parallel transport}).

\subsection{Basic ideas on differential graded parallel transport}
\label{Parallel transport.}

\subsubsection{Review of classical parallel transport}

Here are some notes on very classical, well-known uses for connections.

Classically, given a smooth manifold $M$ and a smooth real vector bundle $V$ on $M$, here are two equivalent definitions of a connection on $V$:
\begin{itemize}
\item In works which are written from a perspective closer to algebraic geometry, a connection on $V$ is most often defined as a $\mathbb{R}$-linear function \begin{align*} \nabla: \Gamma(V) &\rightarrow \Gamma(T^*M\otimes V)\end{align*} satisfying the Leibniz rule $\nabla(fs) = df\otimes s + f\nabla(s)$ for all smooth functions $f$ on $M$ and all smooth sections $s$ of $V$. 
\item In works which are written from a perspective closer to differential geometry, a connection on $V$ is most often defined as an $\mathbb{R}$-linear function \begin{align*} \nabla: \Gamma( TM\otimes V) &\rightarrow \Gamma(V)\end{align*} satisfying the Leibniz rule $\nabla_v(fs) = d_v(f)s + f\nabla_vs$ for all smooth functions $f$ on $M$ and smooth sections $s$, where $d_v(f)$ is the directional derivative of $f$ in the direction of the tangent vector $v$.
\end{itemize}
Given a connection $\nabla$ on $V$ and a smooth path $\gamma: [0,1]\rightarrow M$, for each choice of $v\in V_{\gamma(0)}$, there exists a unique section $s_v$ of $V$ along $\gamma$ satisfying $\nabla_{\dot{\gamma}(t)}(s_v)=0$ for all $t\in [0,1]$ and satisfying $s_v(\gamma(0)) = v$. This is a consequence of the existence and uniqueness theorem  (of Picard-Lindel\"{o}f) of solutions to first-order ODEs. The function $v\mapsto s_v(\gamma(1)): V_{\gamma(0)}\rightarrow V_{\gamma(1)}$ is then an $\mathbb{R}$-linear isomorphism, called the {\em parallel transport along $\gamma$} associated to $\nabla$.

One can try to implement these ideas more generally, replacing $\mathbb{R}$ with a field of characteristic zero\footnote{There are serious difficulties with solving ODEs over a field of positive characteristic: consider the problem of solving the ODE $\frac{dy}{dx} = x^{p-1}$ over a field of characteristic $p$, for example. In this paper we do not try to state the greatest level of generality for when $k$-linear connections can be ``integrated'' to yield parallel transport isomorphisms, since for our purposes, it suffices to know that certain Kummer connections with rational Kummer parameter can be integrated, and in \cref{Parallel Kummer analytic at 0...} we are able to give sufficient conditions for this integrability.}, and replacing smooth maps with algebraic maps. We work through the calculation of the parallel transport isomorphisms for the Artin-Schreier and Kummer connections, below, in sections \ref{Parallel Kummer analytic at 0...} and \ref{Parallel transport...Artin-Schreier}.

First, we settle on the following versions of the basic definitions:
\begin{definition}\label{def of connection}
Suppose we are given a homomorphism of commutative rings $R\rightarrow S$. 
\begin{itemize}
\item Given an $S$-module $M$, an {\em $R$-linear connection on $M$} is an $R$-module homomorphism $\nabla: M \rightarrow M\otimes_S \Omega^1_{S/R}$ such that $\nabla(ms) = (\nabla m)s + m \otimes ds$ for all $m\in M$ and all $s\in S$.
\item We often think of a pair $(M,\nabla)$, with $M$ an $S$-module and $\nabla$ an $R$-linear connection on $M$, as a single object, an {\em $S$-module with $R$-linear connection}. Usually $R$ is taken to be an obvious ground field. 

Given two modules with connection $(M,\nabla)$ and $(M^{\prime},\nabla^{\prime})$, {\em a morphism of modules with connection} is an $S$-module morphism $M \rightarrow M^{\prime}$ which commutes with the connections. This definition yields a category $\Mod_{FG}\Conn_R(S)$ of finitely-generated $S$-modules with $R$-linear connection.
\item Given a $S$-algebra $A$, a {\em multiplicative $R$-linear connection on $A$} is a connection $\nabla: A \rightarrow A\otimes_S \Omega^1_{S/R}$ on the underlying $S$-module of $A$, satisfying the additional condition that $\nabla$ is a derivation, i.e., $\nabla(a_1a_2) = (\nabla a_1)a_2 + a_1(\nabla a_2)$ for all $a_1,a_2\in A$. 
\end{itemize}
\end{definition}

It is classical that, given a connection $\nabla: M \rightarrow M\otimes_S \Omega^1_{S/R}$ on a free $S$-module $M$, there is a unique extension of $\nabla$ to a multiplicative connection $\Lambda_S(M) \rightarrow \Lambda_S(M) \otimes_S \Omega^1_{S/R}$ on the exterior $S$-algebra $\Lambda_S(M)$ of $M$. Consequently, to give a grading-preserving multiplicative connection $\Lambda^{\bullet}_{n} \rightarrow \Lambda^{\bullet}_{n} \otimes_{R[x]} \Omega^1_{R[x]/R}$ is the same as to give a connection $\Lambda^{1}_{n} \rightarrow \Lambda^1_{n} \otimes_{R[x]} \Omega^1_{R[x]/R}$ on the $R[x]$-module of $1$-cochains $\Lambda^1_n$ of $\Lambda_n^{\bullet}$, i.e., a connection on the free $R[x]$-module spanned by $\{ h_{i,j}: i,j\in Z_n\}$. Similarly, to give a grading-preserving multiplicative connection $\Lambda^{\bullet}_{n,punc} \rightarrow \Lambda^{\bullet}_{n,punc} \otimes_{R[x^{\pm 1}]} \Omega^1_{R[x^{\pm 1}]/R}$ is the same as to give a connection $\Lambda^{1}_{n,punc} \rightarrow \Lambda^1_{n,punc} \otimes_{R[x^{\pm 1}]} \Omega^1_{R[x^{\pm 1}]/R}$.

\subsubsection{Parallel transport in a Kummer connection}  
\label{Parallel Kummer analytic at 0...}

\begin{definition}\label{def of kummer conn}
Let $R$ be a commutative ring, and let $a\in R$.
\begin{itemize}
\item
The {\em $R$-linear Kummer module with parameter $a$}, written $\mathcal{K}_a$, is an $R[x^{\pm 1}]$-module with $R$-linear connection defined as follows. The underlying $R[x^{\pm 1}]$-module of $\mathcal{K}_a$ is the free $R[x^{\pm 1}]$-module on one generator, for which we will write $h$. The connection $\nabla: R[x^{\pm 1}] \rightarrow R[x^{\pm 1}]\otimes_{R[x^{\pm 1}]} \Omega^1_{R[x^{\pm 1}]/R}$ is the $R$-linear map that sends $1$ to $a\otimes x^{-1} dx$ and satisfies the Leibniz rule $\nabla(sf) = (\nabla s)f + s(df)$ for all $f\in R[x^{\pm 1}]$.
\item
A {\em Kummer-type connection} (respectively, {\em finite Kummer-type connection}) is a module with connection which is isomorphic to a direct sum (respectively, a finite direct sum) of modules with Kummer connection. 
\item If an $R[x^{\pm 1}]$-module $M$ with connection is equipped with an $R[x^{\pm 1}]$-linear basis $B$ such that, for each $b\in B$,  $\nabla(b)$ is a scalar times $b\otimes x^{-1}dx$, then we say that $B$ is a {\em Kummer basis for $M$.} 
\item
If $M$ is a $R[x^{\pm 1}]$-algebra equipped with a Kummer-type connection $\nabla$, then a {\em multiplicative Kummer basis for $M$} is a Kummer basis $B$ for $(M,\nabla)$ such that $1 \in B$, if $b_1,b_2 \in B$, then $b_1 \cdot b_2 \in B$, and the Kummer parameter of $b_1 \cdot b_2$ is equal to the sum of the Kummer parameters of $b_1$ and $b_2$. 
\item 
If $M$ is a differential graded $R[x^{\pm 1}]$-algebra equipped with a Kummer-type differential graded multiplicative connection $\nabla$ in the sense of Definition \ref{def of dg mult conn}, then a {\em DG Kummer basis for $M$} is a multiplicative Kummer basis $B$ for $M$ such that, if $b \in B$, then $d(b)$ is a $k[x]$-linear combination of elements of $B$ whose Kummer parameter is equal to that of $b$. 
\item
We say that a Kummer-type connection {\em has rational parameters} (respectively, {\em integral parameters}) if it is isomorphic to a direct sum of modules with Kummer connections whose parameters are rational numbers (respectively, integers).
\end{itemize}
\end{definition}

We will make significant use of the fact that the collection of Kummer-type connections is closed under taking exterior powers: 
\begin{prop}\label{exterior powers of kummer conns}
Let $R$ be a commutative ring, let $a_1, \dots ,a_n\in R$, and let $\mathcal{K}$ denote the Kummer-type connection
\[ \mathcal{K} = \mathcal{K}_{a_1} \oplus \dots \oplus \mathcal{K}_{a_n}.\]
Let $m$ be an integer. Then the $m$th exterior power of $\mathcal{K}$ is isomorphic to the Kummer-type connection
\[ \Lambda^m(\mathcal{K})\cong \bigoplus_{S} \mathcal{K}_{\sum_{i\in  S} a_i},\]
where the direct sum is taken over all subsets $S$ of $\{ 1, \dots ,n\}$ with exactly $m$ elements.

Consequently the natural multiplicative connection on the exterior algebra of $\mathcal{K}$ (as in the discussion immediately following Definition \ref{def of connection}) satisfies
\[ \Lambda(\mathcal{K}) \cong \bigoplus_{S\subseteq \{ 1, \dots ,n\}} \mathcal{K}_{\sum_{i\in  S} a_i},\]
where now the direct sum is taken over all subsets $S$ of $\{ 1, \dots ,n\}$.
\end{prop}
\begin{proof}
Elementary.
\end{proof}

\begin{prop}\label{kummer modules iso 1}
Let $R$ be a commutative ring, let $r,s$ be rational numbers, and suppose that $r-s$ is an integer. Then we have an isomorphism of $R$-linear Kummer modules with connection $\mathcal{K}_r\cong\mathcal{K}_s$.
\end{prop}
\begin{proof}
Let $f$ be the $R[x^{\pm 1}]$-linear homomorphism $R[x^{\pm 1}] \rightarrow R[x^{\pm 1}]$ which sends $1$ to $x^{r-s}$. It is routine to verify that $f$ makes the diagram
\[\xymatrix{
 R[x^{\pm 1}] \ar[r]^(.3){\nabla_r} \ar[d]^f & R[x^{\pm 1}]\otimes_{R[x^{\pm 1}]}\Omega^1_{R[x^{\pm 1}]/R}\ar[d]^{f\otimes \Omega^1_{R[x^{\pm 1}]/R}} \\
 R[x^{\pm 1}] \ar[r]^(.3){\nabla_s} & R[x^{\pm 1}]\otimes_{R[x^{\pm 1}]}\Omega^1_{R[x^{\pm 1}]/R}
}\]
commute, where $\nabla_r,\nabla_s$ are the connections on $\mathcal{K}_r$ and on $\mathcal{K}_s$, respectively. It is straightforward to see that $f$ is bijective as well. 
\end{proof}

A classical calculation over the complex numbers shows that, given the Kummer connection $\mathcal{K}_a$, a complex number $v$, and a nonzero complex number $\epsilon$, a differentiable function $f(x)$ is a solution to the ODE
\begin{align} \label{cauchy-euler ode 1}  \frac{df}{dx} + \frac{a}{x} f &= 0 
\end{align}
if and only if, for a closed path $\gamma: [0,1]\rightarrow \mathbb{A}^1_{\mathbb{C}}\backslash \{ 0\}$, the parallel transport isomorphism from $\mathbb{C}\cong (\mathcal{K}_a)_{\gamma(0)}$ to $\mathbb{C}\cong (\mathcal{K}_a)_{\gamma(1)}$ is the $\mathbb{C}$-linear map that sends $f(\gamma(0))$ to $f(\gamma(1))$. 
An elementary calculation yields that the solutions $f$ to \eqref{cauchy-euler ode 1} are scalar multiples of $x^{-a}$. When the Kummer parameter $a$ is an integer, the situation is quite simple: for each $\epsilon,\delta\in\mathbb{C}^{\times}$, the parallel transport map $(\mathcal{K}_a)_{\epsilon}\rightarrow (\mathcal{K}_a)_{\delta}$ is simply multiplication by $(\epsilon/\delta)^a$. 

On the other hand, when $a$ is a rational number but not an integer, one still gets that the parallel transport isomorphism from the fiber of $\mathcal{K}_a$ at $\epsilon$ to the fiber of $\mathcal{K}_a$ at $\delta$ is multiplication by $(\epsilon/\delta)^{a}$, but notice that the function $x\mapsto x^{a}$ is not quite well-defined: it depends on a choice of branch cut for the $s$th root function on $\mathbb{C} - \{0\}$, where $s$ is the denominator of the (reduced) rational number $a$. To put it another way, when $a$ is a non-integer rational number, the Kummer connection with parameter $a$ fails to have {\em globally} path-independent parallel transport: it matters {\em which homotopy class} of path in $\mathbb{C} - \{0\}$ from $\epsilon$ to $\delta$ we choose, and specifically, how many times the path winds around the singularity at $0$. 

Consequently we get a nontrivial monodromy operator on the tangent space to $\mathbb{A}^1_{\mathbb{C}}$ at a nonzero complex number: in the case $a = r/s$ with $r$ coprime to $s$, the monodromy operator $T: \mathbb{C}\rightarrow\mathbb{C}$ cyclically permutes the $s$th roots of unity in $\mathbb{C}$. In the particular case $a=1/s$, fix a primitive $s$th root of unity $\omega$ in $\mathbb{C}$, and then $T$ is given on each fiber $k$ of the Kummer module with connection $\mathcal{K}_{1/s}$ by the multiplication-by-$\omega$ map $k\rightarrow k$. 

Everything discussed so far in this subsection is very well-known. We have included this material because our next task is to explain how {\em the same constructions can be made over a general field whose characteristic does not divide the denominator of a rational Kummer parameter.} 

Suppose $a$ is the rational number $r/s$, and suppose that $k$ is a field whose characteristic does not divide $s$. Suppose that $\epsilon$ and $\delta$ are nonzero elements of $k$, and suppose that $k$ has $s$ distinct $s$th roots of $\epsilon/\delta$. Then, for each of those $s$th roots $\zeta$ of $\epsilon/\delta$, we have an isomorphism 
\begin{align}
\nonumber \mathcal{K}_a/(x-\epsilon)\cong k &\rightarrow k\cong \mathcal{K}_a/(x-\delta)  \\
\label{map 11203} v&\mapsto v\cdot \zeta^r
\end{align}  
from the fiber of the Kummer module $\mathcal{K}_a$ at $\epsilon$ to the fiber of the Kummer module $\mathcal{K}_a$ at $\delta$. We think of \eqref{map 11203} as ``parallel transport, from the fiber at $\epsilon$ to the fiber at $\delta$, along any path in the homotopy class of $\zeta$.'' 

In particular, in the case that $\epsilon = \delta$, after fixing a choice of primitive $s$th root of unity $\omega$, we have a Picard-Lefschetz operator
\begin{align}
\nonumber \mathcal{K}_{r/s}/(x-\epsilon)\cong k &\stackrel{T}{\longrightarrow} k\cong \mathcal{K}_{r/s}/(x-\epsilon)  \\
\label{map 11203a} T(v)&= v\cdot \omega^r,
\end{align} 
i.e., $T$ is the monodromy operator on the smooth fiber of $\mathcal{K}_{r/s}$.

This is quite figurative. For most fields $k$, it is not clear what we ought to mean by a ``path'' in $\mathbb{A}^1_k$. This is particularly true in characteristic $p$, where it is not generally even true that the maps \eqref{map 11203} for various $\zeta$ are all the solutions of the differential equation \eqref{cauchy-euler ode 1}. Differential equations over fields of characteristic $p>0$ have many ``spurious'' solutions (i.e., solutions with no corresponding solution in characteristic zero) that arise from the vanishing of $\frac{d}{dx} (x^p)$. We take the maps \eqref{map 11203} to be the {\em definition} of parallel transport in a Kummer connection over a general field $k$ whose characteristic does not divide the denominator of the Kummer parameter. 

The map \eqref{map 11203} depends only on the choice of the root $\zeta$ of $\epsilon/\delta$. Consequently we will refer to $\zeta$ as a {\em path} from $\epsilon$ to $\delta$. Consequently, a root of unity is a path from $\epsilon$ to $\epsilon$, i.e., a {\em loop} based at $\epsilon$. This terminology is justified by a strong analogy with the behavior of parallel transport over $\mathbb{C}$ using a flat connection: just as the set of parallel transport isomorphisms \eqref{map 11203} (i.e., the set of $s$th roots of $\epsilon/\delta$) is a torsor for the group of $s$th roots of unity in $k$, the set of homotopy classes of paths from $\epsilon$ and $\delta$ in $\mathbb{A}^1_k - \{0\}$ morally ought to be a torsor for the fundamental group of $\mathbb{A}^1_k - \{0\}$.

\subsubsection{Parallel transport in an Artin-Schreier connection}
\label{Parallel transport...Artin-Schreier}

\begin{definition}
Given a commutative ring $R$ and an element $a\in R$, the {\em $R$-linear Artin-Schreier module over $\mathbb{A}^1_R$ with parameter $a$}, written $\mathcal{L}_a$, is an $R[x]$-module with $R$-linear connection defined as follows. The underlying $R[x]$-module of $\mathcal{L}_a$ is the free $R[x]$-module on one generator. We will write $h$ for this generator. The connection $\nabla: R[x] \rightarrow R[x]\otimes_{R[x]} \Omega^1_{R[x]/R}$ on $\mathcal{L}_a$ is the unique $R$-linear map that sends $h$ to $-ah\otimes dx$ and satisfies the Leibniz rule $\nabla(sf) = (\nabla s)f + s \otimes df$ for all $f\in R[x]$.

An $R[x]$-module with connection is {\em Artin-Schreier-type} if it is a direct sum of Artin-Schreier modules with connection.
\end{definition}

\begin{prop} \label{kummer modules iso 2}
Let $R$ be a commutative ring, and let $r$ be an integer. Then the $R$-linear Kummer module $\mathcal{K}_r$ on $\mathbb{A}^1_R$ is isomorphic to the $R$-linear Artin-Schreier module $\mathcal{L}_0$ restricted to $\mathbb{A}^1_R\backslash \{0\}$.
\end{prop}
\begin{proof}
By Proposition \ref{kummer modules iso 1}, $\mathcal{K}_r$ is isomorphic to $\mathcal{K}_0$. The connection on $\mathcal{K}_0$ and the connection on $\mathcal{L}_0\otimes_{R[x]}R[x^{\pm 1}]$ are each determined by their behavior on a generator for $R[x^{\pm 1}]$, and each vanishes on that generator.
\end{proof}

By a calculation analogous to that carried out for Kummer connections in \cref{Parallel Kummer analytic at 0...}, over $\mathbb{C}$ the parallel transport isomorphisms for the Artin-Scheier connection $\mathcal{L}_a$ are the solutions to the ordinary differential equation $\frac{df}{dx} - af = 0$. Hence the parallel transport isomorphism along any path $\gamma$ from the fiber of $\mathcal{L}_a$ at $\epsilon$ to the fiber of $\mathcal{L}_a$ at $\delta$ sends $v$ to $v\cdot e^{a\cdot (\delta-\epsilon)}$. 


In the case $k = \mathbb{C}$, the Artin-Schreier connection has 
globally path-independent parallel transport, since it is defined on all of $\mathbb{A}^1_{\mathbb{C}}$.
Hence it does not matter which path $\gamma$ we choose, from one point $\epsilon\in \mathbb{A}^1_{\mathbb{C}}$ to another point $\delta\in \mathbb{A}^1_{\mathbb{C}}$: the parallel transport isomorphism depends only on the endpoints of the path. So the calculation of the parallel transport isomorphisms for $\mathcal{L}_a$ made in this section, for one particular choice of path $\gamma$ from $\epsilon$ to $\delta$, are in fact complete calculations of all the parallel transport isomorphisms for $\mathcal{L}_a$.

The above does not generalize to an arbitrary field $k$: consider what $e^{a\cdot (\delta - \epsilon)}$ ought to mean in the case $k = \overline{\mathbb{Q}}$, for example. But it {\em does} make sense for arbitrary fields $k$ if the Artin-Scheier parameter $a$ is zero. In this paper, we will only ever need to work with Kummer-type connections with rational parameters, and Artin-Schreier-type connections with parameter zero. We explained in \cref{Parallel Kummer analytic at 0...} that the former have well-defined parallel transport isomorphisms as long as the characteristic of the field does not divide the denominator of the Kummer parameter. In this subsection, we saw that the latter have well-defined parallel transport isomorphisms over arbitrary fields. Consequently, in the rest of this paper, we will be able to prove some theorems about the relevant cases of parallel transport over fields of positive characteristic, even without having a {\em general} theory of parallel transport over such fields.

\subsubsection{Functoriality of parallel transport}


We have the category $\Proj\Mod_{FG}\Conn_{\mathbb{C}}(\mathbb{C}[x^{\pm 1}])$ of finitely generated projective $\mathbb{C}[x^{\pm 1}]$-modules equipped with a $\mathbb{C}$-linear connection. For each object $(M,\Delta)$ of $\Proj\Mod_{FG}\Conn_{\mathbb{C}}(\mathbb{C}[x^{\pm 1}])$ and each path $\gamma: [0,1]\rightarrow \mathbb{C} - \{ 0\}$, we have the parallel transport isomorphism 
\begin{align*}
 PT_{\gamma}(M,\Delta): M/(x-\gamma(0)) &\stackrel{\cong}{\longrightarrow} M/(x-\gamma(1)) .
\end{align*}
The map $PT_{\gamma}(M,\Delta)$ is an isomorphism of $\mathbb{C}$-vector spaces, i.e., an object of the category $\Iso\Mod(\mathbb{C})$.
It is routine to check that
\begin{align*} PT_{\gamma}: \Proj\Mod_{FG}\Conn_{\mathbb{C}}(\mathbb{C}[x^{\pm 1}]) &\rightarrow \Iso\Mod(\mathbb{C}) \end{align*}
 is an additive functor, and moreover, it is lax monoidal, if we equip $\Proj\Mod_{FG}\Conn_{\mathbb{C}}(\mathbb{C}[x^{\pm 1}])$ with the following monoidal product:
\begin{definition}
Let $R\rightarrow S$ be a homomorphism of commutative rings.
Given $S$-modules $(M_1,\nabla_1)$ and $(M_2,\nabla_2)$ with $R$-linear connection, the {\em tensor product of $(M_1,\nabla_1)$ and $(M_2,\nabla_2)$} is the $S$-module with $R$-linear connection given by $M_1\otimes_S M_2$, with connection 
\begin{align*}
 M_1\otimes_S M_2 &\longrightarrow M_1\otimes_S M_2\otimes_S \Omega^1_{S/R} \\
 m_1\otimes m_2 &\mapsto \nabla_1(m_1)\otimes m_2 + m_1\otimes \nabla_2(m_2).
\end{align*}
With this tensor product operation, the category $\Proj\Mod_{FG}\Conn_R(S)$ is a symmetric monoidal category.
\end{definition}

The importance of the functoriality and additivity of $PT_{\gamma}$ is that these properties ensure that, if \[ \dots 
 \stackrel{d^{-1}}{\longrightarrow} (M^0,\nabla^0) 
 \stackrel{d^0}{\longrightarrow} (M^1,\nabla^1) 
 \stackrel{d^1}{\longrightarrow} (M^2,\nabla^2) 
 \stackrel{d^2}{\longrightarrow} \dots \]
is a sequence of homomorphisms of finitely generated projective modules with connection such that $d^n\circ d^{n-1} = 0$ for all $n$, then there are induced maps on fibers
\[\xymatrix{
 \vdots \ar[d]_{PT_{\gamma}(d^{-1})_{domain}} && 
  \vdots \ar[d]^{PT_{\gamma}(d^{-1})_{codomain}} \\
 M^0_{\gamma(0)} \ar[rr]^{PT_{\gamma}(M^0,\nabla^0)}\ar[d]_{PT_{\gamma}(d^0)_{domain}} &&
  M^0_{\gamma(1)}\ar[d]^{PT_{\gamma}(d^0)_{codomain}} \\
 M^1_{\gamma(0)} \ar[rr]^{PT_{\gamma}(M^1,\nabla^1)}\ar[d]_{PT_{\gamma}(d^1)_{domain}} &&
  M^1_{\gamma(1)}\ar[d]^{PT_{\gamma}(d^1)_{codomain}} \\
 M^2_{\gamma(0)} \ar[rr]^{PT_{\gamma}(M^2,\nabla^2)}\ar[d]_{PT_{\gamma}(d^2)_{domain}} &&
  M^2_{\gamma(1)}\ar[d]^{PT_{\gamma}(d^2)_{codomain}} \\
 \vdots && \vdots
}\]
such that the vertical composites 
\begin{align*} PT_{\gamma}(d^{n+1})_{domain} &\circ PT_{\gamma}(d^n)_{domain}\mbox{\ \ \ and}\\ PT_{\gamma}(d^{n+1})_{codomain} &\circ PT_{\gamma}(d^n)_{codomain}\end{align*} are both zero. In other words, {\em parallel transport using a differential graded connection yields an isomorphism of cochain complexes.}

The reason to bother with lax monoidality of $PT_{\gamma}$ is simply that lax monoidal functors send monoid objects to monoid objects. Monoid objects in $\Proj\Mod_{FG}\Conn_R(S)$ are multiplicative connections, as defined in Definition \ref{def of connection}. 
We also have a differential graded version:
\begin{definition}\label{def of dg mult conn}
Let $R\rightarrow S$ be a homomorphism of commutative rings, and let $A$ be a differential graded $S$-algebra whose underlying $S$-module is finitely generated and projective in each grading degree. By a {\em differential graded multiplicative connection on $A$} we mean a connection $\nabla^n: A^n\rightarrow A^n\otimes_S \Omega^1_{S/R}$ for each integer $n$ satisfying the conditions
\begin{itemize}
\item $\nabla^{n+1} \circ d = (d\otimes_S \id)\circ \nabla^n$ for all $n$, i.e., the sequence $\dots ,\nabla^{-1},\nabla^0,\nabla^1, \dots$ defines a connection on the underlying cochain complex of $A$,
\item and, if $a\in A^m$ and $b\in A^n$, then $\nabla^{m+n}(ab) = (\nabla^m a_1)a_2 + a_1(\nabla^n a_2)$, i.e., $\bigoplus_n\nabla^n$ is a multiplicative connection on the underlying graded $S$-module of $A$.
\end{itemize} 
\end{definition}

The upshot is:
\begin{observation}\label{structured parallel transport}\leavevmode
\begin{itemize}
\item If $(M,\nabla)$ is a module with multiplicative connection, then its parallel transport isomorphisms are {\em algebra} homomorphisms.
\item If $(M,\nabla)$ is a module with differential graded multiplicative connection, then its parallel transport isomorphisms are {\em differential graded} algebra homomorphisms.
\end{itemize}
\end{observation}

This formal, categorical approach to structured parallel transport generalizes to fields other than $\mathbb{C}$. Let $\mathcal{K}_a$ be a Kummer connection whose parameter $a$ is rational, with denominator $s$. 
Let $k$ be a field whose characteristic does not divide $s$, and which has $s$ distinct $s$th roots of each of its nonzero elements. As explained in \cref{Parallel Kummer analytic at 0...}, for each ``path'' $\gamma$ in $\mathbb{A}^1_k - \{ 0\}$---i.e., each pair of nonzero elements $\delta,\epsilon$ in $k$ and each choice of $s$th root of $\epsilon/\delta$---we have a parallel transport isomorphism from $M_{\epsilon}$ to $M_{\delta}$. Write $PT_{\gamma}(\mathcal{K}_a)$ for the parallel transport isomorphism $k[x^{\pm 1}]/(x-\epsilon) \rightarrow k[x^{\pm 1}]/(x-\delta)$. The construction of $PT_{\gamma}$ extends, in a basically obvious way, to direct sums of such Kummer connections, as well as copies of the Artin-Schreier module $\mathcal{L}_0$. The argument for additivity and monoidality of $PT_{\gamma}$ over $\mathbb{C}$ works in this setting as well, and Observation \ref{structured parallel transport} remains true.

\subsection{Differential graded parallel transport in the deformed Ravenel model}
\label{Differential graded parallel transport}

In Proposition \ref{h-diag isos} we saw that, for any field $k$ and any positive integer $n$ and any pair of nonzero elements $\delta,\epsilon\in k$ such that $\epsilon/\delta$ has an $n$th root in $k$, there exists a $\sigma$-equivariant isomorphism of differential graded $k$-algebras $\Lambda^{\bullet}_{n,\epsilon}\stackrel{\cong}{\longrightarrow} \Lambda^{\bullet}_{n,\delta}$ for each $n$th root of $\delta/\epsilon$. In Definition \ref{def of lambda punc}, we organized the differential graded $k$-algebras $\Lambda^{\bullet}_{n,\epsilon}$, for the various nonzero choices of $\epsilon\in k$, into a single bundle of DGAs over the punctured affine line.
Now we are in a position to answer the following question: 
\begin{question}\label{existence q}Does there exist a connection on $\Lambda^{\bullet}_{n,punc}$ whose associated parallel transport isomorphisms are the isomorphisms constructed in Proposition \ref{h-diag isos}?\end{question}
To be clear, it is easy to find a connection on $\Lambda^{\bullet}_{n,punc}$ whose associated parallel transport isomorphisms are $k$-linear isomorphisms, or even graded $k$-algebra isomorphisms. The trivial connection does the job, for example. But we are being much more restrictive by asking that the parallel transport isomorphisms are isomorphisms of DGAs, not just isomorphisms of graded $k$-algebras.

Question \ref{existence q} has a positive answer:
\begin{theorem}\label{main thm 1}
Let $n$ be a positive integer. Suppose that every nonzero element in the field $k$ has $n$ distinct $n$th roots. 
Recall that the $1$-cochains $\Lambda_{n,punc}^1$ in $\Lambda_{n,punc}^{\bullet}$ are $k[x^{\pm 1}]$-linearly spanned by the elements $\{ h_{i,j} : i,j\in Z_n\}$. Equip $\Lambda_{n,punc}^1$ with the connection $\nabla: \Lambda_{n,punc}^1 \rightarrow \Lambda_{n,punc}^1\otimes_{k[x^{\pm 1}]} \Omega^1_{k[x^{\pm 1}]/k}$ given by letting $\nabla(h_{i,j}) = -\frac{i}{n}h_{i,j}\otimes \frac{dx}{x}$.  Equip $\Lambda_{n,punc}^{\bullet}$ with the resulting multiplicative connection\footnote{See the discussion following Definition \ref{def of connection} for the fact that a multiplicative connection on $\Lambda_{n,punc}^{\bullet}$ is uniquely determined by a connection on its $1$-cochains.}.
Then the following statements are each true:
\begin{enumerate}
\item $\Lambda_{n,punc}^{\bullet}$ is Kummer-type with rational Kummer parameters. 
\item The parallel transport isomorphisms $\Lambda_{n,\epsilon}^{\bullet} \rightarrow \Lambda_{n,\delta}^{\bullet}$ induced by the connection on $\Lambda_{n,punc}^{\bullet}$ are the $\sigma$-equivariant isomorphisms constructed in Proposition \ref{h-diag isos}. In particular, for each $\epsilon,\delta\in \mathbb{A}^1_k - \{0\}$, the parallel transport isomorphism $\Lambda_{n,\epsilon}^{\bullet} \rightarrow \Lambda_{n,\delta}^{\bullet}$ associated to $\nabla$, along the path given by any $n$th root of $\epsilon/\delta$ in $k$, is a $\sigma$-equivariant isomorphism of differential graded $k$-algebras\footnote{In \cref{Parallel Kummer analytic at 0...} we explained that, in a Kummer-type connection whose Kummer parameters are divisors of $n$, the $n$th roots of $\epsilon/\delta$ are to be understood as the (homotopy classes of) ``paths'' from $\epsilon$ to $\delta$, along which one can carry out parallel transport.}.
\item Let $\omega\in k$ be a primitive $n$th root of unity\footnote{Again, see \cref{Parallel Kummer analytic at 0...} for the idea that $\omega$ behaves, morally speaking, like a generator for the fundamental group of $\mathbb{A}_1^k\backslash \{0\}$, based at $\epsilon$, reduced modulo $n$.} Let $T: \Lambda_{n,\epsilon}^{\bullet} \rightarrow \Lambda_{n,\epsilon}^{\bullet}$ be the monodromy operator arising from parallel transport along the path given by $\omega$ using the connection $\nabla$ on the bundle of DGAs $\Lambda_n^{\bullet}$. Then the action of $T$ on $\Lambda_{n,\epsilon}^{\bullet}$ is diagonalizable, with diagonal basis given by the $h$-basis. The element $h_{i_1,j_1}\wedge \dots \wedge h_{i_m,j_m}\in \Lambda_{n,\epsilon}^m$ is an eigenvector for the action of $T$ with eigenvalue $\omega^{-(i_1 + \dots + i_m)}$. 
\item Let $\epsilon\in k$ be nonzero, and let $(\Lambda_{n,\epsilon}^{\bullet})^T$ denote the fixed points of the action of $T$ on $\Lambda_{n,\epsilon}^{\bullet}$. Then $(\Lambda_{n,\epsilon}^{\bullet})^T$ is equal to the first-subscript complex in $\Lambda_{n,\epsilon}^{\bullet}$, defined in Definition \ref{crit complex def-prop}.
\end{enumerate}
\end{theorem}
\begin{proof}
Handling the claims in order:
\begin{enumerate}
\item This follows immediately from Proposition \ref{exterior powers of kummer conns}.
\item This follows from the calculation of the splitting of $\Lambda_{n,\epsilon}^{\bullet}$ as a direct sum of Kummer modules with connection, in Proposition \ref{exterior powers of kummer conns}, together with the calculation of the parallel transport isomorphisms associated to Kummer modules with connection in \cref{Parallel Kummer analytic at 0...}. 
\item This follows from the calculation, in \eqref{map 11203a} in \cref{Parallel Kummer analytic at 0...}, of the monodromy operator on the fibers of Kummer modules with connection.
\item This follows from the calculation of the splitting of $\Lambda_{n,\epsilon}^{\bullet}$ as a direct sum of Kummer modules with connection in Proposition \ref{exterior powers of kummer conns}.
\end{enumerate}
\end{proof}

Theorem \ref{main thm 1} will play an essential role in the proof of Theorem \ref{ss collapse thm 304}, but we will not use Theorem \ref{main thm 2} in the rest of the paper.
\begin{theorem}\label{main thm 2}
Let $p$ be a prime.
Let $n$ be a positive integer, and let $k$ be a perfect field of characteristic $\neq p$. Suppose that every nonzero element in the field $k$ has $\frac{p^n-1}{p-1}$ distinct $\frac{p^n-1}{p-1}$th roots. 
Let $\phi:k\rightarrow k$ be the Frobenius automorphism.
Consider the differential graded $k$-algebra $\Lambda_{n,punc}^{\bullet}$. Equip $\Lambda_{n,punc}^1$ with the connection $\nabla: \Lambda_{n,punc}^1 \rightarrow \Lambda_{n,punc}^1\otimes_{R[x^{\pm 1}]} \Omega^1_{R[x^{\pm 1}]/R}$ given by letting $\nabla(h_{i,j}) = -\frac{p^{i+j}-p^j}{p^n-1} h_{i,j}\otimes \frac{dx}{x}$. Equip $\Lambda_{n,punc}^{\bullet}$ with the resulting multiplicative connection.
Then the following statements are each true:
\begin{enumerate}
\item $\Lambda_{n,punc}^{\bullet}$ is Kummer-type with rational Kummer parameters.
\item The parallel transport isomorphisms $\Lambda_{n,\epsilon}^{\bullet} \rightarrow \Lambda_{n,\delta}^{\bullet}$ induced by the connection on $\Lambda_{n,punc}^{\bullet}$, are the $\tilde{\sigma}$-equivariant $k$-linear isomorphisms constructed in Proposition \ref{h-diag isos}. In particular, for each $\epsilon,\delta\in \mathbb{A}^1_{\mathbb{F}_p(\omega)} - \{0\}$ and each $\frac{p^n-1}{p-1}$th root $\zeta$ of $\epsilon/\delta$ in $k$, the parallel transport isomorphism
\[ \Lambda_{n,\epsilon}^{\bullet} \rightarrow \Lambda_{n,\delta}^{\bullet}\] associated to $\nabla$ along the path given by $\zeta$ is a $\tilde{\sigma}$-equivariant isomorphism of differential graded $k$-algebras.
\item Let $T: \Lambda_{n,\epsilon}^{\bullet} \rightarrow \Lambda_{n,\epsilon}^{\bullet}$ be the Picard-Lefschetz (i.e., monodromy) operator arising from a primitive $\frac{p^n-1}{p-1}$th root of unity in $k$.
Then the action of $T$ on $\Lambda_{n,\epsilon}^{\bullet}$ is diagonalizable, with diagonal basis given by the $h$-basis. The element $h_{i_1,j_1}\wedge \dots \wedge h_{i_m,j_m}\in \Lambda_{n,\epsilon}^m$ is an eigenvector for the action of $T$, with eigenvalue $\omega^{-i_1 - \dots - i_m}$.
\item Let $\epsilon\in k^{\times}$, and let $(\Lambda_{n,\epsilon}^{\bullet})^T$ denote the fixed points of the action of $T$ on $\Lambda_{n,\epsilon}^{\bullet}$. Then $(\Lambda_{n,\epsilon}^{\bullet})^T$ is equal to the critical complex in $\Lambda_{n,\epsilon}^{\bullet}$, defined in Definition \ref{crit complex def-prop}.
\end{enumerate}
\end{theorem}
\begin{proof}
Analogues of the methods used in Theorem \ref{main thm 1} work equally well here. 
\end{proof} 

\subsection{Medial layers and cores}
\label{Medial layers and cores.}

A Kummer module $\mathcal{K}_a$ is a free $k[x^{\pm 1}]$-module on a single generator, equipped with a particular connection. The generator of that free $k[x^{\pm 1}]$-module was called $h$ in Definition \ref{def of kummer conn}. We will have reason to consider the $k[x]$-submodule of $\mathcal{K}_a$ generated by $h$, so we give it a name:
\begin{definition}
The {\em medial layer} of a Kummer module $\mathcal{K}_a$ is the $k[x]$-submodule of $\mathcal{K}_a$ generated by $h$. 

More generally, given a Kummer-type module $M$ equipped with a Kummer basis $B$, its {\em medial layer} is the $k[x]$-submodule of $M$ generated by $B$. We write $\medial(M)$ for the medial layer of $M$.
\end{definition}
The medial layer of a Kummer module is usually not a module with connection, since the $k[x]$-submodule of $\mathcal{K}_a$ generated by $h$ is usually not closed under $\nabla$. The exception is the case $a = 0$, which indeed is a module with connection, isomorphic to the Artin-Schreier module $\mathcal{L}_0$.

We might consider how to enlarge the medial layer in order to get a $k[x]$-submodule of $\mathcal{K}_a$, or more generally any Kummer-type module, which {\em is} a submodule-with-connection. This is impossible if the Kummer parameter $a$ is not an integer, due to nontriviality of the monodromy action. However, for Kummer modules with integral parameters, such a construction is possible. We will call this construction the {\em core} of the Kummer module. 
If the Kummer parameters are integral and all nonnegative, then the core will contains the medial layer; on the other hand, if the Kummer parameters are integral and all nonpositive, then the medial layer contains the core. Because our motivating examples (from Theorem \ref{main thm 1}) all have nonpositive Kummer parameters, we will focus on the nonpositive case, but with a bit of care, analogous results to those in this section can also be proven in the case of nonnegative parameters.

A Kummer module, by definition, has smooth fibers but does not have a fiber at zero. The core is a canonical way to ``fill in'' such a fiber at zero, yielding a vector bundle with connection over $\mathbb{A}^1_k$, not only over $\mathbb{A}^1_k \backslash \{0\}$. Inverting $x$ (i.e., localizing away from zero) then recovers the original Kummer module. 
We think of taking the core of a Kummer module with integral parameters as an analogue of taking the ring of integers of a local number field, i.e., an operation that canonically ``fills in'' the missing information at a certain fiber (the special fiber, in the case of a local field; the fiber at zero, in the case of a Kummer connection).

\begin{definition-proposition}\label{def of core}
Let $k$ be a field. Suppose that $M$ is a free $k[x^{\pm 1}]$-module equipped with a Kummer-type connection with integral parameters. Choose a Kummer basis $B$ for $M$. For each $b\in B$, write $\alpha(b)$ for the Kummer parameter of $b$. By the {\em core of $M$}, written $\core(M)$, we mean the $k[x]$-submodule of $M$ generated by
\[ \left\{ x^{-\alpha(b)} b: b\in B\right\}.\]
By {\em projection into the core,} we mean the map of $k[x]$-modules which sends $b$ to $x^{-\alpha(b)}b$. 

It is routine to verify the following claims:
\begin{enumerate}
\item Using Propositions \ref{kummer modules iso 1} and \ref{kummer modules iso 2}, the connection $M \rightarrow M\otimes_{k[x^{\pm 1}]} \Omega^1_{k[x^{\pm 1}]/k}$ restricts to a connection $\core(M) \rightarrow \core(M)\otimes_{k[x]} \Omega^1_{k[x]/k}$. 
\item The core of $M$, as a module with connection, is independent of the choice of Kummer basis for $M$.
\item Inverting $x$, the module with connection $\core(M)\otimes_{k[x]}k[x^{\pm 1}]$ recovers $M$.
\item The core of $M\oplus N$ is naturally isomorphic to $\core(M)\oplus \core(N)$.
\item Projection into the core does {\em not} usually commute with the connections on each side. In particular, projection into the core defines a $k[x]$-module isomorphism from the medial layer of $M$ to the core of $M$, but this isomorphism is usually {\em not} an isomorphism of modules with connection.
\item The core of $M$ is contained in the medial layer of $M$ if and only if all Kummer parameters of $M$ are nonpositive.
\item If $M$ is a $k[x^{\pm 1}]$-algebra equipped with a multiplicative Kummer-type connection with a multiplicative Kummer basis, then the core of $M$ is a $k[x]$-subalgebra of $M$. 
\item If $M$ is a differential graded $k[x^{\pm 1}]$-algebra equipped with a Kummer-type differential graded multiplicative connection with a DG Kummer basis, then the core of $M$ is a differential graded $k[x]$-subalgebra of $M$. 
\end{enumerate}
\end{definition-proposition}

Since the core of $M$ is defined on all of $\mathbb{A}^1_k$, not only $\mathbb{A}^1_k\backslash \{ 0\}$, we may use the connection on $\core(M)$ to get a parallel transport isomorphism from any smooth fiber {\em all the way to the singular fiber}. It does not matter what path we choose, again since the connection is defined on all of $\mathbb{A}^1_k$:
\begin{prop}
Let $M,B,\alpha$ be as in Definition-Proposition \ref{def of core}. For each element $\delta\in k$, we have an equality between the fiber $\core(M)_{\delta}$ of $\core(M)$ at $\delta$ and the $k$-vector space with $k$-linear basis $\{ x^{-\alpha(b)}b : b\in B\}$. Using this equality, for every $\epsilon\in k$ the parallel transport isomorphism 
\begin{align*}
 k\{ x^{-\alpha(b)}b : b\in B\} = \core(M)_{\epsilon}
  &\rightarrow \core(M)_0 = k\{ x^{-\alpha(b)}b : b\in B\}
\end{align*} is the identity map on $k\{ x^{-\alpha(b)}b : b\in B\}$.
\end{prop} 
\begin{proof}
Since $M$ is of Kummer-type with integral parameters, $M$ is isomorphic to a direct sum of Kummer modules $\mathcal{K}_a$ for various integral parameters $a$, and has path-independent parallel transport. From Definition \ref{def of core} and equation \eqref{def of core}, we have that the core of $\mathcal{K}_a$ is the $k[x]$-submodule of $k[x^{\pm 1}]$ generated by $x^{-a}$, with connection satisfying $\nabla(x^{-a}) = 0$. That is, $\core(\mathcal{K}_a)$ is isomorphic to the Artin-Schreier module $\mathcal{L}_0$. By the calculation of parallel transport in an Artin-Schreier module in \cref{Parallel transport...Artin-Schreier}, the parallel transport isomorphism on $\mathcal{L}_0$, from any fiber to any other fiber, is the identity map.
\end{proof}

We remind the reader of the convention that, given a finitely generated projective module $M$ over a subset of $\mathbb{A}^1_k$, for each $\epsilon$ in that subset of $\mathbb{A}^1_k$, we write $M_{\epsilon}$ for the fiber $M\otimes_{k[x]} k[x]/(x-\epsilon)$ of $M$ at $\epsilon$. 

Given a cochain complex of free modules with connection, we have a spectral sequence whose input is a direct sum of copies of the cohomology of the smooth fiber, and whose output is the cohomology of the core:
\begin{definition-proposition}
Let $k$ be a field. Suppose that $C^{\bullet}$ is a cochain complex equipped with a Kummer-type differential graded connection with integral parameters. By the {\em core monodromy spectral sequence of $C^{\bullet}$} we mean the spectral sequence associated to the filtration on $\core(C^{\bullet})$ by powers of $x$. This spectral sequence has the form\footnote{The internal degrees $t$ in this spectral sequence are a little bit subtle. See Example \ref{monodromy and medial ss} for an example which we think does a good job of demonstrating how this internal degree behaves.}
\begin{align}
\label{monodromy ss} E_1^{s,t} 
  \cong H^s\left(  C^{\bullet}_{sm}\right)[x] 
  &\Rightarrow H^s\left( \core(C^{\bullet})\right) \\
\nonumber d_r: E_r^{s,t} &\rightarrow E_r^{s+1,t+r}
\end{align}
where $x\in E_1^{0,1}$, and where $C^{\bullet}_{sm}$ is any smooth fiber of $C^{\bullet}$.
If $C^{\bullet}$ is furthermore a differential graded algebra with differential graded multiplicative connection, then the core monodromy spectral sequence has a product satisfying the Leibniz rule.
\end{definition-proposition}
\begin{proof}
It is standard that the spectral sequence of a suitably filtered cochain complex converges to the cohomology of that complex, and its $E_1$-term is the cohomology of the associated graded cochain complex. In the case of the filtration of $\core(C^{\bullet})$ by powers of $x$, it follows from the freeness of $\core(C^{\bullet})$ as a $k[x]$-module that the associated graded cochain complex is isomorphic to $(\core(C^{\bullet})/x)[x]$, i.e., to $(\core(C^{\bullet}))_0[x]$. It follows from Proposition \ref{kummer modules iso 2} that parallel transport yields an isomorphism of $(\core(C^{\bullet}))_0$ with any smooth fiber $(\core(C^{\bullet}))_{\epsilon}$.
\end{proof}

By the same argument, we get a spectral sequence for the cohomology of the medial layer as well:
\begin{definition}
Let $k$ be a field. Suppose that $C^{\bullet}$ is a cochain complex equipped with a Kummer-type differential graded connection with integral Kummer parameters.
Suppose that the Kummer parameters of the connection on $C^{\bullet}$ are all nonpositive. 
By the {\em medial monodromy spectral sequence of $C^{\bullet}$} we mean the spectral sequence associated to the unique decreasing filtration on $\medial(C^{\bullet}) \supseteq \core(C^{\bullet})$ such that:
\begin{itemize}
\item multiplication by $x$ increases filtration by $1$, and
\item the filtration on $\medial(C^{\bullet})$, restricted to $\core(C^{\bullet})$, coincides with the $x$-adic filtration on $\core(C^{\bullet})$.
\end{itemize}
This spectral sequence again has the form
\begin{align}
\label{monodromy ss 2} E_1^{s,*} 
  \cong H^s\left( C^{\bullet}_{sm}\right)[x] 
  &\Rightarrow H^s\left( \medial(C^{\bullet})\right) \\
\nonumber d_r: E_r^{s,t} &\rightarrow E_r^{s+1,t+r}
\end{align}
where $x\in E_1^{0,1}$, and where a Kummer basis element $b\in C^{s}$ with Kummer parameter $\alpha(b)$ is in $E_1^{s,-\alpha(b)}$. 
Inclusion of the core into the medial layer induces a homomorphism of spectral sequences from \eqref{monodromy ss} to \eqref{monodromy ss 2}.

If $C^{\bullet}$ is furthermore a differential graded algebra with differential graded multiplicative connection, then the medial monodromy spectral sequence has a product satisfying the Leibniz rule.
\end{definition}
\begin{example}\label{monodromy and medial ss}
Here is a simple example to demonstrate the small (but nevertheless important and useful, e.g. in the proof of Theorem \ref{main local inv cycles thm}) difference between the core monodromy spectral sequence and the medial monodromy spectral sequence. Let $k$ be a field, and consider the differential graded $k[x]$-algebra $\Lambda^{\bullet}_1$ defined in Definition \ref{def of lambda punc}, with the differential graded multiplicative connection defined in Theorem \ref{main thm 2}. That is, $\Lambda^{\bullet}_1 = k[x,h_{10}]/h_{10}^2$, with zero differential, with $x$ in cohomological degree $0$, with $h_{10}$ in cohomological degree $1$, and with connection $\nabla(h_{10}) = -h_{10}\otimes \frac{dx}{x}$ on $\Lambda^{\bullet}_1\otimes_{k[x]}k[x^{\pm 1}]$. The following table depicts the medial layer of $\Lambda^{\bullet}_1$, the core of $\Lambda^{\bullet}_1$, and the filtrations on each which yield their respective monodromy spectral sequences:

\noindent\begin{tabular}{ccc}
 Filtration \hspace{7pt}
  & \hspace{7pt} $k$-linear basis for $\medial(\Lambda_1^{\bullet})$ \hspace{7pt}
  & \hspace{7pt} $k$-linear basis for $\core(\Lambda_1^{\bullet})$ \\
\hline
 -1
  & $h_{10}$
  & $ $\\
 0
  & $1,xh_{10}$
  & $1,xh_{10}$\\
 1
  & $x$
  & $x,x^2h_{10}$\\
 2
  & $x^2,xh_{10}$
  & $x^2,x^3h_{10}$\\
 3
  & $x^3,x^2h_{10}$
  & $x^3,x^4h_{10}$\\
 \vdots & \vdots & \vdots
\end{tabular}

The inclusion of the core into the medial layer sends each element in $\core(\Lambda_1^{\bullet})$ to the element with the same name in $\medial(\Lambda_1^{\bullet})$. Consequently the map from the core monodromy spectral sequence to the medial monodromy spectral sequence is an isomorphism on $E_1^{s,t}$ for all $t\geq 0$, but it is not surjective on $E_1^{1,-1}$. 
\end{example}

\subsection{A derived local invariant cycles theorem}
\label{The monodromy fixed-points...}

\begin{lemma}\label{monodromy and fibers}
Let $k$ be a field. Suppose that $M$ is a free $k[x^{\pm 1}]$-module with a Kummer-type connection. We write $B$ for a Kummer basis for $M$, and for each $b\in B$, we write $\alpha(b)$ for its Kummer parameter. Let $D$ be a positive integer, and suppose that each Kummer parameter $\alpha(b)$ is a rational number whose denominator is a divisor of $D$. Suppose that $k$ contains a primitive $D$th root of unity. 
Let $M^T$ denote the $k[x^{\pm 1}]$-submodule of $M$ spanned by the elements $b\in B$ such that $\alpha(b)$ is an integer. Then, for each nonzero $\epsilon\in k$, the fiber $(M^T)_{\epsilon}$ of $M^T$ coincides with the fixed-points $(M_{\epsilon})^T$ of the action of the monodromy operator $T$ on $M_{\epsilon}$.

\end{lemma}
\begin{proof} 
This follows from the calculation of the monodromy operator on fibers of Kummer modules in \cref{Parallel Kummer analytic at 0...}.
\end{proof}

Given a $k[x]$-module $M$ with an appropriate Kummer-type connection on $M\otimes_{k[x]}k[x^{\pm 1}]$, Lemma \ref{monodromy and fibers} tells us that taking the smooth fiber---i.e., the fiber at any nonzero point $\epsilon\in k$---commutes with taking the fixed-points of the Picard-Lefschetz operator $T$. 

It is also true that taking the singular fiber---i.e., the fiber at $0\in k$---commutes with taking the Picard-Lefschetz fixed points. However, there is something less obvious about that claim. At a glance, one perhaps would not expect $T$ to even be {\em defined} on the singular fiber: a Kummer-type connection is defined only on $\mathbb{A}^1_k \backslash \{0\}$, not on $\mathbb{A}^1_k$, so it does not yield a parallel transport isomorphism from the singular fiber to any other fiber, and in particular, it does not yield any parallel transport isomorphism from the singular fiber to itself. So the usual construction of $T$ on a smooth fiber (via parallel transport along a loop based at that fiber) cannot be carried out for the singular fiber.

Nevertheless, since the connection $\nabla$ was assumed to be Kummer-type, there is a natural way to define $T$ on the singular fiber. The calculation of $T$ on any smooth fiber of a Kummer-type module in \eqref{map 11203a} in \cref{Parallel Kummer analytic at 0...} makes sense {\em even on the singular fiber.} We take \eqref{map 11203a} as the definition of $T$ on the singular fiber, i.e., 
\begin{itemize}
\item we let $D$ be the least common multiple of the denominators of the Kummer parameters of $M$, 
\item we fix a primitive $D$th root of unity $\omega$ in $k$,
\item and, for each Kummer basis element $b$ with Kummer parameter $\alpha(b) = r/D$, we let $T(b) = b\cdot \omega^r$. 
\end{itemize}
Consequently $\nabla$ defines a $\mathbb{Z}/D\mathbb{Z}$-grading on $M$, by letting the degree of a Kummer basis element $b$ with Kummer parameter $r/D$ be $r$. 
Call this grading the {\em Kummer grading} on $M$. Then, on each fiber of $M$---whether smooth or singular---the fixed points of $T$ consists precisely of those elements in Kummer degree zero. We will write $M^T$ for the fixed points of the Picard-Lefschetz operator on $M$ itself. Consequently $M^T$ is a canonical Kummer-type $k[x]$-submodule of $M$ with integral Kummer parameters. This means that we may speak of the core of a Kummer-type module $M$ with {\em rational} integral parameters, not necessarily {\em integral} Kummer parameters: simply take the core of the Picard-Lefschetz fixed-points $M^T$.

We need one more definition before stating the main theorem of this section, Theorem \ref{main local inv cycles thm}.
\begin{definition}\label{def of core-homogeneity}
A $k[x]$-module $M$ with Kummer-type differential graded connection will be called {\em core-homogeneous} if the differential on $M$ restricts to a differential on $\core(M^T)$ which is strictly compatible\footnote{In \cite{MR0498551}, Deligne defines ``strict compatibility'' in the following way: a map $d$ is {\em strictly compatible} with a given filtration if the canonical map $\coim(d)\rightarrow\im(d)$ is an isomorphism of filtered objects. In our setting, the most practical way to verify this condition is usually to check that, given a Kummer basis element $b$, {\em every} term in the expansion of $d(b)$ as a linear combination of Kummer basis elements has the same filtration as $b$.} with the $x$-adic filtration.
\end{definition}

\begin{theorem}\label{main local inv cycles thm}
Let $k$ be an infinite field. Let $A^{\bullet}$ be a differential graded $k[x]$-algebra equipped with a differential graded multiplicative finite Kummer-type connection on $A^{\bullet}\otimes_{k[x]}k[x^{\pm 1}]$, with rational, nonpositive Kummer parameters. Suppose that $A^{\bullet}$ is core-homogeneous. Then the cohomology of the Picard-Lefschetz fixed-points on the singular fiber is isomorphic, as a graded vector space, to the cohomology of the Picard-Lefschetz fixed-points on any smooth fiber $A^{\bullet}_{sm}$:
\begin{equation}
\label{main iso 97088}  H^*\left(A^{\bullet}_{sm}\right)^T \cong H^*\left((A^{\bullet}_{sm})^T\right) \cong H^*\left((A^{\bullet}_{0})^T\right) \cong H^*\left(A^{\bullet}_0\right)^T.
\end{equation}
\end{theorem}
\begin{proof}
Since $A^{\bullet}\otimes_{k[x]}k[x^{\pm 1}]$ is of {\em finite} Kummer-type, it is free and of finite rank as a $k[x^{\pm 1}]$-module. Consequently $H^*(A^{\bullet})[x^{-1}] \cong H^*(A^{\bullet}\otimes_{k[x]}k[x^{\pm 1}])$ is a finitely-generated $k[x^{\pm 1}]$-module. In particular, there exist only finitely many elements $\epsilon\in k^{\times}$ such that the $k[x]$-module $H^*(A_{\bullet})$ has nontrivial $(x-\epsilon)$-torsion. Since our connection is a differential graded multiplicative connection, by Observation \ref{structured parallel transport} we have an isomorphism of DGAs between $A^{\bullet}_{\delta}$ and $A^{\bullet}_{\epsilon}$ for any nonzero $\delta,\epsilon\in k$, so we are free to choose some particular element $\delta\in k$ and to use $A^{\bullet}_{\delta}$ as our model for the smooth fiber $A^{\bullet}_{sm}$ of $A^{\bullet}$. Here is how we will choose such an element: since $k$ is infinite, there must exist some element $\delta\in k^{\times}$ such that $H^*(A_{\bullet})$ is $(x-\delta)$-torsion free. We will fix such an element $\delta\in k$, and we will use $A^{\bullet}_{\delta}$ as our model for the smooth fiber $A^{\bullet}_{sm}$ of $A^{\bullet}$.

The advantage of making this choice of $\delta$ is that, since $H^*\left( \medial\left((A^{\bullet})^T\right)\right)$ is $(x-\delta)$-torsion-free, the universal coefficient spectral sequence 
\begin{align*}
 E_2^{*,*} &\cong \Tor^{k[x^{\pm 1}]}_*\left( H^*\left( \medial\left((A^{\bullet}[x^{\pm 1}])^T\right)\right), k[x^{\pm 1}]/(x-\delta)\right) \\ & \Rightarrow H^*\left(\medial\left((A^{\bullet}[x^{\pm 1}])^T\right)\otimes_{k[x^{\pm 1}]} k[x^{\pm 1}]/(x-\delta) \right)
\end{align*}
collapses at $E_2$, yielding the isomorphism\footnote{In the statement of the theorem, we made a finiteness assumption on the Kummer-type connection, and we assumed that $k$ is infinite. These assumptions let us make the argument given in this paragraph and the previous paragraph, which yields the isomorphism \eqref{iso 97089}, i.e., yields that cohomology commutes with taking the smooth fiber. If one already has the isomorphism \eqref{iso 97089} by some other means, then the infiniteness assumption on the field $k$ can be dropped.}
\begin{align}
\label{iso 97089} H^*\left( \medial\left((A^{\bullet}[x^{-1}])^T\right)_{sm}\right) & \rightarrow H^*\left(\medial\left((A^{\bullet}[x^{-1}])^T\right)\right)\otimes_{k[x^{\pm 1}]} k[x^{\pm 1}]/(x-\delta).
\end{align}

From very general results about spectral sequences of filtered chain complexes (e.g. Proposition 1.3.2 of \cite{MR0498551}), core-homogeneity is equivalent to collapse of the core monodromy spectral sequence at $E_1$. The $E_1$-term of the core monodromy spectral sequence maps injectively into the $E_1$-term of the medial monodromy spectral sequence, and $x$ is an infinite cocycle (i.e., $d_r(x)=0$ for all $r$) in both spectral sequences.

Suppose, by contrapositive, that there is a nonzero differential in the medial monodromy spectral sequence. Choose the least $r$ such that there is a nonzero $d_r$-differential in that spectral sequence. Then the $E_r$-page of the medial monodromy spectral sequence is isomorphic to the $E_1$-page, hence is a free $k[x]$-module with the property that, for each bihomogeneous element $w$, there exists some integer $n$ such that $x^nw$ is in the image of the map from the core monodromy $E_r$-page to the medial monodromy $E_r$-page. Consequently, if we have a nonzero differential $d_r(y) = z$ in the medial monodromy spectral sequence for some $y,z$, then we may choose an integer $n$ large enough for $x^ny$ and $x^nz$ to be in the image of the map from the core monodromy $E_r$-page. Since $d_r(x^ny) = x^nz$, this contradicts the immediate collapse of the core monodromy spectral sequence. 

Hence the medial monodromy spectral sequence must collapse at $E_1$ as well, yielding isomorphism \eqref{iso 97090} in the chain of isomorphisms
\begin{align}
 \label{iso 97086} H^*\left(A^{\bullet}_{sm}\right)^T
 &\cong H^*\left((A^{\bullet}_{sm})^T\right) \\
 \nonumber  &\cong H^*\left(\left((A^{\bullet})^T\right)_{sm}\right) \\
\nonumber  &\cong H^*\left(\medial\left(((A^{\bullet}[x^{-1}])^T)\right)_{sm}\right) \\
\nonumber  &\cong H^*\left( \medial\left((A^{\bullet}[x^{-1}])^T\right)\right) \otimes_{k[x]}k[x]/(x-\delta) \\
\label{iso 97090}  &\cong H^*\left( \medial\left((A^{\bullet}[x^{-1}])^T\right)_0[x]\right) \otimes_{k[x]}k[x]/(x-\delta) \\
\nonumber  &\cong H^*\left( \medial\left((A^{\bullet}[x^{-1}])^T\right)_0\right)[x] \otimes_{k[x]}k[x]/(x-\delta) \\
\nonumber  &\cong H^*\left( \medial\left((A^{\bullet}[x^{-1}])^T\right)_0\right)\\
\nonumber  &\cong H^*\left( (A^{\bullet}_0)^T\right)\\
\label{iso 97091}  &\cong H^*\left( A^{\bullet}_0\right)^T.
\end{align}
Isomorphisms \eqref{iso 97086} and \eqref{iso 97091} are simply a consequence of the $T$-fixed-points being the degree $0$ summand of the Kummer grading on the cochain complex $A^{\bullet}$, so taking $T$-fixed-points commutes with taking cohomology.
\end{proof}

The following corollary of Theorem \ref{main local inv cycles thm} is curiously similar to the local invariant cycles theorem, and it is the reason we have chosen to refer to Theorem \ref{main local inv cycles thm} as a ``derived local invariant cycles theorem.''\footnote{Surely there are more general and more powerful theorems one could prove, which might also be recognizable as derived versions of the local invariant cycles theorem. We would be pleased to see further developments along such lines. We have chosen to develop the theory in the level of generality in this paper because we think the resulting theory is relatively concrete and accessible, and it is powerful enough to prove the main result, the calculation of the cohomology of the extended Morava stabilizer group with trivial coefficients in Theorem \ref{letteredthm 1}.}
\begin{corollary}\label{local inv cycles cor}
Let $k,A^{\bullet}$ be as in Theorem \ref{main local inv cycles thm}. Then the cohomology of the singular fiber surjects on to the Picard-Lefschetz fixed points of the cohomology of the smooth fiber:
\begin{align*}
 H^*\left((A^{\bullet})_{0}\right) &\twoheadrightarrow H^*\left((A^{\bullet})_{sm}\right)^T.
\end{align*}
\end{corollary}
\begin{proof}
Since $(A^{\bullet})^T$ is the degree zero summand in the Kummer grading on the cochain complex $A^{\bullet}$, the inclusion map $(A^{\bullet})^T \hookrightarrow A^{\bullet}$ has a canonical retraction $A^{\bullet}\rightarrow (A^{\bullet})^T$. Compose the isomorphism \eqref{main iso 97088} with the map induced in cohomology by the projection $A^{\bullet}_{0} \rightarrow (A^{\bullet}_{0})^T$.
\end{proof}

\subsection{Cohomology of the singular and smooth fibers in the deformed Ravenel model}
\label{Cohomology of the singular...}

Recall that, for a field $k$ of characteristic $p>0$ and a positive integer $n$, the deformed Ravenel model $\Lambda^{\bullet}_n$ is the differential graded $k[x]$-algebra defined in Definition \ref{def of lambda punc}. Recall also that, in Theorem \ref{main thm 2}, $\Lambda^{\bullet}_n \otimes_{k[x]}k[x^{\pm 1}]$ was equipped with the differential graded multiplicative finite Kummer-type connection $\nabla$ determined by the rule
\begin{align*}
 \nabla(h_{i,j}) &= -\frac{i}{n} h_{i,j}\otimes \frac{dx}{x},
\end{align*}
as long as the characteristic of $k$ does not divide $n$.
Finally, recall from Proposition \ref{def of Adot} that the fiber of $\Lambda^{\bullet}_n$ at $1$ is the Chevalley-Eilenberg DGA of the Lie algebra $\mathfrak{gl}_n(k)$, while the singular fiber of $\Lambda^{\bullet}_n$ is the Chevalley-Eilenberg DGA of Ravenel's model $L(n,n)$ for the height $n$ Morava stabilizer group. 

Let $CE^{\bullet}(\mathfrak{gl}_n(k))$ be the Chevalley-Eilenberg DGA of the Lie algebra $\mathfrak{gl}_n(k)$. Decreasingly filter $CE^{\bullet}(\mathfrak{gl}_n(k))$ by first subscript\footnote{To be clear, here it is important that the first subscript on the $h_{i,j}$ elements is an element of $Z_n = \{ 1, \dots ,n\}$ but {\em not} regarded only as a residue class in $\mathbb{Z}/n\mathbb{Z}$. For example, $h_{1,0}h_{n-1,1}$ is in first-subscript filtration $n$, {\em not} zero.}, i.e., $h_{i_1,j_1}\dots h_{i_m,j_m}$ is in filtration $i_1 + \dots + i_m$. By inspection of the formula \eqref{diff 0945} for the differential in $CE^{\bullet}(\mathfrak{gl}_n(k))$, the associated graded $E_0^{FS}CE^{\bullet}(\mathfrak{gl}_n(k))$ of the first-subscript filtration on $CE^{\bullet}(\mathfrak{gl}_n(k))$ is the $\epsilon=0$ case of the differential \eqref{lambda diff formula} on $\Lambda^{\bullet}_{n,\epsilon}$, i.e., $E_0^{FS}CE^{\bullet}(\mathfrak{gl}_n(k))$ is the Chevalley-Eilenberg complex $CE^{\bullet}(L(n,n))$. 

Consequently we have a spectral sequence
\begin{align}
\label{ss 37494} E_1^{s,t,u} \cong H^{s,t,u}(L(n,n);k) &\Rightarrow H^{s,u}(\mathfrak{gl}_n(k);k)\\
\nonumber d_r: E_r^{s,t,u} &\rightarrow E_r^{s+1,t-r,u},
\end{align}
where $s$ is the cohomological degree, $t$ the first-subscript filtration degree, and $u$ the internal degree. So, for example, $h_{i,j}$ is in tridegree $(1,i,2(p^i-1)p^j)$. 

The sub-DGA of $CE^{\bullet}(L(n,n))$ consisting of the elements in internal degrees divisible by $2(p^n-1)$ is precisely the critical complex $cc^{\bullet}(L(n,n))$ defined in \Cref{The reduction to Lie algebra cohomology}. Meanwhile, the sub-DGA of $CE^{\bullet}(\mathfrak{gl}_n(k))$ consisting of the elements in internal degrees divisible by $2(p^n-1)$ is precisely the critical complex $cc^{\bullet}(\mathfrak{gl}_k(n))$. Hence spectral sequence \eqref{ss 37494} splits as a direct sum of spectral sequences, in which one summand is of the form
\begin{align}
\label{ss 37495} E_1^{s,t,u} \cong H^{s,t,u}(cc^{\bullet}(L(n,n))) &\Rightarrow H^{s,u}(cc^{\bullet}(\mathfrak{gl}_n(k)))\\
\nonumber d_r: E_r^{s,t,u} &\rightarrow E_r^{s+1,t-r,u}.
\end{align}

\begin{theorem}\label{ss collapse thm 304}
Let $k$ be any field of characteristic $p$, and let $n$ be a positive integer such that $p>2n^2$. Then spectral sequence \eqref{ss 37495} collapses at the $E_1$-page.
\end{theorem}
Before we prove Theorem \ref{ss collapse thm 304}, we remark that the claim would {\em not} be true if we had not restricted our attention to the critical complex by considering only the summand \eqref{ss 37495} in spectral sequence \eqref{ss 37494}. If $n>1$, then spectral sequence \eqref{ss 37494} has very many nonzero differentials---but, by Theorem \ref{ss collapse thm 304}, none in internal degrees divisible by $2(p^n-1)$.   
\begin{proof}[Proof of Theorem \ref{ss collapse thm 304}]
Spectral sequence \eqref{ss 37495} over a given field extension $k$ agrees with the tensor product of $k$, over $\mathbb{F}_p$, with spectral sequence \eqref{ss 37495} over $\mathbb{F}_p$. If we show that spectral sequence \eqref{ss 37495} collapses with no nonzero differentials in the case $k = \overline{\mathbb{F}}_p$, then there can be no nonzero differentials in the case $k = \mathbb{F}_p$, and hence there can be no nonzero differentials for any extension of $\mathbb{F}_p$ either.

Consequently we need only prove the theorem in the case $k = \overline{\mathbb{F}}_p$.
Both $L(n,n)$ and $\mathfrak{gl}_n(k)$ are finite-dimensional Lie $k$-algebras. If there were a nonzero differential in spectral sequence \eqref{ss 37495}, then the $k$-linear dimension of $H^*(cc^{\bullet}(\mathfrak{gl}_n(k)))$ would be strictly smaller than that of $H^*(cc^{\bullet}(L(n,n)))$. 

We now argue that this does not happen. Apply Theorem \ref{main local inv cycles thm} in the case where $k = \overline{\mathbb{F}}_p$, where $A^{\bullet} = \Lambda^{\bullet}_n$, and where $A^{\bullet}\otimes_{k[x]} k[x^{\pm 1}]$ is equipped with the multiplicative differential-graded Kummer-type connection given by letting $\nabla(h_{i,j}) = -\frac{i}{n}h_{i,j}\otimes \frac{dx}{x}$. By Theorem \ref{main thm 1}, this is the connection in which the $h$-basis is a Kummer basis and the parallel transport isomorphisms are the $\sigma$-equivariant $h$-diagonal isomorphisms constructed in Proposition \ref{h-diag isos}. By Theorem \ref{main thm 1}, with this connection, the Picard-Lefschetz fixed point DGA $(\Lambda^{\bullet}_n)^T$ is the first-subscript complex $FSC^{\bullet}_n$, described in Definition \ref{crit complex def-prop}: $FSC^{\bullet}_n$ is the sub-DGA of $(\Lambda^{\bullet}_n)^T$ consisting of all $k$-linear combinations of products of elements $h_{i,j}$ whose first subscripts sum to a multiple of $n$. 

In this connection, the $h$-basis is a Kummer basis, and the Kummer parameter of $h_{i,j}$ is $-i/n$. Clearly the Kummer parameters are rational and nonpositive. The only hypothesis of Theorem \ref{main local inv cycles thm} left to be checked is that $\Lambda^{\bullet}_n$ is core-homogeneous. Write $\tilde{h}_{i,j}$ for the projection of $h_{i,j}$ into the core, as defined in Definition-Proposition \ref{def of core}. Then the core is generated, as a $k[x]$-module, by $k[x]$-linear combinations of products of elements of the form $\tilde{h}_{i,j}$. Hence, if we can show that every monomial in $d(\tilde{h}_{i,j})$ has the same $x$-adic filtration degree as $\tilde{h}_{i,j}$ itself---namely, zero---then $\Lambda^{\bullet}_n$ is core-homogeneous. 
We have equalities
\begin{align*}
    d(\tilde{h}_{i,j}) &= d(x^{\frac{i}{n}}h_{i,j}) \\
    &= x^{\frac{i}{n}}d(h_{i,j}) \\
    &= x^{\frac{i}{n}}\left( \sum_{\ell =1}^{i-1} x^{-\ell/n}\tilde{h}_{\ell,j}x^{(\ell-i)/n}\tilde{h}_{i-\ell,j+\ell} + x\sum_{\ell=i}^n x^{-\ell/n}\tilde{h}_{\ell,j}x^{(\ell-i-n)/n}\tilde{h}_{i-\ell+n,j+\ell}\right) \\
    &= \left( \sum_{\ell =1}^{i-1} \tilde{h}_{\ell,j}\tilde{h}_{i-\ell,j+\ell} + \sum_{\ell=i}^n \tilde{h}_{\ell,j}\tilde{h}_{i-\ell+n,j+\ell}\right) .
    \end{align*}
Hence, if $\tilde{z}$ is an element of the core, and none of the monomials in $x$ in the $h$-basis are multiples of $x$, then $d(\tilde{z})$ has the same property. That is, the differential on the core is strictly compatible with the $x$-adic filtration, i.e., $\Lambda^{\bullet}_n$ is core-homogeneous.

We have shown that all the hypotheses of Theorem \ref{main local inv cycles thm} are satisfied. The conclusion of the theorem then gives us that $H^m(FSC^{\bullet}_{n,0})$ has the same vector space dimension as $H^m(FSC^{\bullet}_{n,1})$, for each $m$. Now we use Proposition \ref{critical complex is a subset of first-subscript complex}: the critical complex $cc^{\bullet}_n$ is a sub-DGA of the first-subscript complex $FSC^{\bullet}_n$. Filtering both DGAs by first subscript, we get a map of spectral sequences
\begin{equation}
\label{ss map 37498}\xymatrix{
 H^{*,*,*}(cc^{\bullet}_{n,0}) \ar[d]\ar@{=>}[r] & H^{*,*}(cc^{\bullet}_{n,1}) \ar[d]\\
 H^{*,*,*}(FSC^{\bullet}_{n,0}) \ar@{=>}[r] & H^{*,*}(FSC^{\bullet}_{n,1}) 
}
\end{equation}
whose domain is the spectral sequence \eqref{ss 37494}. The map of spectral sequences \eqref{ss map 37498} is a split inclusion, since the DGA $cc^{\bullet}_n$ is precisely the summand of $FSC^{\bullet}_n$ consisting of elements in internal degrees divisible by $2(p^n-1)$. Theorem \ref{main local inv cycles thm} has told us that $H^{*}(FSC^{\bullet}_{n,0})$ and $H^{*}(FSC^{\bullet}_{n,1})$ are isomorphic as graded vector spaces, hence have the same (finite) dimension in each cohomological degree. Consequently there can be no nonzero differentials in the codomain of the map of spectral sequences \eqref{ss map 37498}. Since \eqref{ss 37494} is a summand in that codomain spectral sequence, there also cannot be nonzero differentials in spectral sequence \eqref{ss 37494}, as claimed.   
\end{proof}

Consequently we have the desired calculation of the cohomology of Ravenel's Lie model for the Morava stabilizer group, in internal degrees divisible by $2(p^n-1)$:
\begin{corollary}\label{main cor 1409}
Let $p>2n^2$. Then the subring of $H^*(L(n,n);\mathbb{F}_p)$ consisting of elements in internal degrees divisible by $2(p^n-1)$ is isomorphic, as a graded ring, to the associated graded of some finite filtration on $H^*(U(n);\mathbb{F}_p)$. In particular, the $\mathbb{F}_p$-linear subspace of $H^*(L(n,n);\mathbb{F}_p)$ consisting of elements in internal degrees divisible by $2(p^n-1)$ is isomorphic, as a a graded $\mathbb{F}_p$-vector space, to $H^*(U(n);\mathbb{F}_p)$.
\end{corollary}
\begin{proof}
Consequence of Theorem \ref{ss collapse thm 304}, using Proposition \ref{gln cohomology} together with Theorem \ref{gln and cc cohomology} to identify the cohomology of $cc^{\bullet}(\mathfrak{gl}_n(\mathbb{F}_p))$ with the cohomology of $U(n)$.  
\end{proof}
We suspect, but have not seriously tried to prove (except for the low heights $n\leq 4$), that there are no multiplicative filtration jumps in the abutment of the spectral sequence \eqref{ss 37494} in internal degrees divisible by $2(p^n-1)$, so that in fact $H^*(U(n);k)$ is isomorphic as a graded $k$-algebra to the graded ring of elements of $H^*(L(n,n);k)$ in internal degrees divisible by $2(p^n-1)$.

\begin{corollary}\label{main cor 1} For each positive integer $n$ and each prime $p>2n^2$, we have an isomorphism of graded $\mathbb{F}_p$-vector spaces
\begin{align*}
H^*(cc^{\bullet}_{n,0}) &\cong H^*(cc^{\bullet}_{n,1}).
\end{align*}
\end{corollary}

Finally, we have the main corollary:
\begin{corollary}\label{main cor 0239}
Conditionally upon the main result of the preprint \cite{salch2023ravenels}: for all heights $n$ and for $p>>n$, the cohomology $H^*(\extendedGn;\mathbb{F}_{p^n})$ of the height $n$ extended Morava stabilizer group, with $\extendedGn$ acting on the coefficients $\mathbb{F}_{p^n}$ as described in \cref{What this paper is about}, is isomorphic as a graded $\mathbb{F}_p$-algebra to the associated graded of a finite filtration on the cohomology $H^*(U(n);\mathbb{F}_p)$ of the unitary group. In particular, $H^*(\extendedGn;(\mathbb{F}_{p^n})_{triv})$ is isomorphic to $H^*(U(n);\mathbb{F}_p)$ as a graded $\mathbb{F}_p$-vector space.
\end{corollary}

We note that the arguments given in the proof of Theorem \ref{ss collapse thm 304} and its corollaries yield a calculation not only of the cohomology of the critical complex $cc^{\bullet}(L(n,n))$ for appropriately large primes, but also a calculation of the {\em first-subscript complex} $FSC^{\bullet}(L(n,n))$ for all $n$ and for appropriately large primes. For primes $p>n+1$, the cohomology of the Chevalley-Eilenberg DGA $CE^{\bullet}(L(n,n))$ is the input for the Ravenel-May spectral sequence which converges to the cohomology $H^*(\strict\Aut(\mathbb{G}_{1/n});\mathbb{F}_p)$ of the height $n$ strict Morava stabilizer group scheme. By the main theorem in the preprint \cite{salch2023ravenels}, the Ravenel-May spectral sequence collapses for $p>>n$. Hence a calculation of the cohomology of a sub-DGA of $CE^{\bullet}(L(n,n))$ is a calculation of a portion of $H^*(\strict\Aut(\mathbb{G}_{1/n});\mathbb{F}_p)$, for large primes; the larger the sub-DGA, the more of the cohomology of $H^*(\strict\Aut(\mathbb{G}_{1/n});\mathbb{F}_p)$ has been computed. 

If $p>2n^2$, then by Proposition \ref{critical complex is a subset of first-subscript complex}, the critical complex is a sub-DGA of the first-subscript complex. In fact the first-subscript complex becomes quite a bit larger than the critical complex, as $n$ grows. We include some $\mathbb{F}_p$-linear dimension counts for $p>>n$, to demonstrate their relative sizes:
\begin{equation}\label{dims table 1}
\begin{array}{llll}
\mbox{n}          & \dim_{\mathbb{F}_p}cc^{\bullet}(L(n,n)) & \dim_{\mathbb{F}_p}FSC^{\bullet}(L(n,n)) & \dim_{\mathbb{F}_p}CE^{\bullet}(L(n,n)) = 2^{(n^2)}\\
1 & 2 & 2 & 2 \\
2 & 8      
      & 8 & 16 \\ 
3 & 80     
      & 176  
          & 512 \\
4 & 2,432   
      & 16,384    
          & 65,536 \\
5 & 247,552 
      & 6,710,912 
          & 33,554,432 
\end{array}\end{equation}
The $\mathbb{F}_p$-linear dimension counts in table \eqref{dims table 1} are each divisible by $2^n$, since both the critical complex and the first-subscript complex are differential graded $\Lambda(h_{n,0},\dots ,h_{n,n-1})$-algebras, free as modules over $\Lambda(h_{n,0},\dots ,h_{n,n-1})$. Dividing by the factor of $2^n$ in each entry, we get the following counts\footnote{It has not escaped our attention that the $\mathbb{F}_p$-linear dimension of $cc^{\bullet}(L(4,4))$ divided by $2^4$ is $152$, which is precisely equal to the $\mathbb{F}_p$-linear dimension of $H^*(L(3,3);\mathbb{F}_p)$, calculated by Ravenel \cite{ravenel1977cohomology}, \cite[section 6.3]{MR860042}. As far as we know, this is simply a curious coincidence.}:
\begin{equation}\label{dims table 2}
\begin{array}{llll}
\mbox{n}          & \dim_{\mathbb{F}_p}cc^{\bullet}(L(n,n))/2^n & \dim_{\mathbb{F}_p}FSC^{\bullet}(L(n,n))/2^n & \dim_{\mathbb{F}_p}CE^{\bullet}(L(n,n))/2^n\\
1 & 1 & 1 & 1 \\
2 & 2 & 2 & 4 \\ 
3 & 10 & 22 & 64 \\
4 & 152 & 1024 & 4096 \\
5 & 7,736 & 209,716 & 1,048,576 
\end{array}\end{equation}
The dimension of $\dim_{\mathbb{F}_p}cc^{\bullet}(L(n,n))$ is the number of labelled Eulerian directed graphs with $n$ nodes (e.g. see \cite{MR1099250} or \cite[A007080]{oeis}), while its quotient by $2^n$ is the number of unlabelled Eulerian directed graphs with $n$ nodes \cite[A229865]{oeis}. We know no comparable description of $\dim_{\mathbb{F}_p}FSC^{\bullet}(L(n,n))$ in terms of an already-studied combinatorial sequence. 

The $\mathbb{F}_p$-linear dimension of the first-subscript complex is\footnote{This claim is true for low $n$ by inspection of the table \eqref{dims table 1}. Here is an argument for a version of this claim for large $n$ as well. Consider a multiset $M_n$ consisting of the integers $1,2,\dots ,n$, with each integer appearing with multiplicity $n$, so the multiset has cardinality $n^2$. It is not difficult to see that the sub-multisets of $M_n$ comprise an $\mathbb{F}_p$-linear basis for $CE^{\bullet}(L(n,n))$, since the latter has the $\mathbb{F}_p$-linear basis consisting of monomials in the elements $h_{i,j}$ with $i,j\in \{1,\dots ,n\}$, and writing $i$ in place of $h_{i,j}$, one has a bijection between such monomials and sub-multisets of $M_n$. 

Under this bijection, an $\mathbb{F}_p$-linear basis for $FSC^{\bullet}_n$ is given by those sub-multisets whose members sum to a multiple of $n$. Obviously there are $2^{(n^2)}$ sub-multisets of $M_n$, and for each one, its members sum to some residue modulo $n$. If those sub-multiset sums are asymptotically (as $n\rightarrow \infty$) uniformly distributed among the residue classes modulo $n$, then one would have the asymptotic result $\lim_{n\rightarrow\infty} \frac{\dim_{\mathbb{F}_p}CE^{\bullet}(L(n,n))}{\dim_{\mathbb{F}_p}FSC^{\bullet}_n} = 1/n$. These sub-multiset sums are indeed uniformly distributed, by a combinatorial argument due to D. Frohardt.} approximately $1/n$ times the $\mathbb{F}_p$-linear dimension of the Chevalley-Eilenberg complex $L(n,n)$. As one sees from \eqref{dims table 1} and \eqref{dims table 2}, this is {\em much} larger than the critical complex, i.e., much larger than the part of the Chevalley-Eilenberg complex in internal degrees divisible by $2(p^n-1)$. 
Hence the arguments given in the proof of Theorem \ref{ss collapse thm 304} and its corollaries tell us that {\em approximately $1/n$ of the mod $p$ cohomology of the height $n$} strict {\em Morava stabilizer group, at large primes, is isomorphic to the exterior algebra $E(x_1, \dots ,x_{2n-1})$.} 

If $p > \frac{n^2+1}{2}$ and the $p$-primary Smith-Toda complex $V(n-1)$ exists, then for degree reasons, there can be no nonzero differentials in the descent spectral sequence \cite{MR2030586}
\begin{align*} 
 E_2^{s,t} \cong H^s\left(\extendedGn; E(\mathbb{G}_{1/n}\otimes_{\mathbb{F}_p} \mathbb{F}_{p^n})_t(V(n-1))\right) &\Rightarrow \pi_{t-s}L_{K(n)}V(n-1) \\
 d_r: E_r^{s,t} &\rightarrow E_r^{s+r,t+r-1} .\end{align*}
Consequently, for such primes, the argument in the previous paragraph establishes that ``$1/n$ of the $K(n)$-local homotopy groups of $V(n-1)$ are given by $\Lambda_{\mathbb{F}_p}(x_1, \dots ,x_{2n-1})\otimes_{\mathbb{F}_p} K(n)_*$.'' The same statements hold with Johnson-Wilson $E(n)$ in place of Morava $K(n)$, since $L_{K(n)}V(n-1)\simeq L_{E(n)}V(n-1)$.

\subsection{Explicit computed examples at heights $\leq 3$.}

Now for some examples. Height $n=3$ is the illuminating case, but for completeness we begin with the cases $n=1,2$, where there is very little to say.
\begin{example} {\bf Height 1.} The DGA $\Lambda^{\bullet}_1$ is the exterior $k[x]$-algebra on a single class $h_{10}$ in cohomological degree $1$, with trivial boundary, with internal degree $2(p-1)$, and with connection 
\begin{align*} 
 \nabla: \Lambda^{\bullet}_1\otimes_{k[x]} k[x^{\pm 1}] 
  &\rightarrow \Lambda^{\bullet}_1\otimes_{k[x]} k[x^{\pm 1}] \otimes_{k[x^{\pm 1}]} \Omega^1_{k[x^{\pm 1}]/k}\end{align*}
given by $\nabla(h_{10}) = -h_{10}\otimes \frac{dx}{x}$. That is, $\Lambda^{1}_1$ is a Kummer module with Kummer parameter $-1$. 
We have equalities and an isomorphism
\[ FSC^{\bullet}_1 = cc^{\bullet}_1 = \Lambda^{\bullet}_1 \cong \core(\Lambda^{\bullet}_1),\]
each of these DGAs has trivial differential and hence is isomorphic to its own cohomology ring, and the Picard-Lefschetz operator $T$ acts trivially on each.
The fiber at every point $\epsilon\in k$ is the exterior $k$-algebra on $h_{10}$.
\end{example}

\begin{example} {\bf Height 2.} The DGA $\Lambda^{\bullet}_2$ is the exterior $k[x]$-algebra on classes $h_{10},h_{11},h_{20},$ and $h_{21}$ in cohomological degree $1$, with boundary 
\begin{align*}
 d(h_{i,j}) &= \sum_{\ell =1}^{i-1} h_{\ell,j}h_{i-\ell,j+\ell} + x\sum_{\ell=i}^n h_{\ell,j}h_{i-\ell+n,j+\ell}, \end{align*}
with internal degrees given by $\left| h_{i,j} \right| = 2(p^i-1)p^j$, and with connection 
\begin{align*} 
 \nabla: \Lambda^{\bullet}_2\otimes_{k[x]} k[x^{\pm 1}] 
  &\rightarrow \Lambda^{\bullet}_2\otimes_{k[x]} k[x^{\pm 1}] \otimes_{k[x^{\pm 1}]} \Omega^1_{k[x^{\pm 1}]/k}\end{align*}
given by $\nabla(h_{i,j}) = -\frac{i}{2} h_{i,j}\otimes \frac{dx}{x}$. That is, $\Lambda^{1}_2$ is a Kummer-type module with Kummer parameters $\frac{-1}{2}, \frac{-1}{2}, -1$, and $-1$.
The critical complex $cc^{\bullet}_2$ is the free $k[x]$-submodule of $\Lambda^{\bullet}_2$ spanned by all products of $h_{20},h_{21},$ and $h_{10}h_{11}$, i.e., $cc^{\bullet}_2 \cong \Lambda_{k[x]}(h_{20},h_{21},h_{10}h_{11})$. The first-subscript complex $FSC^{\bullet}_2$ is equal to $cc^{\bullet}_2$. The Picard-Lefschetz operator $T$ acts by $T(h_{1,j}) = -h_{1,j}$ and $T(h_{2,j}) = h_{2,j}$.

The cohomology of the singular fiber $H^*(\Lambda^{\bullet}_{2,0})$ is the cohomology of the height 2 Morava stabilizer group at large primes, i.e. \cite{MR0458423} an exterior algebra on generators $h_{10},h_{11},\zeta_2$ in cohomological degree $1$, tensored with a polynomial algebra on generators $g_0 = (h_{20}-h_{21})h_{10}$ and $g_1 = (h_{20} - h_{21})h_{11}$ in cohomological degree $2$, modulo the relations $h_{10}g_1 = - h_{11}g_0$ and $h_{10}g_0 = 0 = h_{11}g_1$ and $g_1^2 = g_0g_1 = g_0^2 = 0$. 

Meanwhile, the results of \cref{Applications to the cohomology...} and of \cref{Parallel Kummer analytic at 0...} established that the cohomology of any smooth fiber of $\Lambda^{\bullet}_{2}$ is isomorphic to the cohomology of $U(2)$, i.e., an exterior algebra on generators $x_1,x_3$ in cohomological degrees $1$ and $3$. The fixed-points of the Picard-Lefschetz operator on $H^*(\Lambda^{\bullet}_{2,0})$ is precisely the subring generated by $\zeta_2$ and $h_{10}g_1 = -h_{11}g_0$, i.e., $x_1$ and $x_3$. 
\end{example}

What really is the difficult part of what we have done in \cref{Monodromy section}? It was establishing that {\em the cohomology of the singular fiber of the critical complex is isomorphic to the cohomology of the smooth fiber of the critical complex.} This was a consequence of the derived local invariant cycles theorem (Theorem \ref{main local inv cycles thm}) together with the comparison of the critical complex and the first-subscript complex in the proof of Theorem \ref{ss collapse thm 304}. But, for $n=1$ and $n=2$, all this work is for nothing: at such low heights, one can check that the singular fiber of the critical complex is simply {\em equal} to the smooth fiber at $\epsilon=1$ of the critical complex. 

One does not really have to put in work until the height $n$ is at least $3$, since at $n=3$ the smooth fiber and singular fiber of the critical complex begin to diverge, and it is a nontrivial task to establish (as we have done) that the two still have isomorphic cohomology, despite no longer being equal to one another. Here is a treatment of the structure of the critical complex at height $3$:
\begin{example}\label{height 3 semilinear example}
Let $p$ be a prime, $p>3$, and let $\omega$ be a primitive $(p^2+p+1)$st root of unity in $\mathbb{F}_{p^n}$.
Consider the differential graded $\mathbb{F}_p[x]$-algebra $\Lambda_3^{\bullet}$. The set of $(p^2+p+1)$th roots of unity in $\mathbb{F}_{p^n}$ is in bijection with the set of Kummer-type connections on $\Lambda_3^{\bullet}\otimes_{\mathbb{F}_p} \mathbb{F}_p(\omega)[x^{\pm 1}]$ whose parallel transport isomorphisms are $h$-diagonal and $\tilde{\sigma}$-equivariant. For the connection corresponding to the primitive $(p^2+p+1)$st root $\omega$, we have
\begin{align}
\label{conn 11} \nabla(h_{i,j}) &= h_{i,j} \otimes -\frac{p^{i+j}-p^j}{p^3-1} \frac{dx}{x}
\end{align}
for all $i,j\in Z_3$. Write $A$ for the differential graded algebra $\Lambda_3^{\bullet}\otimes_{\mathbb{F}_p}k[x^{\pm 1}]$ with the connection $\nabla$, where $k$ is the algebraic closure of $\mathbb{F}_{p}$.

Using Theorem \ref{main thm 2}: for each nonzero $\epsilon\in k$, the Picard-Lefschetz operator $T$ on $A_{\epsilon}$ sends $h_{i,j}$ to $\omega^{-p^j - \dots - p^{i+j-1}}h_{i,j}$. The fixed-point DGA $A^T$ is the critical complex. The critical complex is generated, as an $\mathbb{F}_{p^n}[x]$-algebra, by $h_{3j}$ and $\kappa_{i}:= h_{1,i}h_{2,i+1}$ for all $i\in Z_3$, and by $L_1 := h_{11}h_{12}h_{13}$ and $L_2 := h_{21}h_{22}h_{23}$. Hence we have the presentation for $cc^{\bullet}_3$ as a $\tilde{\sigma}$-equivariant differential graded $\mathbb{F}_{p^n}[x]$-algebra:
\begin{align} 
\nonumber cc^{\bullet}_3 &= A^T \\
 &= \Lambda_{\mathbb{F}_{p^n}[x]}\left( h_{31},h_{32},h_{33},L_1,L_2 \right)\otimes_{\mathbb{F}_{p^n}} \mathbb{F}_{p^n}\left[\kappa_1,\kappa_2,\kappa_3\right] \\
\nonumber & \ \ \ \ \ \ \ \ \mbox{modulo\ relations:} \left( \kappa_{i}^2,\ \kappa_1\kappa_2\kappa_3 + L_1L_2,\ L_i\kappa_j\right) ,
\nonumber \\ d(h_{3i}) &= \kappa_i - \kappa_{i-1} ,\\
\label{d kappa} d(\kappa_{i}) &= -L_1 - xL_2  ,\\
\label{d L1} d(L_1) &= x\left(\kappa_1\kappa_2 + \kappa_2\kappa_3 + \kappa_3\kappa_1\right) ,\\
\label{d L2} d(L_2) &= -\left(\kappa_1\kappa_2 + \kappa_2\kappa_3 + \kappa_3\kappa_1\right),\\
\nonumber \sigma(h_{3,i}) &= h_{3,i+1} ,\\
\nonumber \sigma(\kappa_{i}) &= \kappa_{i+1},\\
\nonumber \sigma(L_1) &= L_1,\\
\nonumber \sigma(L_2) &= L_2.
\end{align}
Equation \ref{d kappa} shows that the differential in a smooth fiber ($x=1$) of the critical complex is indeed different from the differential in the singular fiber ($x=0$) of the critical complex\footnote{The difference in the critical complex between the singular fiber and the smooth fiber is a minor one at height $n=3$, and with a bit of calculation, it is not difficult to write down an explicit isomorphism of DGAs (not just a quasi-isomorphism!) between the smooth fiber and the singular fiber at that height. However, the situation gets worse as $n$ increases: writing down an explicit presentation for the critical complex at height $n=6$, and a quasi-isomorphism between its singular fiber and smooth fiber, is a truly miserable calculation. The purpose of all the techniques involving connections, parallel transport, and the derived local invariant cycles theorem---the whole point of \cref{Monodromy section} of this paper---is that it is an elegant and conceptual way to get isomorphisms in cohomology {\em for all heights $n$.}}.

Consider what happens when we project the elements $h_{3,j}, \kappa_j$, $L_1$, and $L_2$ into the core, in the sense of Definition-Proposition \ref{def of core}, using this connection. We will write $\tilde{x}$ for the projection of $x$ into the core, as defined in Definition-Proposition \ref{def of core}. We then have 
\begin{align*}
 \tilde{L}_1 
  &= x^{-(p^4-p)/(p^3-1)}L_1 \\
  &= x^{-p}L_1,\\
 \tilde{L}_2 
  &= x^{-(p^4-p+p^5-p^2)/(p^3-1)}L_2 \\
  &= x^{-p^2-p}L_2,\mbox{\ and\ by\ a\ similar\ calculation,}\\
 \tilde{h}_{3,j}
  &= x^{-p^j} h_{3,j},\mbox{\ and}\\
 \tilde{\kappa}_{j}
  &= x^{-p^j} \kappa_{j}.
\end{align*}
Consequently, a presentation for the core of $A^T$ is:
\begin{align} 
\nonumber \core(A^T) &= \Lambda_{k[x]}\left( \tilde{h}_{3i},\tilde{L}_1,\tilde{L}_2, : i\in Z_3\right)\otimes_{k} k\left[\tilde{\kappa}_{i}: i\in Z_3\right] \\
\nonumber \mbox{modulo\ relations:} & \left( \tilde{\kappa}_{i}^2,\tilde{\kappa}_1\tilde{\kappa}_2\tilde{\kappa}_3 + \tilde{L}_1\tilde{L}_2\right),\\
\label{d eq2 1} d(\tilde{h}_{3i}) &= \tilde{\kappa}_i - x^{p^{i-1}(1-p)}\tilde{\kappa}_{i-1} ,\\
\label{d eq2 2} d(\tilde{\kappa}_{i}) &= -x^{p(1-p^{i-1})}\tilde{L}_1 - x^{1+p^2-p^i}\tilde{L}_2  ,\\
\label{d eq2 3} d(\tilde{L}_1) &= x^{1+p^2}\tilde{\kappa}_1\tilde{\kappa}_2 + x^{1-p+p^2+p^3}\tilde{\kappa}_2\tilde{\kappa}_3 + x^{1+p^3}\tilde{\kappa}_3\tilde{\kappa}_1 ,\\
\label{d eq2 4} d(\tilde{L}_2) &= -\tilde{\kappa}_1\tilde{\kappa}_2 - x^{p^3-p}\tilde{\kappa}_2\tilde{\kappa}_3 - x^{p^3-p^2}\tilde{\kappa}_3\tilde{\kappa}_1.
\end{align}
From inspection of \eqref{d eq2 1} through \eqref{d eq2 4}, it is clear that $d$ does {\em not} preserve the $x$-adic filtration on the core, i.e., the DGA $\Lambda^{\bullet}_3\otimes_{\mathbb{F}_p} k$ is not core-homogeneous {\em when equipped with the connection } \eqref{conn 11}. Hence Theorem \ref{main local inv cycles thm} does not apply to $\Lambda^{\bullet}_3\otimes_{\mathbb{F}_p} k$ {\em equipped with the connection} \eqref{conn 11}.
\end{example}

Example \ref{height 3 semilinear example} demonstrates that, while the connection \eqref{conn 11} has good properties---its Picard-Lefschetz fixed-points are precisely the critical complex, and by Proposition \ref{h-diag isos}, it is essentially the unique connection on $\Lambda^{\bullet}_n\otimes_{\mathbb{F}_p} k$ whose parallel transport isomorphisms are $\tilde{\sigma}$-equivariant and $h$-diagonal---it has a fatal flaw: it is not core-homogeneous for $n\geq 3$. This is why, in our proof of Theorem \ref{ss collapse thm 304}, we used the {\em other} connection on $\Lambda^{\bullet}_n\otimes_{\mathbb{F}_p} k$ considered in Proposition \ref{h-diag isos}, that is, the connection whose parallel transport isomorphisms are $\sigma$-equivariant and $h$-diagonal. Taking the Picard-Lefschetz fixed-points of that connection yields the first-subscript complex, rather than the critical complex, which necessitated a somewhat delicate spectral sequence argument to compare $FSC^{\bullet}_n$ with $cc^{\bullet}_n$ in the proof of Theorem \ref{ss collapse thm 304}. This was necessary, because it is this latter connection which made $\Lambda^{\bullet}_n\otimes_{\mathbb{F}_p} k$ core-homogeneous, so that the derived local invariant cycles theorem, Theorem \ref{main local inv cycles thm}, could be used.

We now give an explicit description of the structure of $\Lambda^{\bullet}_3$ and its core, using the latter connection:

\begin{example}\label{height 3 linear example}
Let $p$ be a prime, $p>3$, and let $\omega$ be a primitive cube root of unity in some algebraic closure $\overline{\mathbb{F}}_p$ of $\mathbb{F}_p$.
Equip the differential graded $\overline{\mathbb{F}}_p[x]$-algebra $\Lambda_3^{\bullet}$ with the differential graded multiplicative connection $\nabla$ constructed in Theorem \ref{main thm 1}. We have
\begin{align*}
 \nabla(h_{i,j}) &= -\frac{i}{3} h_{i,j} \otimes \frac{dx}{x}
\end{align*}
for all $i,j\in Z_n$. Write $A$ for the differential graded algebra $\Lambda_3^{\bullet}[x^{\pm 1}]$ with the connection $\nabla$.

For each nonzero $\epsilon\in k$, the monodromy operator $T$ on $A_{\epsilon}$ sends $h_{i,j}$ to $\omega^{-i}h_{i,j}$. The fixed-point DGA $A^T$ is the first-subscript complex, by Theorem \ref{main thm 1}. The first-subscript complex is generated, as an $\overline{\mathbb{F}}_p[x]$-algebra, by $h_{3j}$ and $k_{i,j}:= h_{1i}h_{2j}$ for all $i,j\in Z_3$, and by $L_1 := h_{11}h_{12}h_{13}$ and $L_2 := h_{21}h_{22}h_{23}$. 

The Kummer parameter of $h_{3j}$ is $1$, so the parallel transport isomorphism of the fiber $A_1^T$ at $1$ with the fiber $(\core(A^T))_0$ at $0$ sends $h_{3j}\in A_1^T$ to $x^1h_{3j}\in \Lambda_3^{\bullet}$ in the core. Similarly, the Kummer parameters of $h_{1i}$ and $h_{2j}$ sum to $\frac{1}{3} + \frac{2}{3} = 1$, so parallel transport sends $k_{i,j}\in A_1^T$ to $x^1k_{i,j}\in \core(A^T)\subseteq \Lambda_3^{\bullet}$. For the same reason, $L_1\in A_1^T$ is sent by parallel transport to $xL_1\in \Lambda_3^{\bullet}$ in the core. Finally, the sum of the Kummer parameters of $L_2$ is $\frac{2}{3} + \frac{2}{3} + \frac{2}{3} = 2$, so parallel transport sends $L_2\in A_1^T$ to $x^2L_2\in \Lambda_3^{\bullet}$ in the core. By construction, the fiber of the core at zero is isomorphic via parallel transport to $A_1^T$, so the core of $A^T$ must be the differential graded $\overline{\mathbb{F}}_p$-subalgebra of $\Lambda_3^{\bullet}$ generated by $xh_{3j}$ and $xk_{i,j}$ for all $i,j\in Z_3$, and by $xL_1$ and $x^2L_2$. 

A presentation for $A^T$, as a $\sigma$-equivariant differential graded $\overline{\mathbb{F}}_p[x]$-algebra, is as follows:
\begin{align*} 
 A^T &= \Lambda_{\overline{\mathbb{F}}_p[x]}\left( h_{31},h_{32},h_{33},L_1,L_2\right)\otimes_{\overline{\mathbb{F}}_p} \overline{\mathbb{F}}_p\left[k_{i,j}: i,j\in Z_3\right] \\
 \mbox{modulo\ relations:} & \left( k_{i,j}k_{i^{\prime},j^{\prime}} + k_{i^{\prime},j}k_{i,j^{\prime}}, k_{i,j}k_{i,j^{\prime}}, k_{i^{\prime},j}k_{i,j}, k_{1,1}k_{2,2}k_{3,3} + L_1L_2\right) ,\\
 d(h_{3i}) &= k_{i,i+1} - k_{i-1,i} ,\\
 d(k_{i,i}) &= x(2h_{3,i} - h_{3,i+1} - h_{3,i+2})k_{1,i} ,\\
 d(k_{i,i+1}) &= -L_1 - xL_2  ,\\
 d(k_{i,i+2}) &= -x(2h_{3,i+1} - h_{3,i+2} - h_{3,i})k_{i,i+2} ,\\
 d(L_1) &= x\left(k_{12}k_{23} + k_{23}k_{31} + k_{31}k_{12}\right) ,\\
 d(L_2) &= -\left(k_{12}k_{23} + k_{23}k_{31} + k_{31}k_{12}\right),\\
 \sigma(h_{3,i}) &= h_{3,i+1} ,\\
 \sigma(k_{i,j}) &= k_{i+1,j+1},\\
 \sigma(L_1) &= L_1,\\
 \sigma(L_2) &= L_2.
\end{align*}

Projection into the core, using this connection, is given
\begin{align*}
\tilde{h}_{3j} &= x^{-1} h_{3j},\\ 
\tilde{k}_{j} &= x^{-1} k_j,\\
\tilde{L}_1 &= x^{-1} L_1,\mbox{\ and} \\
\tilde{L}_2 &= x^{-2} L_2.\end{align*}
Consequently, a presentation for the core of $A^T$ is:
\begin{align*} 
 \core(A^T) &= \Lambda_{\overline{\mathbb{F}}_p[x]}\left( \tilde{h}_{31},\tilde{h}_{32},\tilde{h}_{33},\tilde{L}_1,\tilde{L}_2\right)\otimes_{\overline{\mathbb{F}}_p} \overline{\mathbb{F}}_p\left[\tilde{k}_{i,j}: i,j\in Z_3\right] \\
 \mbox{modulo\ relations:} & \left( \tilde{k}_{i,j}\tilde{k}_{i^{\prime},j^{\prime}} + \tilde{k}_{i^{\prime},j}\tilde{k}_{i,j^{\prime}}, \tilde{k}_{i,j}\tilde{k}_{i,j^{\prime}}, \tilde{k}_{i^{\prime},j}\tilde{k}_{i,j}, \tilde{k}_{1,1}\tilde{k}_{2,2}\tilde{k}_{3,3} + \tilde{L}_1\tilde{L}_2\right) ,\\
 d(\tilde{h}_{3i}) &= \tilde{k}_{i,i+1} - \tilde{k}_{i-1,i} ,\\
 d(\tilde{k}_{i,i}) &= (2\tilde{h}_{3,i} - \tilde{h}_{3,i+1} - \tilde{h}_{3,i+2})\tilde{k}_{1,i} ,\\
 d(\tilde{k}_{i,i+1}) &= -\tilde{L}_1 - \tilde{L}_2  ,\\
 d(\tilde{k}_{i,i+2}) &= -(2\tilde{h}_{3,i+1} - \tilde{h}_{3,i+2} - \tilde{h}_{3,i})\tilde{k}_{i,i+2} ,\\
 d(\tilde{L}_1) &= \tilde{k}_{12}\tilde{k}_{23} + \tilde{k}_{23}\tilde{k}_{31} + \tilde{k}_{31}\tilde{k}_{12} ,\\
 d(\tilde{L}_2) &= -\left(\tilde{k}_{12}\tilde{k}_{23} + \tilde{k}_{23}\tilde{k}_{31} + \tilde{k}_{31}\tilde{k}_{12}\right).
\end{align*}

One sees by inspection of \eqref{d kappa}, \eqref{d L1}, and \eqref{d L2} that the differentials in $\core(A^T)$ are all homogeneous of degree zero with respect to the filtration by powers of $x$. That is, $A^T$ is ``core-homogeneous'' in the sense of Definition \ref{def of core-homogeneity}. Consequently the core monodromy spectral sequence collapses with no nonzero differentials at $E_1$. For the most part, these observations are simply special cases of the general results in Theorems \ref{main local inv cycles thm} and \ref{ss collapse thm 304}. There is a noteworthy curiosity here at height $n=3$: we did not need to make the assumption $p>2n^2 = 18$ which appears in the statement of Theorem \ref{ss collapse thm 304}. Indeed, the only need for the bound $p>2n^2$ in Theorem \ref{ss collapse thm 304} was in order to use Proposition \ref{critical complex is a subset of first-subscript complex} in order to establish that the first-subscript complex contains the critical complex. As pointed out in Remark \ref{remark on prime bound}, the bound $p>2n^2$ is not sharp, and in fact, in the present case ($n=3$), the first-subscript complex contains the critical complex at {\em all} primes.

We now consider the inclusion of $\core(A^T)$ into the critical complex $cc^{\bullet}_3$ of $\Lambda_3$. The core is a sub-DGA of $cc^{\bullet}_3$, and the inclusion $\core(A^T)\rightarrow cc^{\bullet}_3$ preserves the filtration by powers of $x$. Consequently we have the map of spectral sequences from the core monodromy spectral sequence of $A^T$ to the spectral sequence 
\begin{align}
\label{ss 40392} E_1^{s,t} \cong H^s(cc_{n,\epsilon}^{\bullet}/x)[\overline{x}] &\Rightarrow H^s(cc_{n,\epsilon}^{\bullet}) 
\end{align}
of the filtration of the critical complex by powers of $x$.

At a glance, one expects that \eqref{ss 40392} could very well have nonzero differentials, because the formula \eqref{d L1} for the differential on $L_1$ in $cc_{n,\epsilon}^{\bullet}$ raises the $x$-adic valuation by $1$. 
However, the target of that differential would be $\kappa_1\kappa_2 + \kappa_2\kappa_3 + \kappa_3\kappa_1$, which does not survive to the $E_1$-page, since by \eqref{d L2}, $\kappa_1\kappa_2 + \kappa_2\kappa_3 + \kappa_3\kappa_1$ is already a coboundary in $cc_{n,\epsilon}^{\bullet}/x$. (One might say that $\kappa_1\kappa_2 + \kappa_2\kappa_3 + \kappa_3\kappa_1$ ``was already hit by a $d_0$-differential,'' since it is consistent with the usual conventions of spectral sequences of filtered cochain complexes to refer to the differential in $cc_{n,\epsilon}^{\bullet}/x$ as ``the $d_0$-differential.'') So one seems to have gotten lucky.

Happily, Theorem \ref{ss collapse thm 304} tells us that it is not mere luck, but in fact this apparently-nontrivial collapse of the spectral sequence happens at all heights $n$. 
\end{example}

\appendix
\section{Review and background on Morava stabilizer groups}
\label{Review of Morava stabilizer groups}
There are many things called Morava stabilizer groups---strict, full, extended, and their group scheme variants---and their cohomology, with coefficients taken in various coefficient modules, serve as input for various spectral sequences used to calculate stable homotopy groups. Even for readers who are experts in one perspective on the cohomology of Morava stabilizer groups, it can be nontrivial to make sense of statements formulated from some other perspective, e.g. to compare a statement about $\Cotor_{BP_*BP}^{*,*}\left(BP_*, v_2^{-1} BP_*/(p,v_1)\right)$ from chapters 5 and 6 of \cite{MR860042} to the cohomology of the height $2$ extended Morava stabilizer group with appropriate coefficients. We provide this appendix to review:
\begin{itemize}
\item the basics of Morava stabilizer groups
\item how their cohomology is used in computational stable homotopy theory,
\item and how the main theorem of this paper is expressed using several different, commonly used ways of talking about the cohomology of the Morava stabilizer groups.
\end{itemize}

The story told in this appendix is old, and it is well-told in several places, e.g. \cite{MR1333942}.
\begin{itemize}
\item Fix a prime $p$, and let $\mathbb{G}_{1/n}$ be the height $n$ commutative one-dimensional\footnote{From now on, all formal group laws in this paper will be assumed to be commutative and one-dimensional.} formal group law over $\mathbb{F}_{p}$ classified by the ring map $BP_*\rightarrow \mathbb{F}_p$ sending the Hazewinkel generator $v_n$ to $1$, and sending $v_i$ to zero for all $i\neq n$. Here $BP_*$ is the classifying ring of $p$-typical formal group laws, whose universal formal group law was proved by Quillen \cite{MR0253350} to be precisely the formal group law arising from the complex orientation on $p$-primary Brown-Peterson homology.
\item The automorphism group scheme $\Aut(\mathbb{G}_{1/n})$ is called the {\em full Morava stabilizer group scheme}. It is pro-\'{e}tale, and becomes constant after base change to $\mathbb{F}_{p^n}$. That is, if we write $\mathbb{G}_{1/n}\otimes_{\mathbb{F}_p}\mathbb{F}_{p^n}$ for the same formal group law as $\mathbb{G}_{1/n}$ but with its coefficients regarded as living in $\mathbb{F}_{p^n}$ rather than $\mathbb{F}_p$, then the automorphism group scheme $\Aut(\mathbb{G}_{1/n}\otimes_{\mathbb{F}_p}\mathbb{F}_{p^n})$ is simply a profinite {\em group}. It is called the {\em full Morava stabilizer group of height $n$.} This profinite group acts continuously on the Lubin-Tate space $\Def(\mathbb{G}_{1/n}\otimes_{\mathbb{F}_p}\mathbb{F}_{p^n})$ of deformations of $\mathbb{G}_{1/n}\otimes_{\mathbb{F}_p}\mathbb{F}_{p^n}$. 
\item
Since we have base-changed from $\mathbb{F}_p$ to $\mathbb{F}_{p^n}$, the Galois group $\Gal(\mathbb{F}_{p^n}/\mathbb{F}_p)$ also acts on $\Def(\mathbb{G}_{1/n}\otimes_{\mathbb{F}_p}\mathbb{F}_{p^n})$. The two group actions assemble into an action of the semidirect product $\Aut(\mathbb{G}_{1/n}\otimes_{\mathbb{F}_p}\mathbb{F}_{p^n})\rtimes \Gal(\mathbb{F}_{p^n}/\mathbb{F}_p)$ on $\Def(\mathbb{G}_{1/n}\otimes_{\mathbb{F}_p}\mathbb{F}_{p^n})$. This semidirect product is called the {\em extended Morava stabilizer group of height $n$}, and written $\extendedGn$.

In fact a little bit more is true: one can define a moduli space $\Def(\mathbb{G}_{1/n}\otimes_{\mathbb{F}_p}\mathbb{F}_{p^n})_*$ of ``deformations of $\mathbb{G}_{1/n}\otimes_{\mathbb{F}_p}\mathbb{F}_{p^n}$ equipped with a $1$-form,'' in the sense of \cite{MR1333942}. The extended Morava stabilizer group of height $n$ also acts continuously on $\Def(\mathbb{G}_{1/n}\otimes_{\mathbb{F}_p}\mathbb{F}_{p^n})_*$.
\item 
Let the extended Morava stabilizer group $\extendedGn$ act on $\mathbb{F}_{p^n}$ by letting the Galois group $\Gal(\mathbb{F}_{p^n}/\mathbb{F}_p)$ acts in its canonical way, while the full Morava stabilizer group $\Aut(\mathbb{G}_{1/n}\otimes_{\mathbb{F}_p}\mathbb{F}_{p^n})\subseteq \extendedGn$ acts trivially on the $\Gal(\mathbb{F}_{p^n}/\mathbb{F}_p)$-fixed points of $\mathbb{F}_{p^n}$. If $p\nmid n$, then the group cohomology of the extended Morava stabilizer group $H^*(\extendedGn;\mathbb{F}_{p^n})$, with this particular action of $\extendedGn$ on $\mathbb{F}_{p^n}$, agrees with the cohomology of the full Morava stabilizer group {\em scheme} $H^*(\Aut(\mathbb{G}_{1/n});(\mathbb{F}_p)_{triv})$. 

The point is that the reader who prefers the simplicity of thinking of a trivial action can think of this as a paper about $H^*(\Aut(\mathbb{G}_{1/n});(\mathbb{F}_p)_{triv})$, while the reader who dislikes group schemes and wants to think about honest {\em groups} can instead think of this as a paper about $H^*(\extendedGn;\mathbb{F}_{p^n})$. The two cohomology rings are isomorphic, both arise in the input of spectral sequences used to calculate stable homotopy groups (as explained below), and this paper calculates both cohomology rings for $p>>n$.
\item The deformation space $\Def(\mathbb{G}_{1/n}\otimes_{\mathbb{F}_p}\mathbb{F}_{p^n})_*$ is an affine formal scheme, and comes equipped with an action of the multiplicative group scheme $\mathbb{G}_m$. This is another way of saying that $\Def(\mathbb{G}_{1/n}\otimes_{\mathbb{F}_p}\mathbb{F}_{p^n})_*$ is simply the formal spectrum, $\Spf$, of some adic commutative graded ring $\Gamma(\Def(\mathbb{G}_{1/n}\otimes_{\mathbb{F}_p}\mathbb{F}_{p^n})_*)$, in the sense of \cite[sections 0.7 and I.10]{MR0163908}. This graded ring is quite simple \cite{MR0238854}: it is abstractly isomorphic to \[ W(\mathbb{F}_{p^n})[[u_1, \dots ,u_{n-1}]][u^{\pm 1}],\] with $u$ in degree $2$, and with each $u_i$ in degree $0$. Here $W(\mathbb{F}_{p^n})$ is the Witt ring of $\mathbb{F}_{p^n}$; its simplest description is that it is the ring of $p$-adic integers with a primitive $(p^n-1)$st root of unity adjoined.

Morava \cite{MR782555} constructed a complex-oriented generalized homology theory, {\em Morava $E$-theory}, with the property that its coefficient ring $E(\mathbb{G}_{1/n}\otimes_{\mathbb{F}_p}\mathbb{F}_{p^n})_*$ is isomorphic to the graded ring $\Gamma(\Def(\mathbb{G}_{1/n}\otimes_{\mathbb{F}_p}\mathbb{F}_{p^n})_*)$, and furthermore the formal group law on $E(\mathbb{G}_{1/n}\otimes_{\mathbb{F}_p}\mathbb{F}_{p^n})_*$ arising from the complex orientation coincides with the universal formal group law on $\Gamma(\Def(\mathbb{G}_{1/n}\otimes_{\mathbb{F}_p}\mathbb{F}_{p^n})_*)$. Goerss--Hopkins \cite{MR2125040} constructed an $E_{\infty}$-ring structure on the representing spectrum of Morava $E$-theory, and showed that the extended Morava stabilizer group $\extendedGn$ acts on that spectrum by $E_{\infty}$-ring automorphisms.
\item Devinatz--Hopkins \cite{MR2030586} constructed, for each finite spectrum $X$, a spectral sequence 
\begin{align}
\label{dhss} E_2^{s,t} \cong H^s\left( \extendedGn; E(\mathbb{G}_{1/n}\otimes_{\mathbb{F}_p}\mathbb{F}_{p^n})_t(X)\right) 
 &\Rightarrow \pi_{t-s}(L_{K(n)}X),
\end{align}
where $L_{K(n)}X$ is the Bousfield localization of $X$ at the Morava $K$-theory $K(n)_*$. For each $n$, there is a homotopy pullback square
\[\xymatrix{
 L_{E(n)}X \ar[r] \ar[d] & L_{K(n)}X \ar[d] \\
 L_{E(n-1)}X\ar[r] & L_{E(n-1)}L_{K(n)}X
}\]
hence a long exact sequence in homotopy groups 
\begin{equation}\label{fracture les}\xymatrix{
 \dots \ar[r] & 
  \pi_1(L_{K(n)}X)\oplus \pi_1(L_{E(n-1)}X) \ar[r] &
  \pi_1(L_{E(n-1)}L_{K(n)}X) \ar`r_l[ll] `l[dll] [dll] \\
 \pi_0(L_{E(n)}X) \ar[r] &
  \pi_0(L_{K(n)}X)\oplus \pi_0(L_{E(n-1)}X) \ar[r] &
  \pi_0(L_{E(n-1)}L_{K(n)}X) \ar`r_l[ll] `l[dll] [dll] \\
 \pi_{-1}(L_{E(n)}X) \ar[r] &
  \pi_{-1}(L_{K(n)}X)\oplus \pi_{-1}(L_{E(n-1)}X) \ar[r] &
  \dots ,}\end{equation}
where $L_{E(n)}X$ is the Bousfield localization of $X$ at the $n$th Johnson-Wilson homology theory $E(n)$. 

Since $\pi_*(L_{E(0)}X)$ is simply the rational homology $H_*(X;\mathbb{Q})$ of $X$, one can in principle carry out an induction, running the spectral sequence \eqref{dhss} for higher and higher values of $n$, and each time, working out the behavior of the long exact sequence \eqref{fracture les} in order to deduce the homotopy groups of $L_{E(n)}X$ for higher and higher values of $n$. By the chromatic convergence theorem \cite{MR1192553}, the homotopy groups of the homotopy limit $\holim_n L_{E(n)}X$ are precisely the $p$-local homotopy groups of $X$ itself. The moral is that one has a method---in principle!---for passing from the profinite group cohomology \begin{equation}\label{coh 120342}  H^*\left( \extendedGn; E(\mathbb{G}_{1/n}\otimes_{\mathbb{F}_p}\mathbb{F}_{p^n})_*(X)\right),\end{equation} for each value of $n$, to the stable homotopy groups of $X$. This is of serious interest, since calculating the stable homotopy groups of finite spectra (e.g. the sphere spectrum $S^0$) is the most fundamental concern of computational stable homotopy theory.
\item One typically calculates \eqref{coh 120342} by first calculating the cohomology of $\extendedGn$ with small coefficients, then running Bockstein spectral sequences to build up inductively to more sophisticated coefficients, eventually arriving at the $E_2$-page of \eqref{coh 120342}. The coefficient module 
\begin{equation}\label{iso n439} E(\mathbb{G}_{1/n}\otimes_{\mathbb{F}_p}\mathbb{F}_{p^n})_*/(p,u_1, \dots ,u_{n-1})\cong \mathbb{F}_{p^n}[u^{\pm 1}]\end{equation}
is the usual starting point. If the $p$-primary Smith-Toda complex $V(n-1)$ exists, then \eqref{iso n439} is its $E(\mathbb{G}_{1/n}\otimes_{\mathbb{F}_p}\mathbb{F}_{p^n})_*$-homology. Even if the $p$-primary $V(n-1)$ does not exist, the cohomology $H^*\left(\extendedGn; \mathbb{F}_{p^n}[u^{\pm 1}]\right)$ provides the input for a Bockstein spectral sequence which converges to 
$H^*\left(\extendedGn; \mathbb{F}_{p^n}[[u_{n-1}]][u^{\pm 1}]\right)$, which is in turn is the input for a Bockstein spectral sequence which converges to $H^*\left(\extendedGn; \mathbb{F}_{p^n}[[u_{n-2},u_{n-1}]][u^{\pm 1}]\right)$, and so on; after $n$ such Bockstein spectral sequences, one arrives at
\begin{align*} H^*\left(\extendedGn; E(\mathbb{G}_{1/n}\otimes_{\mathbb{F}_p}\mathbb{F}_{p^n})_*\right),\end{align*}
the input for the spectral sequence \eqref{dhss} in the most valuable case $X=S^0$, i.e., the case that yields stable homotopy groups of the $K(n)$-localization of the sphere spectrum.

Both the quotient $\Gal(\mathbb{F}_{p^n}/\mathbb{F}_p)$ of $\extendedGn$, and the quotient $\mathbb{F}_{p^n}^{\times}$ of $\Aut(\mathbb{G}_{1/n}\otimes_{\mathbb{F}_p}\mathbb{F}_{p^n})$, act on the coefficients in $\mathbb{F}_{p^n}[u^{\pm 1}]$. The latter action on $\mathbb{F}_{p^n}\{ u^j\}$, the $\mathbb{F}_{p^n}$-linear span of the monomial $u^j$, depends on the exponent $j$. If $j$ is a multiple of $p^n-1$, then the action of $\extendedGn$ on $\mathbb{F}_{p^n}\{ u^j\}$ coincides with the action of $\extendedGn$ on $\mathbb{F}_{p^n}$ described above, three bullet-points ago. Consequently, another way to describe what we calculate in this paper is that it is the bigraded subring $H^*\left(\extendedGn; \mathbb{F}_{p^n}[u^{\pm (p^n-1)}]\right) = H^*\left(\extendedGn; \mathbb{F}_{p^n}[v_n^{\pm 1}]\right)$ of $H^*\left(\extendedGn; \mathbb{F}_{p^n}[u^{\pm 1}]\right)$.

{\em A priori}, it sounds as though $H^*\left(\extendedGn; \mathbb{F}_{p^n}[u^{\pm (p^n-1)}]\right)$ ought to be a small part, namely ``$100/(p^n-1)$ percent,'' of $H^*\left(\extendedGn; \mathbb{F}_{p^n}[u^{\pm 1}]\right)$. Following Corollary \ref{main cor 0239}, we explain a precise sense in which $H^*\left(\extendedGn; \mathbb{F}_{p^n}[u^{\pm (p^n-1)}]\right)$ turns out to instead be a much larger part, namely ``$100/n$ percent,'' of $H^*\left(\extendedGn; \mathbb{F}_{p^n}[u^{\pm 1}]\right)$.
\item
The cohomology of the profinite group $\extendedGn$ with coefficients in \eqref{iso n439} is also the input for a {\em different} sequence of $n$ Bockstein spectral sequences \cite{MR0458423},\cite{MR860042}, which are roughly ``Cohen--Macaulay local duals'' of those described in the previous paragraph, and which eventually arrive at \linebreak $\Cotor_{BP_*BP}^{*,*}(BP_*,v_n^{-1}BP_*/I_n^{\infty})$, the height $n$ layer in the $E_1$-term of the chromatic spectral sequence \cite[chapter 5]{MR860042}. The chromatic spectral sequence then converges to the flat cohomology 
\begin{align*} H^*_{fl}(\mathcal{M}_{fg};\omega^{\otimes *}) &\cong \Cotor_{BP_*BP}^{*,*}(BP_*,BP_*) \end{align*}
of the moduli stack of formal groups over $\Spec \mathbb{Z}_{(p)}$, which is the $E_2$-term of the Adams-Novikov spectral sequence converging to the $p$-local stable homotopy groups of spheres. 
\item
The profinite group $\Aut(\mathbb{G}_{1/n}\otimes_{\mathbb{F}_p}\mathbb{F}_{p^n})$ has a pro-$p$-Sylow subgroup given by the {\em strict Morava stabilizer group of height $n$}, $\strict\Aut(\mathbb{G}_{1/n}\otimes_{\mathbb{F}_p}\mathbb{F}_{p^n})$, the group of strict automorphisms of $\mathbb{G}_{1/n}\otimes_{\mathbb{F}_p}\mathbb{F}_{p^n}$. 
Recall that an automorphism $\phi$ of a formal group law $\mathbb{G}$ over a ring $R$ is a power series $\phi(X)\in R[[X]]$ such that $\mathbb{G}(\phi(X),\phi(Y)) = \phi(\mathbb{G}(X,Y))$. The automorphism is said to be {\em strict} if $\phi(X) \equiv X$ modulo $X^2$. We have a short exact sequence 
\begin{equation}\label{ses fj3049} 1 \rightarrow \strict\Aut(\mathbb{G}_{1/n}\otimes_{\mathbb{F}_p}\mathbb{F}_{p^n}) \rightarrow \Aut(\mathbb{G}_{1/n}\otimes_{\mathbb{F}_p}\mathbb{F}_{p^n}) \stackrel{\pi}{\longrightarrow} \mathbb{F}_{p^n}^{\times} \rightarrow 1.\end{equation}

The module \eqref{iso n439} is graded, with $u$ in degree $2$; we call this grading the {\em internal} grading. The action of $\Aut(\mathbb{G}_{1/n}\otimes_{\mathbb{F}_p}\mathbb{F}_{p^n})$ on the degree $2i$ summand $\mathbb{F}_{p^n}\cdot u^i$ of \eqref{iso n439} is as follows: an element $\sigma \in \Aut(\mathbb{G}_{1/n}\otimes_{\mathbb{F}_p}\mathbb{F}_{p^n})$ sends $u^i$ to the product $\pi(\sigma)^{-i}\cdot u^i\in \mathbb{F}_{p^n}\cdot u^i$, where $\pi$ is the projection  map in the exact sequence \eqref{ses fj3049}.

Consequently, for each integer $i$, $\Aut(\mathbb{G}_{1/n}\otimes_{\mathbb{F}_p}\mathbb{F}_{p^n})$ acts on $\mathbb{F}_{p^n}\cdot u^i$ in the same way that it acts on $\Aut(\mathbb{G}_{1/n}\otimes_{\mathbb{F}_p}\mathbb{F}_{p^n})\cdot u^{i + (p^n-1)}$. Moreover $\Aut(\mathbb{G}_{1/n}\otimes_{\mathbb{F}_p}\mathbb{F}_{p^n})$ acts {\em trivially} on $\mathbb{F}_{p^n}\cdot u^0$. 

Write $v_n$ for the $(p^n-1)$th power of $u$. Then, considering the internal grading, the cohomology ring \eqref{coh ring 1} is a $\mathbb{Z}/2(p^n-1)\mathbb{Z}$-graded algebra. Its internal degree $0$ subring is precisely 
\begin{align}\nonumber
H^*\left(\extendedGn; \mathbb{F}_{p^n}[v_n^{\pm 1}]\right)  &\cong 
H^*\left(\extendedGn; (\mathbb{F}_{p^n})_{triv}\right)\otimes_{\mathbb{F}_p} \mathbb{F}_p[v_n^{\pm 1}] \\
\label{iso fjgh4308}  &\cong H^*\left(\Aut(\mathbb{G}_{1/n}); (\mathbb{F}_{p})_{triv}\right)\otimes_{\mathbb{F}_p} \mathbb{F}_p[v_n^{\pm 1}] ,\end{align}
where \eqref{iso fjgh4308} is to be understood as cohomology of the automorphism group {\em scheme}. 
\end{itemize}


\begin{thebibliography}{10}

\bibitem{barthel2024rationalizationknlocalsphere}
Tobias Barthel, Tomer~M. Schlank, Nathaniel Stapleton, and Jared Weinstein.
\newblock On the rationalization of the ${K}(n)$-local sphere.
\newblock arXiv 2402.00960, 2024.

\bibitem{MR1883387}
Amnon Besser.
\newblock Coleman integration using the {T}annakian formalism.
\newblock {\em Math. Ann.}, 322(1):19--48, 2002.

\bibitem{MR0024908}
Claude Chevalley and Samuel Eilenberg.
\newblock Cohomology theory of {L}ie groups and {L}ie algebras.
\newblock {\em Trans. Amer. Math. Soc.}, 63:85--124, 1948.

\bibitem{MR0444662}
C.~H. Clemens.
\newblock Degeneration of {K}\"{a}hler manifolds.
\newblock {\em Duke Math. J.}, 44(2):215--290, 1977.

\bibitem{MR0498551}
Pierre Deligne.
\newblock Th\'{e}orie de {H}odge. {II}.
\newblock {\em Inst. Hautes \'{E}tudes Sci. Publ. Math.}, (40):5--57, 1971.

\bibitem{MR0601520}
Pierre Deligne.
\newblock La conjecture de {W}eil. {II}.
\newblock {\em Inst. Hautes \'{E}tudes Sci. Publ. Math.}, (52):137--252, 1980.

\bibitem{MR1333942}
Ethan~S. Devinatz and Michael~J. Hopkins.
\newblock The action of the {M}orava stabilizer group on the {L}ubin-{T}ate
  moduli space of lifts.
\newblock {\em Amer. J. Math.}, 117(3):669--710, 1995.

\bibitem{MR2030586}
Ethan~S. Devinatz and Michael~J. Hopkins.
\newblock Homotopy fixed point spectra for closed subgroups of the {M}orava
  stabilizer groups.
\newblock {\em Topology}, 43(1):1--47, 2004.

\bibitem{MR0821318}
Eric~M. Friedlander and Brian~J. Parshall.
\newblock Cohomology of {L}ie algebras and algebraic groups.
\newblock {\em Amer. J. Math.}, 108(1):235--253, 1986.

\bibitem{MR4281382}
Bogdan Gheorghe, Guozhen Wang, and Zhouli Xu.
\newblock The special fiber of the motivic deformation of the stable homotopy
  category is algebraic.
\newblock {\em Acta Math.}, 226(2):319--407, 2021.

\bibitem{MR2125040}
P.~G. Goerss and M.~J. Hopkins.
\newblock Moduli spaces of commutative ring spectra.
\newblock In {\em Structured ring spectra}, volume 315 of {\em London Math.
  Soc. Lecture Note Ser.}, pages 151--200. Cambridge Univ. Press, Cambridge,
  2004.

\bibitem{MR0258824}
Phillip~A. Griffiths.
\newblock Periods of integrals on algebraic manifolds: {S}ummary of main
  results and discussion of open problems.
\newblock {\em Bull. Amer. Math. Soc.}, 76:228--296, 1970.

\bibitem{MR0163908}
A.~Grothendieck.
\newblock \'{E}l\'ements de g\'eom\'etrie alg\'ebrique. {I}. {L}e langage des
  sch\'emas.
\newblock {\em Inst. Hautes \'Etudes Sci. Publ. Math.}, (4):228, 1960.

\bibitem{MR4301320}
Xing Gu, Xiangjun Wang, and Jianqiu Wu.
\newblock The composition of {R}. {C}ohen's elements and the third periodic
  elements in stable homotopy groups of spheres.
\newblock {\em Osaka J. Math.}, 58(2):367--382, 2021.

\bibitem{MR2066503}
Mark Hovey.
\newblock Homotopy theory of comodules over a {H}opf algebroid.
\newblock In {\em Homotopy theory: relations with algebraic geometry, group
  cohomology, and algebraic {$K$}-theory}, volume 346 of {\em Contemp. Math.},
  pages 261--304. Amer. Math. Soc., Providence, RI, 2004.

\bibitem{MR0238854}
Jonathan Lubin and John Tate.
\newblock Formal moduli for one-parameter formal {L}ie groups.
\newblock {\em Bull. Soc. Math. France}, 94:49--59, 1966.

\bibitem{MR1099250}
Brendan~D. McKay.
\newblock The asymptotic numbers of regular tournaments, {E}ulerian digraphs
  and {E}ulerian oriented graphs.
\newblock {\em Combinatorica}, 10(4):367--377, 1990.

\bibitem{MR0458423}
Haynes~R. Miller, Douglas~C. Ravenel, and W.~Stephen Wilson.
\newblock Periodic phenomena in the {A}dams-{N}ovikov spectral sequence.
\newblock {\em Ann. of Math. (2)}, 106(3):469--516, 1977.

\bibitem{MR782555}
Jack Morava.
\newblock Noetherian localisations of categories of cobordism comodules.
\newblock {\em Ann. of Math. (2)}, 121(1):1--39, 1985.

\bibitem{MR0756848}
David~R. Morrison.
\newblock The {C}lemens-{S}chmid exact sequence and applications.
\newblock In {\em Topics in transcendental algebraic geometry ({P}rinceton,
  {N}.{J}., 1981/1982)}, volume 106 of {\em Ann. of Math. Stud.}, pages
  101--119. Princeton Univ. Press, Princeton, NJ, 1984.

\bibitem{MR4574661}
Piotr Pstr\c{a}gowski.
\newblock Synthetic spectra and the cellular motivic category.
\newblock {\em Invent. Math.}, 232(2):553--681, 2023.

\bibitem{MR0253350}
Daniel Quillen.
\newblock On the formal group laws of unoriented and complex cobordism theory.
\newblock {\em Bull. Amer. Math. Soc.}, 75:1293--1298, 1969.

\bibitem{MR0420619}
Douglas~C. Ravenel.
\newblock The structure of {M}orava stabilizer algebras.
\newblock {\em Invent. Math.}, 37(2):109--120, 1976.

\bibitem{ravenel1977cohomology}
Douglas~C Ravenel.
\newblock The cohomology of the {M}orava stabilizer algebras.
\newblock {\em Mathematische Zeitschrift}, 152(3):287--297, 1977.

\bibitem{MR860042}
Douglas~C. Ravenel.
\newblock {\em Complex cobordism and stable homotopy groups of spheres}, volume
  121 of {\em Pure and Applied Mathematics}.
\newblock Academic Press Inc., Orlando, FL, 1986.

\bibitem{MR1192553}
Douglas~C. Ravenel.
\newblock {\em Nilpotence and periodicity in stable homotopy theory}, volume
  128 of {\em Annals of Mathematics Studies}.
\newblock Princeton University Press, Princeton, NJ, 1992.
\newblock Appendix C by Jeff Smith.

\bibitem{height4}
A.~Salch.
\newblock The cohomology of the height $4$ {M}orava stabilizer group at large
  primes.
\newblock arXiv 1607.01108, 2016.

\bibitem{salch2023ravenels}
A.~Salch.
\newblock Ravenel's {M}ay spectral sequence collapses immediately at large
  primes.
\newblock arXiv 2312.17185, 2023.

\bibitem{MR0554237}
Jean-Pierre Serre.
\newblock {\em Local fields}, volume~67 of {\em Graduate Texts in Mathematics}.
\newblock Springer-Verlag, New York-Berlin, 1979.
\newblock Translated from the French by Marvin Jay Greenberg.

\bibitem{oeis}
Neil J.~A. Sloane and The OEIS~Foundation Inc.
\newblock The on-line encyclopedia of integer sequences, 2024.

\bibitem{MR0646078}
Dennis Sullivan.
\newblock Infinitesimal computations in topology.
\newblock {\em Inst. Hautes \'{E}tudes Sci. Publ. Math.}, (47):269--331, 1977.

\bibitem{MR547117}
William~C. Waterhouse.
\newblock {\em Introduction to affine group schemes}, volume~66 of {\em
  Graduate Texts in Mathematics}.
\newblock Springer-Verlag, New York, 1979.

\bibitem{MR1173994}
Atsushi Yamaguchi.
\newblock The structure of the cohomology of {M}orava stabilizer algebra
  {$S(3)$}.
\newblock {\em Osaka J. Math.}, 29(2):347--359, 1992.

\end{thebibliography}

\end{document}